\documentclass[10pt]{amsart}
\usepackage{times,amsmath,amsbsy,amssymb,amscd,mathrsfs}
\usepackage{slashbox}\usepackage{graphicx,subfigure,epstopdf,wrapfig,chemarrow}

\usepackage{algorithm2e} 
\usepackage{multicol,multirow}
\usepackage{mathtools}
\usepackage[usenames,dvipsnames,svgnames,table]{xcolor}
\usepackage[all]{xy}
\usepackage{wrapfig}
\usepackage{tcolorbox}
\usepackage{stmaryrd}

\usepackage{tikz,tikz-cd}
\usepackage[utf8]{inputenc}
\usepackage{pgfplots} 
\usepackage{pgfgantt}
\usepackage{pdflscape}
\pgfplotsset{compat=newest} 
\pgfplotsset{plot coordinates/math parser=false}
\newlength\fwidth

\definecolor{myBlue}{rgb}{0.0,0.0,0.55}
\usepackage[pdftex,colorlinks=true,citecolor=myBlue,linkcolor=myBlue]{hyperref}

\usepackage[hyperpageref]{backref}

\usepackage{comment,enumerate,multicol,xspace}

  \newcounter{mnote}
  \let\oldmarginpar\marginpar
    \renewcommand\marginpar[1]{\-\oldmarginpar[\raggedleft\footnotesize #1]%
    {\raggedright\footnotesize #1}}

\newtheorem{theorem}{Theorem}[section]
\newtheorem{lemma}[theorem]{Lemma}
\newtheorem{corollary}[theorem]{Corollary}
\newtheorem{proposition}[theorem]{Proposition}

\newtheorem{remark}[theorem]{Remark}

\newcommand{\dd}{\,{\rm d}}

\newcommand{\bs}{\boldsymbol}

\newcommand{\curl}{{\rm curl\,}}
\renewcommand{\div}{\operatorname{div}}
\newcommand{\grad}{{\rm grad\,}}
\newcommand{\tr}{\operatorname{tr}}
\newcommand{\dev}{\operatorname{dev}}
\newcommand{\sym}{\operatorname{sym}}

\newcommand{\defm}{\operatorname{def}}

\newcommand{\Oplus}{\ensuremath{\vcenter{\hbox{\scalebox{1.5}{$\oplus$}}}}}

\newcommand{\step}[1]{\noindent\raisebox{1.5pt}[10pt][0pt]{\tiny\framebox{$#1$}}\xspace}

\newcommand{\vertiii}[1]{{\left\vert\kern-0.25ex\left\vert\kern-0.25ex\left\vert #1 
    \right\vert\kern-0.25ex\right\vert\kern-0.25ex\right\vert}}

\begin{document}
\title[A New Div-Div-Conforming Finite Element Space]{A New Div-Div-Conforming Symmetric Tensor Finite Element Space with Applications to the Biharmonic Equation}

\author{Long Chen}%
 \address{Department of Mathematics, University of California at Irvine, Irvine, CA 92697, USA}%
 \email{chenlong@math.uci.edu}%
 \author{Xuehai Huang}%
 \address{Corresponding author. School of Mathematics, Shanghai University of Finance and Economics, Shanghai 200433, China}%
 \email{huang.xuehai@sufe.edu.cn}%

 \thanks{The first author was supported by NSF DMS-2012465, and DMS-2309785. The second author was supported by the National Natural Science Foundation of China Project 12171300, and the Natural Science Foundation of Shanghai 21ZR1480500.}

\makeatletter
\@namedef{subjclassname@2020}{\textup{2020} Mathematics Subject Classification}
\makeatother

\begin{abstract}
A new $H(\div\div)$-conforming finite element is presented, which avoids the need for super-smoothness by redistributing the degrees of freedom to edges and faces. This leads to a hybridizable mixed method with superconvergence for the biharmonic equation. Moreover, new finite element divdiv complexes are established. Finally, new weak Galerkin and $C^0$ discontinuous Galerkin methods for the biharmonic equation are derived.
 \end{abstract}
\maketitle


\section{Introduction}
In recent years, there has been a series of developments in constructing $H(\div\div)$-conforming finite elements~\cite{ChenHuang2020,ChenHuangDivRn2022,Chen;Huang:2020Finite,Hu;Liang;Ma:2021Finite,Hu;Liang;Ma;Zhang:2022conforming,Hu;Ma;Zhang:2020family}. However, all these elements possess vertex degree of freedom (DoF), which makes them non-hybridizable. In this paper, we present a novel $H(\div\div)$-conforming finite element that is hybridizable, enabling its efficient use in the numerical solutions of the biharmonic equation.

Let $\Omega\subset \mathbb R^d, d\geq 2$, be a Lipschitz domain. With the space $\mathbb S$ of symmetric tensors, the Sobolev space
\[
H(\div{\div },\Omega; \mathbb{S}):=\{\boldsymbol{\tau}\in L^{2}(\Omega; \mathbb{S}): \div {\div}\boldsymbol{\tau}\in L^{2}(\Omega)\}
\]
with the inner $\div$ applied row-wisely to $\boldsymbol \tau$ resulting in a column vector for which the outer $\div$ operator is applied. The $H(\div\div)$-conforming finite elements constructed in~\cite{ChenHuang2020,ChenHuangDivRn2022,Chen;Huang:2020Finite,Hu;Liang;Ma:2021Finite,Hu;Liang;Ma;Zhang:2022conforming,Hu;Ma;Zhang:2020family} include the following DoFs:
\begin{align}
\boldsymbol \tau (\texttt{v}), & \quad \texttt{v}\in \Delta_0(T), \bs \tau\in \mathbb S,\label{intro:HdivdivSfemdof1}\\
(\boldsymbol n_i^{\intercal}\boldsymbol \tau\boldsymbol n_j, q)_f, & \quad q\in\mathbb P_{k-r-1}(f), f\in\Delta_{r}(T), r=1,\ldots, d-1,\label{intro:HdivdivSfemdof2}\\
&\quad \textrm{ and } i,j=1,\ldots, d-r, i\leq j. \notag
\end{align}
Here, $\Delta_r(T)$ denotes the set of $r$-dimensional faces of the simplex $T$. Furthermore, $\boldsymbol{n}_i$ denotes the $i$th normal vector to the face $f$, and $(\cdot, \cdot)_f$ denotes the $L^2$-inner product over the face $f$.
The new element will be constructed by redistributing the vertex and normal plane DoFs~\eqref{intro:HdivdivSfemdof1}-\eqref{intro:HdivdivSfemdof2}. 

We provide a brief explanation of the redistribution process by examining DoFs of vertex $\texttt{v}_0$. Face-normal vectors $\{\bs n_{F_i}, i=1,\ldots, d\}$ form a basis of the ambient Euclidean space $\mathbb R^d,$ $d\geq 2$, where $F_i$ denotes the $(d-1)$-dimensional face containing $\texttt{v}_0$ and opposite to $\texttt{v}_i$ for $i=1,\ldots, d$. We may then determine DoF $\boldsymbol \tau (\texttt{v}_0)\in \mathbb S$ by considering the symmetric matrix $( \bs n_{F_i}^{\intercal}\boldsymbol \tau (\texttt{v}_0)\bs n_{F_j})_{i,j=1,\ldots, d}$.
We redistribute the diagonal entry $\bs n_{F_i}^{\intercal}\boldsymbol \tau (\texttt{v}_0)\bs n_{F_i}$ to face $F_i$ for $i=1,\ldots, d$, while the off-diagonal entries $\bs n_{F_i}^{\intercal}\boldsymbol \tau (\texttt{v}_0)\bs n_{F_j}$ with $1\leq i< j\leq d$ to the $(d-2)$-dimensional face $e_{ij} = F_i\cap F_j$. This process can be extended to DoFs~\eqref{intro:HdivdivSfemdof2} as well by setting $\bs n_i = \bs n_{F_i}$.

In three dimensions, where $d=3$, the faces $F_i$ correspond to two-dimensional faces (i.e., ``faces'') and the $e_{ij}$ correspond to one-dimensional faces (i.e., ``edges''). We refer to this entire process as the redistribution of vertex DoFs to faces and edges. See Fig. \ref{fig:3Dredistribution} for an illustration of the redistribution.
\begin{figure}[htbp]
\begin{center}
\includegraphics[width=4.5cm]{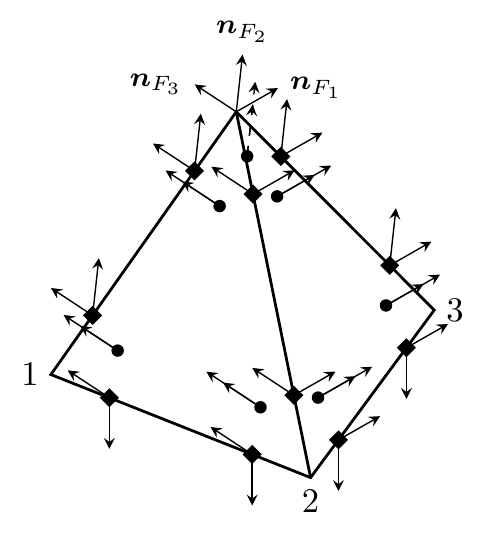}
\caption{Redistribution of vertex degrees of freedom to faces and edges. $\boldsymbol \tau (\texttt{v}_0)\in \mathbb S$ is a symmetric tensor containing $6$ components. Three diagonal entries $\bs n_{F_i}^{\intercal}\boldsymbol \tau (\texttt{v}_0)\bs n_{F_i}$ will be distributed to faces $F_i$ for $i=1,2,3$ and three off-diagonal entries $\bs n_{F_i}^{\intercal}\boldsymbol \tau (\texttt{v}_0)\bs n_{F_j}$ to the edges $e_{ij} = F_i\cap F_j$  with $1\leq i< j\leq 3$.}
\label{fig:3Dredistribution}
\end{center}
\end{figure}

Upon redistribution, we use the geometric decomposition of the Lagrange element to merge facewise DoFs into normal-normal components as shown below:
\begin{equation*}
(\bs n_F^{\intercal}\bs \tau \bs n_F, q )_F, \quad q\in \mathbb P_k(F), F\in \Delta_{d-1}(T), 
\end{equation*}
and merge the off-diagonal DoFs as shown below:
\begin{equation}\label{intro:nnedge}
(\boldsymbol n_{F_1}^{\intercal}\boldsymbol \tau\boldsymbol n_{F_2}, q)_e,\quad q\in \mathbb P_k(e), e\in \Delta_{d-2}(T),
\end{equation}
where  $F_1$ and $F_2$ are the two faces of the element $T$ that share the edge $e$. 

To ensure the $H(\div\div)$-conformity, we modify DoF~\eqref{intro:nnedge} on $\boldsymbol n_{F_1}^{\intercal}\boldsymbol \tau\boldsymbol n_{F_2}$ to an edge jump term given by
\begin{align*}
\tr_e(\bs \tau) = \tr_e^T(\bs \tau) &= \bs n_{F_1,e}^{\intercal}\bs \tau \bs n_{F_1,\partial T}+\bs n_{F_2,e}^{\intercal}\bs \tau \bs n_{F_2,\partial T},
\end{align*}
where $\bs n_{F,e}$ denotes the  normal direction of $e$ on $F$ induced by the orientation of $F$, and $\bs n_{F_i,\partial T}$ is the outward normal direction of face $F_i$ with respect to $\partial T$. Here $T$ represents a simplex and $\intercal$ is the transpose operator. 

We provide DoFs in~\eqref{eq:newdivdivS} and prove the unisolvence to the shape function space $\mathbb P_k(T;\mathbb S)$ for $k\geq 3$. Afterwards, we define the global space $\Sigma_{k}^{\div\div-}$:
\begin{align*}
\Sigma_{k}^{\div\div-} := \{\boldsymbol{\tau}\in L^2(\Omega;\mathbb S): &\, \boldsymbol{\tau}|_T\in \mathbb P_{k}(T;\mathbb S)\textrm{ for each } T\in\mathcal T_h, \\
&\; \textrm{  DoFs on } \tr_1(\bs \tau) \text{ and } \tr_2(\bs \tau) \text{ are single-valued} \},
\end{align*}
where the traces $\tr_1(\bs \tau) = \bs n^{\intercal}\bs \tau \bs n$ and $\tr_2(\bs \tau) =  \boldsymbol n_{\partial T}^{\intercal}\div \boldsymbol \tau +  \div_F(\boldsymbol\tau \boldsymbol n_{\partial T})$ are continuous for $\bs \tau\in \Sigma_{k}^{\div\div-}$. However, the edge jump $\sum_{T\in \omega_e} \tr_e(\bs \tau)|_e$ may not vanish which prevents $\Sigma_{k}^{\div\div-}$ being $H(\div\div)$-conforming, where $\omega_e = \{T\in \mathcal T_h: e\subset T\}$ is the set of all simplices containing $e$.
To obtain an $H(\div\div)$-conforming subspace, we further define the subspace $\Sigma_{k,{\rm new}}^{\div\div}$ as the subspace of $\Sigma_{k}^{\div\div-}$ satisfying the constraint:
\begin{equation*}
\Sigma_{k,{\rm new}}^{\div\div} := \{\bs \tau\in \Sigma_{k}^{\div\div-}:  \sum_{T\in \omega_e} \tr_e(\bs \tau)|_e = 0 \text{  for all  } e\in \mathring{\mathcal E}_h\}.
\end{equation*}
A similar constraint can be found in~\cite{CockbuGopala2005Incompressible} when considering hybridization of edge elements. The space $\Sigma_{k,{\rm new}}^{\div\div}$ is $H(\div\div)$-conforming and compared with other existing elements, the imposed continuity is minimal~\cite[Proposition 3.6]{Fuhrer;Heuer;Niemi:2019ultraweak} and no super-smoothness imposed in lower dimensional sub-simplices. In particular, no vertex DoFs are needed. 

The requirement $k\geq 3$ can be relaxed to $k\geq 2$ by enriching the shape function space
$$
\Sigma_{k^+}(T;\mathbb S):=\mathbb P_k(T;\mathbb S) \oplus \boldsymbol x\boldsymbol x^{\intercal}\mathbb H_{k-1}(T),
$$
which is in the spirit of the Raviart-Thomas (RT) element for $H(\div)$-conforming vector finite element~\cite{Raviart.P;Thomas.J1977,ArnoldFalkWinther2006}. 
A Raviart-Thomas type $H(\div\div)$-conforming finite element space $\Sigma_{k^+}^{\div\div}$ for symmetric tensors can be constructed for $k\geq 2$.

Motivated by the construction in~\cite{FuehrerHeuer2023} in 2D, we further construct a lower order space $H(\div\div)$-conforming finite element $\Sigma_{1^{++}}$ by enriching $\mathbb P_1(T; \mathbb S)$ by some quadratic and cubic polynomials. The 3D version is illustrated in Fig.~\ref{fig:3Dsimple}. 

\begin{figure}[htbp]
\subfigure[The lowest degree $H(\div\div)$-conforming element $\Sigma_{1^{++}}(T;\mathbb S)$ (with $36$ DoFs) for $\bs \sigma$.]{
\begin{minipage}[t]{0.5\linewidth}
\centering
\includegraphics*[width=4.25cm]{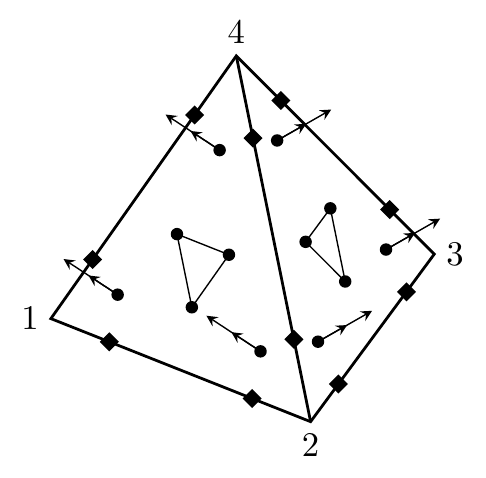}
\end{minipage}}
\subfigure[Discontinuous $\mathbb P_1(T)$ element (with $12$ DoFs) for $u$.]
{\begin{minipage}[t]{0.52\linewidth}
\centering
\includegraphics*[width=4.2cm]{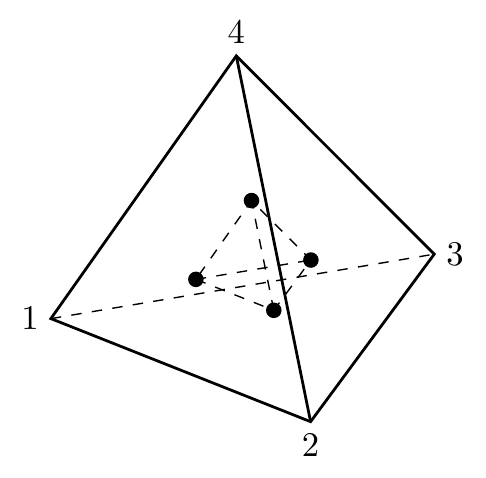}
\end{minipage}}
\caption{The lowest degree pair $\Sigma_{1^{++}}(T;\mathbb S) - \mathbb P_1(T)$ in three dimensions.}
\label{fig:3Dsimple}
\end{figure}

The symmetric tensor finite element with only normal-normal continuity for $k\geq0$ is shown in~\cite{Hellan1967,Herrmann1967,Johnson1973,PechsteinSchoeberl2011,Pechstein;Schoberl:2018analysis}.
For the discretization of the biharmonic equation in two dimensions, referred to as the Hellan-Herrmann-Johnson (HHJ) mixed method~\cite{Hellan1967,Herrmann1967,Johnson1973,ArnoldBrezzi1985,Comodi1989,ChenHuHuang2018}, the normal-normal continuous finite element for symmetric tensors is employed. Notably, there is currently no existing HHJ method for dimensions greater than two. The normal-normal continuous finite element for symmetric tensors is also adopted in~\cite{PechsteinSchoeberl2011,Pechstein;Schoberl:2018analysis} to discretize the linear elasticity, known as the tangential-displacement normal-normal-stress (TDNNS) method. We also refer to~\cite{christiansenFiniteElementSystems2023} for an $H(\rm rot~rot)$-conforming finite element for symmetric tensors on the Clough-Tocher split in two dimensions.

The $H(\div\div)$-conforming finite element constructed in this paper is applicable for discretizing the biharmonic equation for all dimensions $d\geq 2$ and offers optimal convergence for symmetric tensors, along with a fourth-order 
higher superconvergence for the postprocessed deflection. Through a hybridization technique, the implementation of the mixed method developed in this paper can be treated as a generalization of hybridized HHJ methods from 2D to arbitrary dimensions.

The $H(\div\div)$-conforming space $\Sigma_{k,{\rm new}}^{\div\div}$  might be somewhat challenging to implement in practical applications. This complexity arises from the stringent continuity requirements placed on $\tr_1(\bs \tau)$ and $\tr_2(\bs \tau)$, as well as the patch constraint imposed on edge jumps. To mitigate these challenges, we employ a hybridization technique~~\cite{fraeijs1965displacement,ArnoldBrezzi1985} that effectively relaxes these continuity conditions. We utilize the discontinuous stress space $\Sigma_{k}^{-1} = V_{k}^{-1}(\mathcal T_h; \mathbb S)$, 
and broken space 
$$
\mathring{M}^{-1}_{k-2,k-1, k, k} = V_{k-2}^{-1}(\mathcal T_h) \times V_{k-1}^{-1}(\mathring{\mathcal F}_h)\times V_{k}^{-1}(\mathring{\mathcal F}_h)\times V_{k}^{-1}(\mathring{\mathcal E}_h), 
$$
where $V_r^{-1}$ denotes the discontinuous polynomial space of degree $r$ with respect to some finite set, $\mathcal T_h$ is a triangulation, $\mathring{\mathcal F}_h$ is the set of interior $(d-1)$-dimensional faces, and $\mathring{\mathcal E}_h$ the set of interior $(d-2)$-dimensional faces. Spaces on $\mathring{\mathcal F}_h$ and $\mathring{\mathcal E}_h$ can be thought of as Lagrange multipliers for the required continuity. For example, $V_{k-1}^{-1}(\mathring{\mathcal F}_h)$ is for $\tr_2(\bs \sigma)$ which is one degree lower than that of $\bs \sigma$ as $\tr_2(\bs \sigma)$ consists of first-order derivatives of $\bs \sigma$.

Define the weak $(\div\div)_w$ operator
$$
(\div\div)_w\bs \sigma := ((\div\div)_T\bs \sigma, -h_F^{-1}[\tr_2(\bs \sigma)]|_F, h_F^{-3}[\bs n^{\intercal}\bs \sigma\bs n]|_F, h_e^{-2}[ \tr_e(\bs \sigma)]|_e).
$$
A hybridized mixed finite element method for the biharmonic equation is: find $\bs \sigma_h \in \Sigma_{k}^{-1}$ and $u_h\in \mathring{M}^{-1}_{k-2,k-1, k, k}$ s.t.
\begin{subequations}\label{intro:biharmonicMfemhy}
 \begin{align}
\label{intro:biharmonicMfemhy1} (\bs \sigma_h, \bs \tau) + ((\div\div)_w \bs \tau,  u_h)_{0,h}&=0  \qquad\qquad\quad \forall~\boldsymbol{\tau} \in \Sigma_{k}^{-1},\\
\label{intro:biharmonicMfemhy2} ((\div\div)_w \bs \sigma_h, v)_{0,h} & = -(f, v_0) \quad \quad  \forall~v\in \mathring{M}^{-1}_{k-2,k-1, k, k},
\end{align}
\end{subequations}
with appropriate modification of $(f,v_0)$ for $k=0,1,2$.
We will establish the following discrete inf-sup condition,
\begin{equation*}
 \inf_{v\in \mathring{M}^{-1}_{k-2,k-1,k,k}} \sup_{\boldsymbol{\tau}\in \Sigma_{k}^{-1}} \frac{((\div\div)_w\boldsymbol{\tau}, v)_{0,h}}{\|\boldsymbol{\tau}\|_{\div\div_w}\|v\|_{0,h}} = \alpha >0, \quad k\geq 0,
\end{equation*}
from which the well-posedness of~\eqref{intro:biharmonicMfemhy} follows. When $k = 0$,~\eqref{intro:biharmonicMfemhy} is equivalent to using the Morley-Wang-Xu element~\cite{WangXu2006} for the biharmonic equation. In other words,~\eqref{intro:biharmonicMfemhy} generalize the popular quadratic Morley element to higher order and to higher dimensions.

%


Optimal convergence rates will be established for the solution $(\bs \sigma_h, u_h)$ to~\eqref{intro:biharmonicMfemhy}:
\begin{equation*}
\| \bs \sigma-\bs \sigma_h\|_0+ |Q_Mu-u_h|_{2,h} + \|Q_Mu-u_h\|_{0,h} \lesssim h^{k+1}|u|_{k+3}.
\end{equation*}
Post-processing techniques can be used to obtain $u_h^*$ with $k\geq3$ satisfying
 \begin{equation*}
\|\nabla_{h}^2(u-u_h^{\ast})\|_{0} \lesssim h^{k+1}|u|_{k+3}, \quad \|u-u_h^{\ast}\|_{0} \lesssim h^{\min\{2k-2, k+3\}}\|u\|_{k+3}. 
\end{equation*}
Hybridization~\eqref{intro:biharmonicMfemhy} can be also generalized to the Raviart-Thomas type $\Sigma_{k^+}^{-1}-\mathring{M}^{-1}_{k-1,k-1, k, k}$ for $k\geq 2$ and $\Sigma_{1^{++}}^{-1}-\mathring{M}^{-1}_{1,1, 1, 1}$ for $k=1$. 

We define the weak Hessian operator $\nabla_w^2$ as the adjoint of $(\div\div)_w$ with respect to a mesh-dependent inner product $(\cdot,\cdot)_{0,h}$. Using the operator $\nabla_w^2$, we can interpret the hybridization~\eqref{intro:biharmonicMfemhy} as a weak Galerkin method for the biharmonic equation, which does not require any additional stabilization:
\begin{equation}\label{intro:biharmonicMfemWG}
(\nabla_w^2 u_h, \nabla_w^2 v) = (f, v_0) \quad \forall~v\in \mathring{M}^{-1}_{k-2,k-1, k, k},
\end{equation}
with appropriate modification of computing $(f, v_0)$ for low order cases.
Restricting~\eqref{intro:biharmonicMfemWG} to different subspaces of $\mathring{M}^{-1}_{k-2,k-1, k, k}$ will derive new discrete methods:
\begin{itemize}
\item 
Embedding the $H^2$-nonconforming virtual element on simplices in~\cite{ChenHuang2020a} into the broken space $\mathring{M}^{-1}_{k-2,k-1, k, k}$, we acquire a stabilization-free non-conforming virtual element method for the biharmonic equation. 
\item Embedding the continuous Lagrange element $\mathring{V}_k$ into $\mathring{M}^{-1}_{k-2,k-1, k, k}$, we obtain a parameter-free $C^0$ discontinuous Galerkin (DG) method for the biharmonic equation, which generalizes the 2D scheme in~\cite{HuangHuang2014} to arbitrary dimension $d\geq 2$.
\end{itemize}

In three dimensions, we construct the finite element  $\div\div$ complex, for $k\geq 3$,
\begin{equation}\label{intro:divdivcomplex3dfem}
{\bf RT}\xrightarrow{\subset} V_{k+2}^H\xrightarrow{\dev\grad} \Sigma_{k+1}^{\sym\curl}\xrightarrow{\sym\curl} \Sigma_{k,{\rm new}}^{\div\div} \xrightarrow{\div{\div}} V^{-1}_{k-2}(\mathcal T_h)\xrightarrow{}0,
\end{equation}
where ${\bf RT}:= \{a\boldsymbol x + \boldsymbol b: a\in \mathbb R, \boldsymbol b \in \mathbb R^3\}$, and $V_{k+2}^H$ is the vectorial Hermite element space~\cite{Ciarlet1978}.
Since no supersmooth DoFs for space $\Sigma_{k,{\rm new}}^{\div\div}$, we construct  $H(\sym\curl;\mathbb T)$-conforming finite element space $\Sigma_{k+1}^{\sym\curl}$ simpler than those in~\cite{ChenHuang2022,Chen;Huang:2020Finite,Hu;Liang;Ma:2021Finite,Hu;Liang;Ma;Zhang:2022conforming}. Lower order finite element  $\div\div$ complexes for $k=1,2$ in three dimensions are also constructed.
The first half of the complex~\eqref{intro:divdivcomplex3dfem} can be replaced by
\begin{equation}\label{intro:firsthalf}
{\bf RT}\xrightarrow{\subset} V_{k+2}^L\xrightarrow{\dev\grad} \overline{\Sigma}_{k+1}^{\sym\curl}\xrightarrow{\sym\curl}, 
\end{equation}
and the second half by 
$
\xrightarrow{\sym\curl} \Sigma_{k^+}^{\div\div} \xrightarrow{\div{\div}} V^{-1}_{k-1}(\mathcal T_h)\xrightarrow{}0,
$
which leads to several variants of~\eqref{intro:divdivcomplex3dfem}; see Section \ref{sec:femdivdivcomplex} for details.

With the weak $\div\div_w$ operator, for $k\geq 1$, we can construct the distributional finite element divdiv complex 
\begin{equation*}
{\bf RT}\xrightarrow{\subset} V_{k+2}^H\xrightarrow{\dev\grad} \Sigma_{k+1}^{\sym\curl}\xrightarrow{\sym\curl} \Sigma_{k}^{-1} \xrightarrow{\div{\div}_w} \mathring{M}^{-1}_{k-2,k-1, k, k}\xrightarrow{}0.
\end{equation*}
The normal-normal continuous finite element $\Sigma_{k}^{\rm nn}$ can be treated as a subspace of $\Sigma_{k}^{-1}$ and the corresponding distributional divdiv complex becomes, for $k\geq 1$, 
\begin{equation*}
{\bf RT}\xrightarrow{\subset} V_{k+2}^H\xrightarrow{\dev\grad} \Sigma_{k+1}^{\sym\curl}\xrightarrow{\sym\curl} \Sigma_{k}^{\rm nn} \xrightarrow{\div{\div}_w} \mathring{M}^{-1}_{k-2,k-1, \cdot, k}\xrightarrow{}0,
\end{equation*}
which can be treated as a generalization of 2D distributive divdiv complex involving HHJ elements developed in \cite{ChenHuHuang2018} to 3D. Again the first half can be replaced by~\eqref{intro:firsthalf} for $k\geq0$ and more variants, including $k=0$ case, can be found in Section \ref{sec:distributivedivdiv}.

The rest of this paper is organized as follows. Hybridizable $H(\div\div)$-conforming finite elements in arbitrary dimension are constructed in Section~\ref{sec:divdivfem}. 
A mixed finite element method together with error analysis, post-processing, and duality argument are presented in Section~\ref{sec:mfem}. 
Then in Section~\ref{sec:hybridization}, 
the hybridization and its equivalence to other methods are presented for the mixed finite element method of the biharmonic equation.
Several new finite element divdiv complexes in three dimensions are devised in Section~\ref{sec:femdivdivcomplex}.

\section{$H(\div\div)$-conforming finite elements}\label{sec:divdivfem}
In this section, we discuss $H(\div\div)$-conforming finite elements. We review existing finite elements that enforce conformity by ensuring continuity on the normal plane of lower-dimensional sub-simplices, which is known as super-smoothness. By using a redistribution technique, we obtain a new element without such super-smoothness. Additionally, we construct a Raviart-Thomas type element using enriched polynomial spaces.

\subsection{Notation}

Let $\Omega\subset \mathbb{R}^d~(d\geq 2)$ be a bounded
polytope.
Given a bounded domain $D\subset\mathbb{R}^{d}$ and a
non-negative integer $k$, let $H^k(D)$ be the usual Sobolev space of functions
over $D$, whose norm and semi-norm are denoted by
$\Vert\cdot\Vert_{k,D}$ and $|\cdot|_{k,D}$ respectively. 
Define $H_0^k(D)$ as the closure of $C_{0}^{\infty}(D)$ with
respect to the norm $\Vert\cdot\Vert_{k,D}$.
Let $(\cdot, \cdot)_D$ be the standard inner product on $L^2(D)$. If $D$ is $\Omega$, we abbreviate
$\Vert\cdot\Vert_{k,D}$, $|\cdot|_{k,D}$ and $(\cdot, \cdot)_D$ by $\Vert\cdot\Vert_{k}$, $|\cdot|_{k}$ and $(\cdot, \cdot)$,
respectively.
Denote by $h_D$ the diameter of $D$.

For a $d$-dimensional simplex $T$, we let $\Delta(T)$ denote all the subsimplices of $T$, while $\Delta_{\ell}(T)$ denotes the set of subsimplices of dimension $\ell$, for $0\leq \ell \leq d$. 

For $f\in\Delta_{\ell}(T)$ with $0\leq \ell\leq d$, let $\boldsymbol n_{f,1}, \cdots, \boldsymbol n_{f,d-\ell}$ be linearly independent unit normal vectors, and $\boldsymbol t_{f,1}, \cdots, \boldsymbol t_{f,\ell}$ be its orthonormal tangential vectors.
We abbreviate $\boldsymbol n_{F,1}$ as $\boldsymbol n_{F}$ or $\bs n$ when $\ell=d-1$.
We also abbreviate $\boldsymbol n_{f,i}$ and $\boldsymbol t_{f,i}$ as $\boldsymbol n_{i}$ and $\boldsymbol t_{i}$ respectively if not causing any confusion. 
For a $(d-1)$-dimensional face $F\in\partial T$ and a $(d-2)$-dimensional  face $e\in\partial F$, $\bs n_{F,e}$ denotes the  normal direction of $e$ on $F$ induced by the orientation of $F$. When $d=2$, $e$ is a vertex and $F$ is an edge. Then $\bs n_{F,e} = \bs t_F$ if $e$ is the end point of $F$ for the orientation given by $\bs t_F$ and $\bs n_{F,e} = - \bs t_F$ otherwise. We use $\bs n_{\partial T}$ to denote the unit outward normal vector of $\partial T$ which is a piecewise constant vector function.

Given a face $F\in\Delta_{d-1}(T)$, and a vector $\boldsymbol v\in \mathbb R^d$, define 
$$
\Pi_F\boldsymbol v:= (\boldsymbol n_F\times \boldsymbol v)\times \boldsymbol n_F = (\boldsymbol I - \boldsymbol n_F\boldsymbol n_F^{\intercal})\boldsymbol v
$$ 
as the projection of $\boldsymbol v$ onto the face $F$. 
For a scalar function $v$, define the surface gradient
\begin{equation*}
\nabla_{F}v:=\Pi_F\nabla v = \nabla v-\frac{\partial v}{\partial n_{F}}\boldsymbol n_{F}=\sum_{i=1}^{d-1}\frac{\partial v}{\partial t_{F,i}}\boldsymbol t_{F,i},
\end{equation*}
namely the projection of $\nabla v$ to the face $F$, which is independent of the choice of the normal vectors. Denote by $\div_{F}\bs v:=\nabla_{F}\cdot(\Pi_F\bs v)$ the corresponding surface divergence.

Denote by $\mathcal{T}_h$ a conforming triangulation of $\Omega$ with each element being a simplex, where $h:=\max_{T\in\mathcal{T}_h}h_T$.
Let $\mathcal{F}_h$, $\mathring{\mathcal{F}}_h$, $\mathcal{E}_h$ and $\mathring{\mathcal{E}}_h$ be the set of all $(d-1)$-dimensional faces, interior $(d-1)$-dimensional faces, $(d-2)$-dimensional faces and interior $(d-2)$-dimensional faces, respectively.
Set $\mathcal F_h^{\partial}:=\mathcal{F}_h\backslash\mathring{\mathcal{F}}_h$ and $\mathcal E_h^{\partial}:=\mathcal{E}_h\backslash\mathring{\mathcal{E}}_h$.
For $e\in \mathcal E_h$, denote by $\omega_e := \{T\in \mathcal T_h: e\subset T\}$ as the set of all simplices containing $e$. 
We use $\nabla_h, \nabla_h^2$ and $(\div\div)_h$ to represent the element-wise gradient, Hessian and $\div\div$ with respect to $\mathcal T_h$.
Consider two adjacent simplices $T_1$ and $T_2$ sharing an interior face $F$. Define the average and the jump of a function $w$ on $F$ as
\begin{equation*}
\{w\}:=\frac{1}{2}\big((w|_{T_1})|_F+(w|_{T_2})|_F\big),\;\; [w]:=(w|_{T_1})|_F\boldsymbol{n}_F\cdot\boldsymbol{n}_{\partial T_1}+(w|_{T_2})|_F\boldsymbol{n}_F\cdot\boldsymbol{n}_{\partial T_2}.
\end{equation*}
On a face $F$ lying on the boundary $\partial\Omega$, the above terms become
\begin{equation*}
\{w\}:=w|_F,\quad [w]:=w|_F.
\end{equation*}

For a bounded domain $D\subset\mathbb{R}^{d}$ and 
a non-negative integer $k$, 
let $\mathbb P_k(D)$ stand for the set of all polynomials over $D$ with the total degree no more than $k$. When $k<0$, set $\mathbb P_k(D) := \{0\}.$
Let $Q_{k,D}$ be the $L^2$-orthogonal projector onto $\mathbb P_k(D)$, and $Q_{k}$ its element-wise version with respect to $\mathcal{T}_h$. Let $\mathbb H_k(D):=\mathbb P_k(D)\backslash\mathbb P_{k-1}(D)$ be the space of homogeneous polynomials of degree $k$. 
In the binomial coefficient notation ${n\choose k}$, if $n\geq 0, k<0$, we set ${n\choose k} := 0$.

Let $V_{k}^{-1}(\mathcal T_h) := \prod_{T\in \mathcal T_h} \mathbb P_{k}(T)$ for $k\geq 0$ and abbreviate as $V_{k}^{-1}$ when the dependence of $\mathcal T_h$ is not emphasized.

Set $\mathbb M:=\mathbb R^{d\times d}$.
Denote by $\mathbb S$, $\mathbb K$  and $\mathbb T$ the subspace of symmetric matrices, skew-symmetric matrices and traceless matrices of $\mathbb M$, respectively. For a space $B(D)$ defined on $D$,
let $B(D; \mathbb{X}):=B(D)\otimes\mathbb{X}$ be its vector or tensor version for $\mathbb{X}$ being $\mathbb{R}^d$, $\mathbb{M}$, $\mathbb{S}$, $\mathbb{K}$ and $\mathbb{T}$.

Throughout this paper, we use ``$\lesssim\cdots $" to mean that ``$\leq C\cdots$", where
letter $C$ is a generic positive constant independent of $h$,
which may stand for different values at its different occurrences. The notation $A\eqsim B$ means $B\lesssim A\lesssim B$. 

\subsection{Trace and continuity}

We consider the continuity of a piecewise smooth tensor function to be in the Sobolev space
$$
H(\div{\div },\Omega; \mathbb{S}):=\{\boldsymbol{\tau}\in L^{2}(\Omega; \mathbb{S}): \div {\div}\boldsymbol{\tau}\in L^{2}(\Omega)\},
$$
which plays a central role in our later constructions. We start from the Green's identity established  in~\cite{Chen;Huang:2020Finite,ChenHuangDivRn2022} for the operator divdiv. 

The trace $\tr^{\div\div}\bs \sigma$, as a distribution, is defined as the difference 
$$
\langle \tr^{\div\div}\bs \sigma, \tr^{\nabla^2} v \rangle_{\partial T}:= (\div\div\boldsymbol \sigma, v)_T - (\boldsymbol \sigma, \nabla^2v)_T.
$$
We decompose $\tr^{\div\div}\bs \sigma$ and $\tr^{\nabla^2} v$ into two face-wise trace operators and one edge trace operator.

\begin{lemma}[Lemma 5.2 in~\cite{ChenHuangDivRn2022}]
We have for any $\boldsymbol \sigma\in \mathcal C^2(T; \mathbb S)$ and $v\in H^2(T)$ that
\begin{align}
&\notag  (\div\div\boldsymbol \sigma, v)_T=(\boldsymbol \sigma, \nabla^2v)_T \\
&- \sum_{F\in\partial T}\left[( \tr_1(\bs \sigma), \tr_1(v))_{F} -  ( \tr_2(\bs \sigma), \tr_2(v))_F\right] -\sum_{e\in\Delta_{d-2}(T)}(\tr_e(\bs \sigma), \tr_e(v))_e, \label{eq:greenidentitydivdiv} 
\end{align}
where
\begin{align*}
&\tr_1(\bs \sigma) = \boldsymbol  n_{\partial T}^{\intercal}\boldsymbol \sigma\boldsymbol  n_{\partial T}, &  & \tr_1( v) = \partial_n v\mid_{\partial T},\\
&\tr_2(\bs \sigma) =  \boldsymbol n_{\partial T}^{\intercal}\div \boldsymbol \sigma +  \div_F(\boldsymbol\sigma \boldsymbol n_{\partial T}),&  & \tr_2(v) = v\mid_{\partial T},\\
&\tr_e(\bs \sigma) = \sum_{F\in\partial T,e\in \partial F}\boldsymbol n_{F,e}^{\intercal}\boldsymbol \sigma \boldsymbol n_{\partial T}, &  & \tr_e(v) = v \mid_{\Delta_{d-2}(T)}.
\end{align*}
\end{lemma}

When summing over all elements and assuming the test function $v$ is smooth enough, e.g. $v\in C^2(\Omega)$, we can merge the terms on the interior faces and edges. For an interior face $F\in \  \mathring{\mathcal F}_h$, denote by $T_1, T_2$ two elements containing $F$. 
Introduce the jumps
\begin{align*}
&[\tr_1(\bs \sigma)]_F : = \bs n_{\partial T_1}^{\intercal}\bs \sigma \bs n_{\partial T_1} \mid _{F} - \bs n_{\partial T_2}^{\intercal}\bs \sigma \bs n_{\partial T_2} \mid _{F},  \\
&[\tr_2(\bs \sigma)]_F := (\boldsymbol n_{\partial T_1}^{\intercal}\div \boldsymbol \sigma +  \div_F(\boldsymbol\sigma \boldsymbol n_{\partial T_1}))\mid _{F} + (\boldsymbol n_{\partial T_2}^{\intercal}\div \boldsymbol \sigma +  \div_F(\boldsymbol\sigma \boldsymbol n_{\partial T_2}))\mid _{F},  &  &\\
& [ \tr_e(\bs \sigma)]|_e := \sum_{T\in \omega_e} \sum_{F\in\partial T, e\in\partial F} (\boldsymbol n_{F,e}^{\intercal}\boldsymbol \sigma \boldsymbol n_{\partial T})|_e.  &  &
\end{align*}
We recall the results from~\cite{Fuhrer;Heuer;Niemi:2019ultraweak} using our notation.
\begin{lemma}[Proposition 3.6 in~\cite{Fuhrer;Heuer;Niemi:2019ultraweak}]\label{lm:divdivconforming}
Let $\bs \sigma \in L^2(\Omega;\mathbb S)$ and $\bs \sigma|_T\in H^{2}(T;\mathbb S)$ for each $T\in \mathcal T_h$. Then $\bs \sigma \in H(\div\div,\Omega;\mathbb S)$ if and only if 
\begin{enumerate}
\item $[\tr_1(\bs \sigma)]_F = 0$ for all $F\in \  \mathring{\mathcal F}_h$;
\smallskip
\item $[\tr_2(\bs \sigma)]_F = 0$ for all $F\in \  \mathring{\mathcal F}_h$;
\smallskip
\item $[ \tr_e(\bs \sigma)]|_e = 0$ for all $e\in \ \mathring{\mathcal E}_h$. 
\end{enumerate} 
\end{lemma}

Enforcing the jump condition $[ \tr_e(\bs \sigma)]|_e = 0$ in $H(\div\div)$-conforming finite element constructions is a challenging task as the constraint is imposed in the patch of $e$.
The continuity of $\bs \sigma$ projected onto the normal plane $\mathcal N_e$ of $e\in \mathring{\mathcal E}_h$ is sufficient but  by no means necessary.
More specifically, as a $(d-2)$-dimensional sub-simplex, the  dimension of the normal plane $\mathcal N_e$ is two. To enforce the continuity condition, we  choose two orthonormal directions ${\bs n_1, \bs n_2}$ normal to $e$ for each edge $e\in \mathcal E_h$. It is important to note that $\mathcal N_e$ depends solely on $e$ and not on the elements containing it.
We denote the space of $2\times 2$ symmetric matrices on $\mathcal N_e$ by $\mathbb S(\mathcal N_e)$, and define $Q_{\mathcal N_e}(\bs \sigma):=(\bs n_i^{\intercal}\bs \sigma\bs n_j)_{i,j=1,2}$ as the projection of $\bs \sigma\in \mathbb S$ onto $\mathbb S(\mathcal N_e)$. 

\begin{lemma}\label{lm:edgejump}
Let $\bs \sigma \in L^2(\Omega;\mathbb S)$ and $\bs \sigma|_T\in H^{2}(T;\mathbb S)$ for each $T\in \mathcal T_h$.
If $Q_{\mathcal N_e}(\bs \sigma)$ is continuous on $e$, then $[ \tr_e(\bs \sigma)]|_e = 0$ for all $e\in \ \mathring{\mathcal E}_h$. 
\end{lemma}
\begin{proof}
For each $F$ containing $e\in \mathring{\mathcal E}_h$, $F$ is also interior and thus there exist exactly two elements $T_1, T_2$ in the edge patch $\omega_e$ s.t. $F\in \partial T_i, i=1,2$. The normal vector $\bs n_{F,e}$ is induced by the orientation of $F$ which is independent of the elements but $\bs n_{\partial T}$ is the outward normal direction depending on the element  $T$ containing $F$, and $\bs n_{\partial T_1}\mid_F = -\bs n_{\partial T_2}\mid _F$. 
%
Therefore 
$(\boldsymbol n_{F,e}^{\intercal}\boldsymbol \sigma \boldsymbol n_{\partial T_1} + \boldsymbol n_{F,e}^{\intercal}\boldsymbol \sigma \boldsymbol n_{\partial T_2})|_e = 0$ and consequently $[ \tr_e(\bs \sigma)]|_e = 0$.
\end{proof}

\subsection{$H(\div\div)$-conforming finite elements}\label{sec:divdivconforming}
Several $H(\div\div)$-conforming finite elements have been constructed in a series of recent works~\cite{ChenHuang2020,Chen;Huang:2020Finite,ChenHuangDivRn2022,Hu;Liang;Ma:2021Finite,Hu;Liang;Ma;Zhang:2022conforming,Hu;Ma;Zhang:2020family}. In the following, we recall the version presented in~\cite[Theorem 5.10]{ChenHuangDivRn2022} with a slight change in the notation: $r$ in~\eqref{HdivdivSfemdof2} represents the dimension of the sub-simplex while in ~\cite[Theorem 5.10]{ChenHuangDivRn2022}, it is the co-dimension.

Recall that, for a simplex $T$ and an integer $k\geq 0$, the first kind Ned\'el\'ec element~\cite{Nedelec:1986family} is
$$
{\rm ND}_{k}(T) = \mathbb P_{k}(T;\mathbb R^d) \oplus \mathbb H_{k}(T; \mathbb K)\boldsymbol x =\grad \mathbb P_{k+1}(T)\oplus\mathbb P_{k}(T;\mathbb K)\boldsymbol x.
$$
Let ${\bf RM} :={\rm ND}_{0}(T)$ be the kernel of the operator $\defm := {\rm sym }\grad$. We have ${\bf RM}\subseteq{\rm ND}_{k-3}(T)$ when $k\geq3$.

For $k\geq 3$, the shape function space is $\Sigma_k(T; \mathbb S):=\mathbb P_k(T;\mathbb S)$ and degrees of freedom (DoFs) are given by
\begin{subequations}\label{eq:divdivS}
\begin{align}
\boldsymbol \tau (\texttt{v}), & \quad \texttt{v}\in \Delta_0(T), \label{HdivdivSfemdof1}\\
(\boldsymbol  n_i^{\intercal}\boldsymbol \tau\boldsymbol n_j, q)_f, & \quad q\in\mathbb P_{k-r-1}(f),  f\in\Delta_{r}(T),\;  r=1,\ldots, d-1,\label{HdivdivSfemdof2}\\
&\quad \textrm{ and } i,j=1,\ldots, d-r, i\leq j, \notag\\
( \tr_2(\bs \tau), q)_F, &\quad q\in\mathbb P_{k-1}(F), F\in \partial T,\label{HdivdivSfemdof3}\\
(\Pi_F\boldsymbol \tau\boldsymbol n, \boldsymbol q)_F, & \quad \boldsymbol q\in {\rm ND}_{k-2}(F),  F\in\partial T,\label{HdivdivSfemdof4}\\
(\boldsymbol \tau, \defm\boldsymbol q)_T, &\quad \boldsymbol q\in {\rm ND}_{k-3}(T)\backslash {\bf RM},\label{HdivdivSfemdof5}\\
(\boldsymbol \tau, \boldsymbol q)_T, &\quad \boldsymbol q\in  \ker (\cdot\boldsymbol x)\cap \mathbb P_{k-2}(T;\mathbb S), \label{HdivdivSfemdof6}
\end{align}
\end{subequations}
We can view DoFs in~\eqref{HdivdivSfemdof1} as a special case of those in~\eqref{HdivdivSfemdof2} if we treat $\mathbb R^d$ as the normal plane of the vertex $\texttt{v}$. DoFs~\eqref{HdivdivSfemdof1}-\eqref{HdivdivSfemdof2} will determine the trace $\bs n^{\intercal}\boldsymbol \tau\boldsymbol n$ and also imply the continuity of $\bs \tau$ on the normal plane of edges. Notice that DoF~\eqref{HdivdivSfemdof2} only exists for sub-simplex with dimension $r\leq k-1$. DoF~\eqref{HdivdivSfemdof3} is to impose the continuity of $\tr_2(\bs \tau)=\boldsymbol n_F^{\intercal}\div \boldsymbol \tau +  \div_F(\boldsymbol\tau \boldsymbol n_F),$ which is modified from the DoF of $\bs n^{\intercal}\div \bs \tau$ on $F$. To have the surjection $\bs n_F^{\intercal}\div \mathbb P_k(T;\mathbb S) = \mathbb P_{k-1}(F)$, the degree $k\geq 3$ is required; see~\cite[Lemma 5.3]{ChenHuangDivRn2022}. Moreover, $k\geq 3$ is also required so that ${\rm ND}_{k-3}(T)\backslash {\bf RM}$ in DoF~\eqref{HdivdivSfemdof5} is meaningful. The space ${\rm ND}_{k-3}(T)\backslash {\bf RM}$ can be any sub-space $X\subset {\rm ND}_{k-3}(T)$ satisfying ${\rm ND}_{k-3}(T)={\bf RM}\oplus X$.
 Since the kernel of the operator $\defm$ is ${\bf RM}$, in~\eqref{HdivdivSfemdof5}, we can also write ${\rm ND}_{k-3}(T)$ only. 
 
For $k=0,1,2$, one can check by direct calculation that the number of DoFs is more than the dimension of the shape function space. See also Remark \ref{rm:k=2}.

\begin{lemma}[Theorem 5.10~in~\cite{ChenHuangDivRn2022}]\label{lm:divdivS}
For $k\geq 3$, the DoFs~\eqref{eq:divdivS} are unisolvent for the space $\mathbb P_k(T;\mathbb S)$.
\end{lemma}

\begin{remark}\rm 
In~\cite[Theorem 5.10]{ChenHuangDivRn2022}, the requirement $k\geq \max\{d,3\}$ is presented. The condition $k\geq d$ is to ensure DoF~\eqref{HdivdivSfemdof2} exists on $(d-1)$-dimensional faces so that the inf-sup condition holds. Based on the key decomposition in~\cite[Fig. 5.1]{ChenHuangDivRn2022} and the characterization of each component established in Lemma 4.5 for $\tr^{\div}(\mathbb P_k(T; \mathbb S))$ with $k\geq 1$, Lemma 4.11 for $E_0'(\mathbb S)$ with $k\geq 2$, Lemma 5.3 for $\tr^{\div}\div F_r(\mathbb S)$ with $k\geq 3$, and Lemma 5.4 for $F_0'(\mathbb S)$ with $k\geq 3$, the uni-solvence holds with condition $k\geq 3$ only. $\qed$
\end{remark}

The finite element space $\Sigma_{k}^{\div\div}$ is defined as follows
\begin{align*}
\Sigma_{k}^{\div\div}:=\{\boldsymbol \tau\in L^2(\Omega;\mathbb S): & \boldsymbol \tau|_T\in\mathbb P_k(T;\mathbb S) \textrm{ for each } T\in\mathcal T_h, \\
&\; \textrm{ DoFs~\eqref{HdivdivSfemdof1},~\eqref{HdivdivSfemdof2}, and~\eqref{HdivdivSfemdof3} are single-valued} \}.    
\end{align*}
The single-valued DoFs in~\eqref{HdivdivSfemdof1} and~\eqref{HdivdivSfemdof2} imply the continuity of the $Q_{\mathcal N_f}(\bs \tau)$ function for all lower-dimensional sub-simplices $f$ of $T$. In particular, the edge jump vanishes, i.e., $[ \tr_e(\bs \tau)]|_e = 0$ as proven in Lemma \ref{lm:edgejump}. The continuity of $\tr_1(\bs \tau)$ and $\tr_2(\bs \tau)$ are imposed by DoFs~\eqref{HdivdivSfemdof1}-\eqref{HdivdivSfemdof3}. Therefore, we can conclude that $\Sigma_{k}^{\div\div}\subset H(\div\div, \Omega; \mathbb S)$ in view of Lemma \ref{lm:divdivconforming}.

DoF~\eqref{HdivdivSfemdof4} for the tangential-normal component $\Pi_F\boldsymbol \tau\boldsymbol n$ is considered as a local DoF, i.e., it is not single-valued across simplices. 
If DoF~\eqref{HdivdivSfemdof4} is also single-valued, then the function is also in $H(\div,\Omega;\mathbb S)$ and the corresponding element,  which is firstly introduced by Hu, Ma, and Zhang~\cite{Hu;Ma;Zhang:2020family}, is $H(\div\div;\mathbb S)\cap H(\div;\mathbb S)$-conforming.

When $k\geq \max\{d,3\}$, we have the discrete divdiv stability~\cite[Lemma 5.12]{ChenHuangDivRn2022}. Namely $\div\div: \Sigma_{k}^{\div\div}\to V_{k-2}^{-1}(\mathcal T_h)$ is surjective and the following inf-sup condition holds with a constant $\alpha$ independent of $h$
\begin{equation*}
\inf_{p_h\in V^{-1}_{k-2}(\mathcal T_h)}\sup_{\boldsymbol\tau_h\in \Sigma_{k}^{\div\div}}
\frac{(\div\div\boldsymbol{\tau}_h, p_h)}{\|\boldsymbol\tau_h\|_{\div\div} \|p_h\|_0} = \alpha > 0, \quad k\geq \max\{d,3\}.
\end{equation*}
Although the element is well defined for $k\geq 3$, the constraint $k\geq d$ is required for the inf-sup condition. When $k\geq d$, DoF~\eqref{HdivdivSfemdof2} includes the moment $\int_F\boldsymbol n^{\intercal}\boldsymbol \tau\boldsymbol n\dd S$ for $F\in\partial T$, which is required by the fact that the range space $\div\div\Sigma_{k}^{\div\div}$ should include all piecewise linear functions.

Implementing the $H(\div\div)$-conforming element defined by DoF~\eqref{eq:divdivS}  can be challenging due to the high degree $k\geq \max\{d,3\}$ and the relatively complex degrees of freedom. In two dimensions, an $H(\div\div)\cap H(\div)$-conforming element has been successfully implemented and applied to discretize the biharmonic equation using the basis for Hu-Zhang $H(\div; \mathbb S)$-element, as described in~\cite{Hu;Ma;Zhang:2020family}.

We will present a new $H(\div\div)$-conforming finite element with minimal smoothness.
%
For $k\geq 3$, the shape function space is still $\mathbb P_k(T;\mathbb S)$ and the following DoFs~\eqref{eq:newdivdivS} are proposed:
\begin{subequations}\label{eq:newdivdivS}
\begin{align}
(\tr_e(\bs \tau), q)_e, &\quad q\in \mathbb P_k(e), e\in \Delta_{d-2}(T),\label{eq:newdivdivdof1}\\
(\bs n^{\intercal}\bs \tau \bs n, q )_F, &\quad q\in \mathbb P_k(F), F\in \partial T,\label{eq:newdivdivdof2}\\
( \tr_2(\bs \tau), q)_F, &\quad q\in\mathbb P_{k-1}(F), F\in \partial T,\label{eq:newdivdivdof3}\\
(\Pi_F\boldsymbol \tau\boldsymbol n, \boldsymbol q)_F, & \quad \boldsymbol q\in {\rm ND}_{k-2}(F),  F\in\partial T,\label{eq:newdivdivdof4}\\
(\boldsymbol \tau, \defm\boldsymbol q)_T, &\quad \boldsymbol q\in {\rm ND}_{k-3}(T),\label{eq:newdivdivdof5}\\
(\boldsymbol \tau, \boldsymbol q)_T, &\quad \boldsymbol q\in  \ker (\cdot\boldsymbol x)\cap \mathbb P_{k-2}(T;\mathbb S). \label{eq:newdivdivdof6}
\end{align}
\end{subequations}
Comparing with DoFs~\eqref{eq:divdivS}, the difference is that DoFs~\eqref{HdivdivSfemdof1}-\eqref{HdivdivSfemdof2} are redistributed to edges and faces to form DoFs ~\eqref{eq:newdivdivdof1}-\eqref{eq:newdivdivdof2}. 

We now briefly explain the redistribution process. Without loss of generality, consider vertex $\texttt{v}_0$. Choose $\{\bs n_{F_i}, i=1,\ldots, d\}$ as a basis of $\mathbb R^d$, where $F_i$ is the $(d-1)$-dimensional face containing $\texttt{v}_0$ for $i=1,\ldots, d$. DoF $\boldsymbol \tau (\texttt{v}_0)\in \mathbb S$ is determined by the symmetric matrix $( \bs n_{F_i}^{\intercal}\boldsymbol \tau (\texttt{v}_0)\bs n_{F_j})_{i,j=1,\ldots, d}$. We redistribute the diagonal entry $\bs n_{F_i}\boldsymbol \tau (\texttt{v}_0)\bs n_{F_i}$ to face $F_i$, for $i=1,\ldots, d$, and the off-diagonal $\bs n_{F_i}\boldsymbol \tau (\texttt{v}_0)\bs n_{F_j}, 1\leq i< j\leq d$, to edge $e_{ij} = F_i\cap F_j$. 
Such redistribution can be generalized to DoF~\eqref{HdivdivSfemdof2}. For a lower dimensional sub-simplex $f\in \Delta_r(T), r=1,\ldots, d-1$, use $\{\bs n_{F_i}, f\in \Delta_r(F_i), i=1,\ldots, d-r\}$ as the basis of the normal plane $\mathcal N_f$ of $f$. We can redistribute the diagonal $\bs n_{F_i}^{\intercal}\boldsymbol \tau\bs n_{F_i}|_f$ to face $F_i$ and off-diagonal $\bs n_{F_i}^{\intercal}\boldsymbol \tau\bs n_{F_j}|_f$ to edge $e_{ij}= F_i\cap F_j$. 

After the redistribution, we merge DoFs. A function $u\in \mathbb P_k(T)$ can be determined by
\begin{equation}\label{eq:PkDoF}
(u, q)_T, \quad q\in \mathbb P_k(T).
\end{equation}
Recall that the geometric decomposition of the Lagrange element in~\cite[(2.6)]{ArnoldFalkWinther2009} is
\begin{align}
\label{eq:Prdec}
\mathbb P_k(T) &= \Oplus_{r = 0}^{d}\Oplus_{f\in \Delta_{r}(T)} b_f\mathbb P_{k - (r +1)} (f),
\end{align}
where $b_f\in\mathbb P_{r+1}(f), b_f|_{\partial f} = 0,$ is the $\mathbb P_{r+1}$-polynomial bubble function on $f$.
Based on~\eqref{eq:Prdec}, DoF~\eqref{eq:PkDoF} can be decomposed into
\begin{equation}\label{eq:PkdecDoF}
(u, q)_f, \quad q\in\mathbb P_{k-r-1}(f),  f\in\Delta_{r}(T),\;  r=0,1,\ldots, d.
\end{equation}
Vice versa, DoFs in~\eqref{eq:PkdecDoF} can be merged into~\eqref{eq:PkDoF}. 

After redistribution, we merge DoFs facewisely and edgewisely. For example, on a face $F$, we will have DoFs
\begin{equation}\label{eq:nnface}
(\boldsymbol  n_F^{\intercal}\boldsymbol \tau\boldsymbol n_F, q)_f,  \quad q\in\mathbb P_{k-r-1}(f),  f\in\Delta_{r}(F),\;  r=0,1,\ldots, d-1.
\end{equation}
By the decomposition of the Lagrange element~\eqref{eq:Prdec}, we can merge~\eqref{eq:nnface} to DoF~\eqref{eq:newdivdivdof2}. Similarly on an edge $e$ shared by $F_1$ and $F_2$, we merge DoFs for $\boldsymbol  n_{F_1}^{\intercal}\boldsymbol \tau\boldsymbol n_{F_2}$ to 
\begin{equation}\label{eq:nnedge}
(\boldsymbol  n_{F_1}^{\intercal}\boldsymbol \tau\boldsymbol n_{F_2}, q)_e,\quad q\in \mathbb P_k(e), e\in \Delta_{d-2}(T).
\end{equation}
To switch from DoF~\eqref{eq:nnedge} to edge jump DoF~\eqref{eq:newdivdivdof1}, we require the following lemma. 

\begin{lemma}\label{lem:Snnbasis}
For a $(d-2)$-dimensional face $e\in\Delta_{d-2}(T)$, let $F_1$ and $F_2$ be the two $(d-1)$-dimensional faces in $\Delta_{d-1}(T)$ sharing $e$, and $\bs n_{F_i}=\bs n_{F_i,\partial T}$ for $i=1,2$. Then 
$$\{\boldsymbol{n}_{F_1}\otimes\boldsymbol{n}_{F_1}, \boldsymbol{n}_{F_2}\otimes\boldsymbol{n}_{F_2}, \sym(\boldsymbol{n}_{F_1,e}\otimes\boldsymbol{n}_{F_1})+\sym(\boldsymbol{n}_{F_2,e}\otimes\boldsymbol{n}_{F_2})\}$$ 
and 
$$\{\boldsymbol{n}_{F_1}\otimes\boldsymbol{n}_{F_1}, \boldsymbol{n}_{F_2}\otimes\boldsymbol{n}_{F_2}, \sym(\boldsymbol{n}_{F_1}\otimes\boldsymbol{n}_{F_2})\}$$ 
are bases of the symmetric matrix space $\mathbb S(\mathcal N_e)$ on the normal plane of $e$.
\end{lemma}
\begin{proof}
Clearly, $$
\mathbb S(\mathcal N_e) = \operatorname{span}\{\boldsymbol{n}_{F_1}\otimes\boldsymbol{n}_{F_1}, \boldsymbol{n}_{F_2}\otimes\boldsymbol{n}_{F_2}, \sym(\boldsymbol{n}_{F_1}\otimes\boldsymbol{n}_{F_2})\},
$$
and $\sym(\boldsymbol{n}_{F_1,e}\otimes\boldsymbol{n}_{F_1})+\sym(\boldsymbol{n}_{F_2,e}\otimes\boldsymbol{n}_{F_2})\in \mathbb S(\mathcal N_e)$. 

Now we prove that $\boldsymbol{n}_{F_1}\otimes\boldsymbol{n}_{F_1}$, $\boldsymbol{n}_{F_2}\otimes\boldsymbol{n}_{F_2}$ and $\sym(\boldsymbol{n}_{F_1,e}\otimes\boldsymbol{n}_{F_1})+\sym(\boldsymbol{n}_{F_2,e}\otimes\boldsymbol{n}_{F_2})$
are linearly independent. Assume constants $c_1, c_2$ and $c_3$ satisfy
$$
c_1\boldsymbol{n}_{F_1}\otimes\boldsymbol{n}_{F_1}+c_2\boldsymbol{n}_{F_2}\otimes\boldsymbol{n}_{F_2}+c_3\big(\sym(\boldsymbol{n}_{F_1,e}\otimes\boldsymbol{n}_{F_1})+\sym(\boldsymbol{n}_{F_2,e}\otimes\boldsymbol{n}_{F_2})\big)=0.
$$
Let us show that $c_1=c_2=c_3=0$. Multiplying $\sym(\boldsymbol{n}_{F_1,e}\otimes\boldsymbol{n}_{F_2,e})$ on both sides of the last equation, we get
$$
\frac{1}{2}c_3(\boldsymbol{n}_{F_1}\cdot\boldsymbol{n}_{F_2,e} + \boldsymbol{n}_{F_2}\cdot\boldsymbol{n}_{F_1,e})=0.
$$
Noting that both $\boldsymbol{n}_{F_1}\cdot\boldsymbol{n}_{F_2,e}$ and $\boldsymbol{n}_{F_2}\cdot\boldsymbol{n}_{F_1,e}$ are positive, we get $c_3=0$. And this implies
$$
c_1\boldsymbol{n}_{F_1}\otimes\boldsymbol{n}_{F_1}+c_2\boldsymbol{n}_{F_2}\otimes\boldsymbol{n}_{F_2}=0.
$$
Thus, $c_1=c_2=0$.
\end{proof}

We are in the position to prove the uni-solvence. Recall that in the binomial coefficient notation ${n\choose k}$, if $n\geq 0, k<0$, we set ${n\choose k} := 0$.
\begin{lemma}
For $k\geq 3$, the DoFs~\eqref{eq:newdivdivS} are unisolvent for the space $\mathbb P_k(T;\mathbb S)$.
\end{lemma}
\begin{proof}
For a $d$-simplex $T$, the number of sub-simplexes of dimension $r$ is ${d+1 \choose r+1}$. The dimension of $\mathbb P_{k-r-1}(f)$ with $\dim f = r$ is ${ k-r-1 + r \choose k - r - 1}$ which also holds for $r\geq k$ as $\dim \mathbb P_{k-r-1}(f) =0$. The normal plane $\mathcal N_f$ of $f$, will have dimension $d - r$ and the symmetric tensor on $\mathcal N_f$ will have dimension ${ d -r +1 \choose 2}$ which can be split into off-diagonals and diagonal, i.e., ${d-r+1\choose2}={d-r\choose2}+d-r$.
The number of DoFs~\eqref{HdivdivSfemdof1}-\eqref{HdivdivSfemdof2} is
\begin{align}
\notag &\quad\sum_{r=0}^{d-1}{d+1\choose r+1}{k-1\choose k-r-1}{d-r+1\choose2} \\
\label{eq:redistribution} &=\frac{1}{2}d(d+1)\sum_{r=0}^{d-2}{d-1\choose r+1}{k-1\choose k-r-1}+(d+1)\sum_{r=0}^{d-1}{d\choose r+1}{k-1\choose k-r-1} \\
\label{eq:merge} &=\frac{1}{2}d(d+1){k+d-2\choose k}+(d+1){k+d-1\choose k},
\end{align}
which equals the number of DoFs~\eqref{eq:newdivdivdof1}-\eqref{eq:newdivdivdof2}.
Hence the number of DoFs~\eqref{eq:newdivdivS} matches the number of DoFs~\eqref{eq:divdivS} which is the dimension of the space $\mathbb P_k(T;\mathbb S)$ by Lemma \ref{lm:divdivS}. In the derivation above,~\eqref{eq:redistribution} corresponds to the redistribution of DoFs to edges and faces, and~\eqref{eq:merge} is the merge of DoFs for the Lagrange element on edges and faces.

Let $\boldsymbol{\tau} \in \mathbb{P}_k(T;\mathbb{S})$ and suppose that all DoFs given by~\eqref{eq:newdivdivS} vanish. Using Lemma~\ref{lem:Snnbasis}, we know that the vanishing DoFs~\eqref{eq:newdivdivdof1}-\eqref{eq:newdivdivdof2} imply that DoF~\eqref{eq:nnedge} also vanishes. Moreover, the vanishing of~\eqref{eq:nnedge} and~\eqref{eq:newdivdivdof2} implies that DoFs~\eqref{HdivdivSfemdof1}-\eqref{HdivdivSfemdof2} are also zero. Therefore, by the unisolvence property stated in Lemma \ref{lm:divdivS}, we conclude that $\mathbb{P}_k(T;\mathbb{S})$ is unisolvent.
\end{proof}

Define the global space 
\begin{align*}
\Sigma_{k}^{\div\div-} := \{\boldsymbol{\tau}\in L^2(\Omega;\mathbb S): &\, \boldsymbol{\tau}|_T\in \mathbb P_{k}(T;\mathbb S)\textrm{ for each } T\in\mathcal T_h, \\
&\quad\textrm{  DoFs~\eqref{eq:newdivdivdof2} and~\eqref{eq:newdivdivdof3} are single-valued} \}.
\end{align*}
By construction, for $\bs \tau\in \Sigma_{k}^{\div\div-}$, both $\tr_1(\bs \tau)$ and $\tr_2(\bs \tau)$ will be continuous. But the edge jump $[\tr_e(\bs \tau)]|_e$ may not vanish which prevents $\Sigma_{k}^{\div\div-}$ being $H(\div\div)$-conforming in view of Lemma \ref{lm:divdivconforming}. The edge jump condition $[\tr_e(\bs \tau)]|_e = 0$ is imposed patch-wisely on $\omega_e$. Inside each element, $\tr_e(\bs \tau)$ may not be zero and for different elements the edge jump are in general different. Therefore~\eqref{eq:newdivdivdof1} is not single-valued when defining $\Sigma_{k}^{\div\div-}$.

Define the subspace
$$
\Sigma_{k,{\rm new}}^{\div\div} := \{\bs \tau\in \Sigma_{k}^{\div\div-}: [\tr_e(\bs \tau)]|_e = 0 \text{ for all } e\in \mathring{\mathcal E}_h\}.
$$
That is we add constraints on the DoFs of the element-wise edge traces: $\tr_e^{T_1}(\bs \tau)+ \tr_e^{T_2}(\bs \tau) + \cdots + \tr_e^{T_{|\omega_e|}}(\bs \tau) |_e = 0$ to get an $H(\div\div)$-conforming subspace.


Let $I_h^{\div\div}: H^2(\Omega;\mathbb S)\to \Sigma_{k}^{\div\div-}$ be the cannocial interpolation operator based on the DoFs~\eqref{eq:newdivdivS}. Namely $N(I_h^{\div\div}\bs \tau) = N(\bs \tau)$ for all DoFs $N$ in~\eqref{eq:newdivdivS}. To save notation, we will abbreviate $I_h^{\div\div}\bs \tau$  as $\bs \tau_I$. 
Noting that
$$
[\tr_e(\bs \tau_I)]|_e=Q_{k,e}([\tr_e(\bs \tau)]|_e)=0\quad\forall~e\in\mathring{\mathcal{E}}_h, \bs\tau\in H^2(\Omega;\mathbb S), 
$$
so indeed $\bs \tau_I \in \Sigma_{k,{\rm new}}^{\div\div}$. 
\begin{lemma}
 $I_h^{\div\div}$ is a Fortin operator in the sense that: for $\bs\tau\in H^2(\Omega;\mathbb S)$,
\begin{equation}\label{eq:IdivdivQcd}
\div\div(\bs\tau_I)=Q_{k-2}(\div\div\bs\tau).
\end{equation}
\end{lemma}
\begin{proof}
It can be proved by using the Green's identity~\eqref{eq:greenidentitydivdiv} and the definition of  $I_h^{\div\div}$. 
\end{proof}

Using the Fortin operator, we arrive at the following inf-sup condition.
\begin{lemma}
We have the inf-sup condition
\begin{equation}\label{eq:newdivdivinfsup}
\inf_{p_h\in V^{-1}_{k-2}}\sup_{\boldsymbol\tau_h\in \Sigma_{k,{\rm new}}^{\div\div}}
\frac{(\div\div\boldsymbol{\tau}_h, p_h)}{\|\boldsymbol\tau_h\|_{\div\div} \|p_h\|_0} = \alpha > 0, \quad \text{for } k\geq 3.
\end{equation}
\end{lemma}
\begin{proof}
For $p_h\in V^{-1}_{k-2}$, there exists a function $\bs\tau\in H^2(\Omega;\mathbb S)$~\cite{Arnold;Hu:2020Complexes,PaulyZulehner2020} such that 
$$
\|\bs\tau\|_2\lesssim \|p_h\|_0,\quad \div\div\bs\tau=p_h.
$$
Let $\bs\tau_h=\bs\tau_I \in\boldsymbol\Sigma_{k,{\rm new}}^{\div\div}$. By~\eqref{eq:IdivdivQcd},
$$
\div\div\bs\tau_h=Q_{k-2}(\div\div\bs\tau)=p_h.
$$
Apply the scaling argument to get
$$
\|\bs\tau_h\|_{\div\div}\lesssim \|\bs\tau\|_2\lesssim \|p_h\|_0.
$$
Finally, we finish the proof of~\eqref{eq:newdivdivinfsup}.
\end{proof}

Comparing with the existing $H(\div\div)$-conforming elements constructed in~\cite{ChenHuang2020,Chen;Huang:2020Finite,ChenHuangDivRn2022,Hu;Liang;Ma:2021Finite,Hu;Liang;Ma;Zhang:2022conforming,Hu;Ma;Zhang:2020family}, we do not enforce  the normal plane continuity on lower dimensional sub-simplexes and thus no requirement $k\geq d$ for the inf-sup condition. However, the condition $k\geq 3$ is still needed to ensure ${\bf RM}=\ker(\defm)\subseteq{\rm ND}_{k-3}(T)$ in DoF~\eqref{eq:newdivdivdof5}. The following remark shows $\mathbb P_{k}(T;\mathbb S)$, $k\leq 2$, is not feasible.

\begin{remark}\label{rm:k=2}\rm
For a linear polynomial $v\in \mathbb P_1(T)$, by identity~\eqref{eq:greenidentitydivdiv} and the fact $
\nabla^2v=0$, we have for $\boldsymbol\tau\in\mathbb P_k(T; \mathbb S)$ that
\begin{equation}\label{eq:greenidentityP1}
(\div\div\boldsymbol\tau, v)_T=\!\sum_{F\in\partial T}\left[(\tr_2(\boldsymbol{\tau}), v)_F-(\boldsymbol  n^{\intercal}\boldsymbol\tau\boldsymbol  n, \partial_n v)_{F}\right]-\!\sum_{F\in\partial T}\sum_{e\in\partial F}(\boldsymbol n_{F,e}^{\intercal}\boldsymbol\tau \boldsymbol n, v)_e. 
\end{equation}
When $k\leq 2$,  for $\boldsymbol\tau\in\mathbb P_k(T; \mathbb S)$, $\div\div\boldsymbol\tau\in\mathbb P_0(T)$. We can choose a nonzero function $v\in\mathbb P_1(T)\cap L_0^2(T)$ such that $(\div\div\boldsymbol\tau, v)_T=0$, hence it follows
$$
\sum_{F\in\partial T}\left[(\tr_2(\boldsymbol{\tau}), v)_F-(\boldsymbol  n^{\intercal}\boldsymbol\tau\boldsymbol  n, \partial_n v)_{F}\right]-\sum_{F\in\partial T}\sum_{e\in\partial F}(\boldsymbol n_{F,e}^{\intercal}\boldsymbol\tau \boldsymbol n, v)_e=0. 
$$
This means the DoFs~\eqref{eq:newdivdivdof1}-\eqref{eq:newdivdivdof3} for traces are not linearly independent when $k\leq 2$. The range of the $\div\div$ operator should contain $\mathbb P_1(T)$ piecewisely.  $\qed$
\end{remark}


\subsection{Raviart-Thomas type elements}\label{sec:divdivconformingRT}
We enrich the range of the $\div\div$ operator with the addition of high order inner moments. 
Take the space of shape functions as 
$$
\Sigma_{k^+}(T;\mathbb S):= \mathbb P_k(T;\mathbb S) \oplus \boldsymbol x\boldsymbol x^{\intercal}\mathbb H_{k-1}(T),\quad k\geq 2.
$$
The additional component $\boldsymbol x\boldsymbol x^{\intercal}\mathbb H_{k-1}(T)$ expands the range of the $\div\div$ operator to $\mathbb P_{k-1}(T)$ as $\div\div(\boldsymbol x\boldsymbol x^{\intercal}\mathbb H_{k-1}(T)) = \mathbb H_{k-1}(T)$, which is one degree higher than the range $\div\div \mathbb P_k(T;\mathbb S) = \mathbb P_{k-2}(T)$.

For $k\geq 3$, the degrees of freedom are nearly identical to those given in~\eqref{eq:newdivdivS}, with the exception of enriching the DoF in~\eqref{eq:newdivdivdof5} to
\begin{equation}
(\boldsymbol \tau, \defm\boldsymbol q)_T \quad \text{for} \quad \boldsymbol q\in\mathbb P_{k-2}(T;\mathbb R^d).\label{eq:newdivdivdofRT5}
\end{equation}
%
 The degree of freedom $(\boldsymbol \tau, \defm \boldsymbol q)_T $ is increased from $\boldsymbol q \in  {\rm ND}_{k-3}(T)=\grad \mathbb P_{k-2}(T)\oplus\mathbb P_{k-3}(T;\mathbb K)\boldsymbol x$ in~\eqref{eq:newdivdivdof5} to $\mathbb P_{k-2}(T;\mathbb R^d)=\grad \mathbb P_{k-1}(T)\oplus\mathbb P_{k-3}(T;\mathbb K)\boldsymbol x$. All boundary DoFs~\eqref{eq:newdivdivdof1}-\eqref{eq:newdivdivdof4} remain the same as $(\boldsymbol x\boldsymbol x^{\intercal}\mathbb H_{k-1}(T))\boldsymbol n|_F\in \mathbb P_{k}(F;\mathbb R^d)$. 

For $k=2$, $\ker(\defm)={\rm ND}_{0}(K)\not\subseteq\mathbb P_{0}(T;\mathbb R^d)$ in DoF~\eqref{eq:newdivdivdofRT5}. 
We propose the following DoFs for $\Sigma_{2^+}(T;\mathbb S)$ which is a generalization of the $H(\div\div)$-conforming finite element constructed in~\cite{Chen;Huang:2020Finite} by the redistribution process:
\begin{subequations}\label{eq:newdivdivSk2RT}
\begin{align}
(\tr_e(\bs \tau), q)_e, &\quad q\in \mathbb P_2(e), e\in \Delta_{d-2}(T),\label{eq:newdivdivdofk2RT1}\\
(\bs n^{\intercal}\bs \tau \bs n, q )_F, &\quad q\in \mathbb P_2(F), F\in \partial T,\label{eq:newdivdivdofk2RT2}\\
( \tr_2(\bs \tau), q)_F, &\quad q\in\mathbb P_{1}(F), F\in \partial T,\label{eq:newdivdivdofk2RT3}\\
(\Pi_f\boldsymbol \tau\boldsymbol n_{F_r}, \boldsymbol q)_{f}, & \quad \boldsymbol q\in \mathbb B_2^{\div}(f), f= f_{0:r-2}\in \Delta_{r-2}(F_r), r= d,\ldots, 3,
\label{eq:newdivdivdofk2RT4}\\
(\boldsymbol \tau, \boldsymbol q)_T, &\quad \boldsymbol q\in  \ker(\boldsymbol x^{\intercal}\cdot\boldsymbol x)\cap \mathbb P_{1}(T;\mathbb S),\label{eq:newdivdivdofk2RT5}
\end{align}
\end{subequations}
where $f_{0:r}={\rm Convex}(\texttt{v}_0,\texttt{v}_1,\ldots, \texttt{v}_{r})$ is the $r$-dimensional simplex spanned by the vertices $\{\texttt{v}_0,\texttt{v}_1,\ldots, \texttt{v}_{r}\}$. A proof of the unisolvence can be found in Appendix~\ref{apdx:uinsol} (Theorem~\ref{th:k=2}).  


Define the global spaces, for $k\geq 2$, 
\begin{align*}
\Sigma_{k^+}^{\div\div}  := & 
\{ \boldsymbol{\tau}\in L^2(\Omega;\mathbb S): \, \boldsymbol{\tau}|_T\in \Sigma_{k^+}(T;\mathbb S) \textrm{ for each } T\in\mathcal T_h,\\
&\quad \textrm{  DoFs~\eqref{eq:newdivdivdof2} and~\eqref{eq:newdivdivdof3} are single-valued}, [\tr_e(\bs \tau)]|_e = 0 \text{ for all } e\in \mathring{\mathcal E}_h \}.
\end{align*}
We have $\Sigma_{k^+}^{\div\div}\subset H(\div\div,\Omega;\mathbb S)$.


Similar to the proof of~\eqref{eq:newdivdivinfsup} by using the canonical interpolation operator $I_h^{\div\div}$, we have the inf-sup condition
\begin{equation}\label{eq:newdivdivRTinfsup}
\inf_{p_h\in V^{-1}_{k-1}}\sup_{\boldsymbol\tau_h\in \Sigma_{k^+}^{\div\div}}
\frac{(\div\div\boldsymbol{\tau}_h, p_h)}{\|\boldsymbol\tau_h\|_{\div\div} \|p_h\|_0} = \alpha > 0, \quad \text{for } k\geq 2.
\end{equation}

\subsection{A lower order $H(\div\div)$-conforming finite element}\label{sec:lowerdivdivconforming}
For $k=1$, we enrich the $\mathbb P_1(T;\mathbb S)$ space by adding some quadratic and cubic polynomials. Take the shape function space as 
\begin{align}    
\Sigma_{1^{++}}(T;\mathbb S)
\label{eq:sigma1+} &\;= \mathbb P_1(T;\mathbb S)\oplus \sym(\boldsymbol{x}\otimes \mathbb H_1(T;\mathbb R^d)) \oplus \boldsymbol x\boldsymbol x^{\intercal}\mathbb H_{1}(T).
\end{align}
The range $\div\div(\boldsymbol x\boldsymbol x^{\intercal}\mathbb H_{1}(T)) = \mathbb H_{1}(T)$ and $\div\div \sym(\boldsymbol{x}\otimes \mathbb H_1(T;\mathbb R^d)) = \mathbb P_0(T)$. Consequently $\div\div \Sigma_{1^{++}}(T;\mathbb S) = \mathbb P_1(T)$. 

When $\bs\tau\in \Sigma_{1^{++}}(T;\mathbb S)$, we can see that $\tr_e(\bs \tau)\in\mathbb P_1(e)$ for $e\in \Delta_{d-2}(T)$, and $(\bs n^{\intercal}\bs \tau \bs n)|_F, \tr_2(\bs \tau)|_F\in\mathbb P_{1}(F)$ for $F\in \partial T$. Hence, we propose the following DoFs:
\begin{subequations}\label{eq:newdivdivSreduce}
\begin{align}
(\tr_e(\bs \tau), q)_e, &\quad q\in \mathbb P_1(e), e\in \Delta_{d-2}(T),\label{eq:newdivdivSreduce1}\\
(\bs n^{\intercal}\bs \tau \bs n, q )_F, &\quad q\in \mathbb P_1(F), F\in \partial T,\label{eq:newdivdivSreduce2}\\
( \tr_2(\bs \tau), q)_F, &\quad q\in\mathbb P_{1}(F), F\in \partial T.\label{eq:newdivdivSreduce3}
\end{align}
\end{subequations}

\begin{lemma}\label{lm:1++}
The DoFs~\eqref{eq:newdivdivSreduce} are unisolvent for the space $\Sigma_{1^{++}}(T;\mathbb S)$.
\end{lemma}
\begin{proof}
DoFs~\eqref{eq:newdivdivSreduce1}-\eqref{eq:newdivdivSreduce2} are redistribution of vertex DoFs for $\mathbb P_1(T; \mathbb S)$. The enrichment in~\eqref{eq:sigma1+} has dimension $d^2 + d$ while the number of DoF~\eqref{eq:newdivdivSreduce3} is $(d+1)d$. Therefore
the number of DoF~\eqref{eq:newdivdivSreduce} is equal to $\dim\Sigma_{1^{++}}(T;\mathbb S) = \frac{1}{2}d(d+1)(d+3).$

Assume $\boldsymbol{\tau}\in\Sigma_{1^{++}}(T;\mathbb S)$, and all the DoFs~\eqref{eq:newdivdivSreduce} vanish. Then
\begin{equation}\label{eq:vanishingtrace}    
\tr_1(\bs \tau)=0,\quad \tr_2(\bs \tau)=0,\quad Q_{\mathcal N_e}(\bs \tau)=0 \textrm{ for } e\in\Delta_{d-2}(T).
\end{equation}
Apply the integration by parts to get $\div\div\bs \tau=0$. Consequently $\bs \tau \in \mathbb P_1(T;\mathbb S)+ \sym(\boldsymbol{x}\otimes \mathbb P_1(T;\mathbb R^d))$. 

Let $\bs \tau=\bs\tau_1+\sym(\bs x\otimes\bs q)$ with $\bs\tau_1\in\mathbb P_1(T;\mathbb S)$ and $\bs q\in \mathbb H_1(T;\mathbb R^d)$. Then $\div\bs q=0$ follows from $\div\div\bs \tau=0$.
Since $\tr_2(\bs \tau_1)$ is piecewise constant, the fact $\tr_2(\bs \tau)=0$ in~\eqref{eq:vanishingtrace} means $\tr_2(\sym(\bs x\otimes\bs q))|_F\in\mathbb P_0(F)$ for face $F\in\partial T$. By $\div(\bs x\bs q^{\intercal})=\bs q+\bs x\div\bs q$, $\div(\bs q\bs x^{\intercal})=(d+1)\bs q$, and $\div_F(\bs x\bs q\cdot\bs n)=d\bs q\cdot\bs n$, we get
$$
\tr_2(\sym(\bs x\otimes\bs q))|_F=(d+1)\bs q\cdot\bs n+\frac{1}{2}\bs x\cdot\bs n(\div\bs q +\div_F\bs q)\in\mathbb P_0(F).
$$
This indicates $(\bs q\cdot\bs n)|_F\in\mathbb P_0(F)$, which means $\bs q\in\mathbf{RT}$. By $\bs q\in \mathbb H_1(T;\mathbb R^d)\cap\ker(\div)$, $\bs q=0$. Now $\bs\tau\in\mathbb P_1(T;\mathbb S)$. 
The third identity in~\eqref{eq:vanishingtrace} implies $\bs\tau$ vanishes on all the vertices of $T$, therefore $\bs\tau=0$.
\end{proof} 

Define the global space
\begin{align*}
\Sigma_{1^{++}}^{\div\div}  := & 
\{ \boldsymbol{\tau}\in L^2(\Omega;\mathbb S): \, \boldsymbol{\tau}|_T\in \Sigma_{1^{++}}(T;\mathbb S) \textrm{ for each } T\in\mathcal T_h,\\
&\; \textrm{  DoFs~\eqref{eq:newdivdivSreduce2} and~\eqref{eq:newdivdivSreduce3} are single-valued}, [\tr_e(\bs \tau)]|_e = 0 \text{ for all } e\in \mathring{\mathcal E}_h\}.
\end{align*}
We have $\Sigma_{1^{++}}^{\div\div}\subset H(\div\div,\Omega;\mathbb S)$.
Again using the canonical interpolation operator $I_h^{\div\div}$, it holds the inf-sup condition
\begin{equation}\label{eq:newdivdivk1infsup}
\inf_{p_h\in V^{-1}_{1}}\sup_{\boldsymbol\tau_h\in \Sigma_{1^{++}}^{\div\div}}
\frac{(\div\div\boldsymbol{\tau}_h, p_h)}{\|\boldsymbol\tau_h\|_{\div\div} \|p_h\|_0} = \alpha > 0.
\end{equation}
In two dimensions, i.e., $d=2$, the finite element space $\Sigma_{1^{++}}^{\div\div}$ has been constructed in~\cite{FuehrerHeuer2023}. Our construction of $\Sigma_{1^{++}}^{\div\div}$ for general $d\geq 2$ is motivated by their work.

\section{A mixed method for the biharmonic equation}\label{sec:mfem}
This section will discuss a mixed finite element method for solving the biharmonic equation. Optimal convergence rates are obtained. Post-processing techniques will be introduced to further improve the accuracy of the solution. 

\subsection{Mixed methods for the biharmonic equation}
Let $f\in L^2(\Omega)$ be given. Consider the biharmonic equation
\begin{equation}\label{eq:biharmonic}
\begin{cases}
\Delta^2u=f\quad \textrm{in }\Omega, \\
u|_{\partial\Omega}=\partial_{n}u|_{\partial\Omega}=0.
\end{cases}
\end{equation}
The mixed formulation is: find $\bs \sigma \in H(\div\div,\Omega;\mathbb S), u\in L^2(\Omega)$ s.t.
\begin{subequations}\label{eq:biharmonicmixed}
 \begin{align}
(\bs \sigma, \boldsymbol{\tau}) + (\div\div\boldsymbol{\tau}, u)&=0  \qquad\quad\quad\;\; \forall~\boldsymbol{\tau} \in H(\div\div,\Omega;\mathbb S),\\
(\div\div\bs \sigma, v) & = -(f, v) \quad \quad  \forall~v\in L^2(\Omega).
\end{align}
\end{subequations}
Notice that the Dirichlet boundary condition $u|_{\partial\Omega}=\partial_{n}u|_{\partial\Omega}=0$ is built naturally into the weak formulation. 

We will use either the pair $\Sigma_{k,{\rm new}}^{\div\div} - V^{-1}_{k-2}$ or $\Sigma_{k^+}^{\div\div} -V^{-1}_{k-1}$, and unify the notation as 
$$
\Sigma_{k,r}^{\div\div} - V^{-1}_{r} := 
\begin{cases}
\Sigma_{k,{\rm new}}^{\div\div} - V^{-1}_{k-2},  & r = k -2, k\geq 3,\\
\Sigma_{k^+}^{\div\div} -V^{-1}_{k-1}, & r = k -1, k\geq 2,\\
\Sigma_{1^{++}}^{\div\div} - V^{-1}_{1}, & r = k =1.
\end{cases}
$$

A mixed finite element method for biharmonic equation~\eqref{eq:biharmonic} is to find $(\bs \sigma_h, u_h)\in \Sigma_{k,r}^{\div\div}\times V^{-1}_{r}$ with $r\geq 1$, s.t. 
\begin{subequations}\label{eq:biharmonicMfem}
 \begin{align}
\label{eq:biharmonicMfem1}
(\bs \sigma_h, \boldsymbol{\tau}) + (\div\div\boldsymbol{\tau}, u_h)&=0  \qquad\quad\quad\;\; \forall~\boldsymbol{\tau} \in \Sigma_{k,r}^{\div\div},\\
\label{eq:biharmonicMfem2}
(\div\div\bs \sigma_h, v) & = -(f, v) \quad \quad  \forall~v\in V^{-1}_{r}.
\end{align}
\end{subequations}
The mixed method~\eqref{eq:biharmonicMfem} is well-posed due to the discrete inf-sup conditions~\eqref{eq:newdivdivinfsup},~\eqref{eq:newdivdivRTinfsup} and~\eqref{eq:newdivdivk1infsup}. By the standard procedure, we have the following error estimates.

\begin{lemma}
Let $u\in H_0^2(\Omega)$ be the solution of biharmonic equation~\eqref{eq:biharmonic} and $\bs \sigma=-\nabla^2u$. 
Let $(\bs \sigma_h, u_h)\in \Sigma_{k,r}^{\div\div} \times V^{-1}_{r}$ be the solution of the mixed method~\eqref{eq:biharmonicMfem} for $r \geq 1$. 
Assume $u\in H^{k+3}(\Omega)$ and $f\in H^{r+1}(\Omega)$. We have
\begin{align}
\notag
\| \div\div (\bs \sigma - \bs \sigma_h)\|_0 & \lesssim h^{r+1}\|f\|_{r+1},\\
\label{biharmonicMfemErr1}
\| Q_r u - u_h\|_0 + \| \bs \sigma - \bs \sigma_h \|_0 & \lesssim h^{k+1}\|\bs \sigma\|_{k+1},\\
\label{biharmonicMfemErr2}
\| u - u_h \|_0 & \lesssim h^{r+1}\| u\|_{r+1}.
\end{align}
\end{lemma}
\begin{proof}
By~\eqref{eq:biharmonicMfem2},
$$
\| \div\div (\bs \sigma - \bs \sigma_h)\|_0  = \| f - Q_r f \|_0 \lesssim h^{r+1}\|f\|_{r+1}.
$$
From~\eqref{eq:biharmonicmixed} and~\eqref{eq:biharmonicMfem}, we have the error equation
\begin{equation}\label{erroreqn}
(\bs \sigma - \bs \sigma_h, \bs \tau) + (\div\div \bs \tau, Q_r u- u_h) = 0 \quad \forall~\bs \tau \in  \Sigma_{k,r}^{\div\div}. 
\end{equation}
Taking $\bs \tau = \bs \sigma_h - \bs \sigma_I$ and noticing $\div\div (\bs \sigma_h - \bs \sigma_I) = 0$, we obtain the partial orthogonality $(\bs \sigma - \bs \sigma_h,  \bs \sigma_h - \bs \sigma_I) = 0$ and thus 
$$
\| \bs \sigma - \bs \sigma_h \|_0 \leq \| \bs \sigma - \bs \sigma_I \|_0 \lesssim h^{k+1}\|\bs \sigma\|_{k+1}.
$$
By the inf-sup condition, we can find $\bs \tau\in\Sigma_{k,r}^{\div\div}$ s.t. $\div\div\bs \tau = Q_r u - u_h$ and obtain the estimate for $\|Q_r u- u_h\|_0$ by the Cauchy-Schwarz inequality. 

Estimate~\eqref{biharmonicMfemErr2} can be obtained by the triangle inequality and standard error estimate of the $L^2$ projection $\| u - Q_r u\|_0$.
\end{proof}
%
%

Observing that when the parameter $r$ satisfies $r=k-1$ or $r=k-2$, the error estimate~\eqref{biharmonicMfemErr1} exhibits one or two orders of convergence higher than that of~\eqref{biharmonicMfemErr2}. It is expected that a refined interior approximation of higher accuracy than $u_h$ can be obtained via post-processing techniques.



\subsection{Postprocessing}

Following the postprocessing in~\cite{ChenHuang2020} rather than those in~\cite{Stenberg1991,Comodi1989}, we will construct a new superconvergent approximation to deflection $u$ by using the optimal estimate of $\| \bs \sigma-\bs \sigma_h\|_0$ and the superconvergent estimate of $\|Q_r u - u_h\|_0$ in~\eqref{biharmonicMfemErr1}. 

Define a new approximation $u_h^{\ast}\in V_{k+2}^{-1}$ to $u$ elementwisely as a solution of the following problem: for any $T\in\mathcal{T}_h$,
\begin{subequations}\label{eq:postprocess}
 \begin{align}
(\nabla^2u_h^{\ast}, \nabla^2v)_T&=-(\boldsymbol{\sigma}_h, \nabla^2v)_T \quad\forall~v\in \mathbb P_{k+2}(T),\label{eq:postprocess1}  \\
(u_h^{\ast}, v)_T&=(u_h, v)_T \qquad\quad\,\forall~v\in \mathbb P_{1}(T).\label{eq:postprocess2}
\end{align}
\end{subequations}

\begin{theorem}\label{thm:uuhstar1}
Let $u\in H_0^2(\Omega)$ be the solution of biharmonic equation~\eqref{eq:biharmonic} and $\boldsymbol{\sigma}=-\nabla^2u$. 
Let $u_h^{\ast}\in V_{k+2}^{-1}$ be the solution of~\eqref{eq:postprocess} for $r \geq 1$.
Assume $u\in H^{k+3}(\Omega)$. We have
\begin{equation*}
\|u-u_h^{\ast}\|_{0} + \|\nabla_{h}^2(u-u_h^{\ast})\|_{0} \lesssim h^{k+1}|u|_{k+3}.
\end{equation*}
\end{theorem}
\begin{proof}
For simplicity, let $z\in V_{k+2}^{-1}(\mathcal T_h)$ be defined by $z|_T=(I-Q_{1,T})(Q_{k+2,T}u-u_h^{\ast})$. Since $Q_{1,T}z=0$, we have
\begin{equation}\label{eq:zequiv}
\|z\|_{0, T}\eqsim h_T|z|_{1,T}\eqsim h_T^2|z|_{2,T}.
\end{equation}
Take $v=z|_T$ in~\eqref{eq:postprocess1} to obtain
\[
(\nabla^2(u-u_h^{\ast}), \nabla^2z)_T=-(\boldsymbol{\sigma}-\boldsymbol{\sigma}_h, \nabla^2z)_T.
\]
Noting the definition of $z$, we have
\begin{equation*}
|z|_{2,T}^2=(\nabla^2(Q_{k+2,T}u-u_h^{\ast}), \nabla^2z)_T=(\nabla^2(Q_{k+2,T}u-u), \nabla^2z)_T - (\boldsymbol{\sigma}-\boldsymbol{\sigma}_h, \nabla^2z)_T,
\end{equation*}
which implies
\begin{equation}\label{eq:zH2}
|Q_{k+2,T}u-u_h^{\ast}|_{2,T}=|z|_{2,T}\lesssim |u-Q_{k+2,T}u|_{2,T} + \|\boldsymbol{\sigma}-\boldsymbol{\sigma}_h\|_{0,T}.
\end{equation}
Hence $|u-u_h^{\ast}|_{2,h}\lesssim h^{k+1}|u|_{k+3}$ follows from the triangle inequality, the estimate of $Q_{k+2,T}$, and error estimate~\eqref{biharmonicMfemErr1}.

On the other side,
by~\eqref{eq:postprocess2}, we have
\begin{equation*}
\|Q_{1,T}(Q_{k+2,T}u-u_h^{\ast})\|_{0,T}=\|Q_{1,T}(Q_{r,T}u-u_h)\|_{0,T}\leq \|Q_{r,T}u-u_h\|_{0,T},
\end{equation*}
which together with~\eqref{eq:zequiv} yields
\begin{align*}
\|Q_{k+2,T}u-u_h^{\ast}\|_{0,T}&\leq \|Q_{1,T}(Q_{k+2,T}u-u_h^{\ast})\|_{0,T}+\|z\|_{0,T} \\
&\lesssim \|Q_{r,T}u-u_h\|_{0,T} + h_T^2|z|_{2,T}.
\end{align*}
By the triangle inequality and~\eqref{eq:zH2},
\begin{align}    
\label{eq:u-uhast}\|u-u_h^{\ast}\|_{0,T}&\lesssim \|u-Q_{k+2,T}u\|_{0,T}+\|Q_{r,T}u-u_h\|_{0,T} \\
&\quad+ h_T^2(|u-Q_{k+2,T}u|_{2,T} + \|\boldsymbol{\sigma}-\boldsymbol{\sigma}_h\|_{0,T}). \notag
\end{align}
Hence, $\|u-u_h^{\ast}\|_{0}\lesssim h^{k+1}|u|_{k+3}$ follows from the estimate of $Q_{k+2,T}$ and~\eqref{biharmonicMfemErr1}.
\end{proof}

\subsection{Duality argument}
To further enhance the convergence rate of $\|Q_r u - u_h\|_{0}$ and achieve a superconvergent $L^2$-error estimate for the post-processed approximation, we employ a duality argument. Consider the biharmonic equation
\begin{equation*}
\begin{cases}
\Delta^2\widetilde{u}=Q_r u - u_h\quad \textrm{in }\Omega, \\
\widetilde{u}|_{\partial\Omega}=\partial_{n}\widetilde{u}|_{\partial\Omega}=0.
\end{cases}
\end{equation*}
Let $\widetilde{\boldsymbol{\sigma}}:=-\nabla^2\widetilde{u}$. We assume that $\widetilde{u}\in H^4(\Omega)\cap H_0^2(\Omega)$ and the bound
\begin{equation}\label{dualregularH4}
\|\widetilde{\boldsymbol{\sigma}}\|_2 + \|\widetilde{u}\|_4\lesssim \|Q_r u - u_h\|_0.
\end{equation}
In two dimensions, 
when $\Omega$ is  a bounded polygonal domain with all the inner anlges smaller than $126.383696^{\circ}$, the regularity estimate~\eqref{dualregularH4} holds~\cite[Theorem 2]{BlumRannacher1980}.

\begin{theorem}\label{thm:uuhstarL2}
Let $u\in H_0^2(\Omega)$ be the solution of biharmonic equation~\eqref{eq:biharmonic} and $\bs \sigma=-\nabla^2u$. 
Let $(\bs \sigma_h, u_h)\in \Sigma_{k,r}^{\div\div} \times V^{-1}_{r}$ be the solution of the mixed method~\eqref{eq:biharmonicMfem} for $r \geq 1$.
Let $u_h^{\ast}$ be obtained by the post-processing~\eqref{eq:postprocess} using $\bs \sigma_h$ and $u_h$. 
Assume $u\in H^{k+3}(\Omega)$, $f\in H^{r+1}(\Omega)$ and the regularity estimate~\eqref{dualregularH4} holds. We have
\begin{equation*}
\|Q_r u - u_h\|_{0} + \|u-u_h^{\ast}\|_{0} \lesssim h^{k+3}\|u\|_{k+3} + h^{\min\{2r+2, r+5\}}\|f\|_{r+1}.
\end{equation*}
\end{theorem}
\begin{proof}
Set $v=Q_r u - u_h$ for simplicity. By~\eqref{eq:IdivdivQcd},~\eqref{erroreqn} and integration by parts,
\begin{align*}
\|Q_r u - u_h\|_0^2&=-(\div\div\widetilde{\boldsymbol{\sigma}},v)=-(\div\div\widetilde{\boldsymbol{\sigma}}_I, v)=(\boldsymbol{\sigma}-\boldsymbol{\sigma}_h, \widetilde{\boldsymbol{\sigma}}_I) \\
&=(\boldsymbol{\sigma}-\boldsymbol{\sigma}_h, \widetilde{\boldsymbol{\sigma}}_I-\widetilde{\boldsymbol{\sigma}})-(\boldsymbol{\sigma}-\boldsymbol{\sigma}_h, \nabla^2\widetilde{u}) \\
&=(\boldsymbol{\sigma}-\boldsymbol{\sigma}_h, \widetilde{\boldsymbol{\sigma}}_I-\widetilde{\boldsymbol{\sigma}})-(\div\div(\boldsymbol{\sigma}-\boldsymbol{\sigma}_h), \widetilde{u}) \\
&=(\boldsymbol{\sigma}-\boldsymbol{\sigma}_h, \widetilde{\boldsymbol{\sigma}}_I-\widetilde{\boldsymbol{\sigma}})+(f-Q_rf, \widetilde{u}-Q_r\widetilde{u}).
\end{align*}
Apply the Cauchy-Schwarz inequality and interpolation error estimate to get
$$
\|Q_r u - u_h\|_0^2\lesssim h^2\|\boldsymbol{\sigma}-\boldsymbol{\sigma}_h\|_0|\widetilde{\boldsymbol{\sigma}}|_2 + h^{\min\{r+1,4\}}\|f-Q_rf\|_0\|\widetilde{u}\|_4.
$$
Thus the bound on $\|Q_r u - u_h\|_{0}$ follows from the regularity estimate~\eqref{dualregularH4}, and the bound on $\|u-u_h^{\ast}\|_{0}$ follows from~\eqref{eq:u-uhast} and~\eqref{biharmonicMfemErr1}.
\end{proof}


\section{Hybridization}\label{sec:hybridization}
This section will discuss the hybridization of the mixed finite element method~\eqref{eq:biharmonicMfem}. Spaces of Lagrange multipliers are introduced to relax the continuity of $\tr_1(\bs \tau)$, $\tr_2(\bs \tau)$, and the patch constraint imposed on edge jumps. Weak divdiv stability will be proved. 
Equivalent weak Galerkin and non-conforming virtual element methods formulation will also be provided, as well as a $C^0$ discontinuous Galerkin (CDG) method. 

\subsection{Broken spaces and weak differential operators}
For $k\geq 0$, define $$\Sigma_{k,r}^{-1} :=\prod_{T\in \mathcal T_h}\Sigma_{k,r}(T;\mathbb S)$$ with
$$
\Sigma_{k,r}(T;\mathbb S) := 
\begin{cases}
\mathbb P_{k}(T;\mathbb S),  & r = k -2,\\
\mathbb P_{k}(T;\mathbb S)\oplus\boldsymbol x\boldsymbol x^{\intercal}\mathbb H_{k-1}(T), & r = k -1.
\end{cases}
$$
We also write $\Sigma_{k}^{-1} = \Sigma_{k,k-2}^{-1}$ and $\Sigma_{k^+}^{-1} = \Sigma_{k,k-1}^{-1}$ for $k\geq 0$ when $r$ is not emphasized. The case $\Sigma_{1^{++}}^{-1}$ is defined by the enriched local space~\eqref{eq:sigma1+} and is not included in this notation system.
Define the discontinuous polynomial spaces
$$
V_{r}^{-1}(\mathcal F_h) := \prod_{F\in \mathcal F_h} \mathbb P_{r}(F), \quad V_{r}^{-1}(\mathcal E_h) := \prod_{e\in \mathcal E_h} \mathbb P_{r}(e).
$$
Spaces for the scalar function are: for $r = k-2$ or $k-1$
\begin{align*}
M^{-1}_{r,k-1, k, k} &= V_{r}^{-1}(\mathcal T_h) \times V_{k-1}^{-1}(\mathcal F_h)\times V_{k}^{-1}(\mathcal F_h)\times V_{k}^{-1}(\mathcal E_h),\\ 
\mathring{M}^{-1}_{r,k-1, k, k} &= V_{r}^{-1}(\mathcal T_h) \times V_{k-1}^{-1}(\mathring{\mathcal F}_h)\times V_{k}^{-1}(\mathring{\mathcal F}_h)\times V_{k}^{-1}(\mathring{\mathcal E}_h). 
\end{align*}
When the index is less than zero, we use $\cdot$ to de-emphasize it. For example, when $k=1,$ $ r=k-2=-1$, the space is denoted by $M^{-1}_{\cdot, 0, 1, 1}$; for $k=0$, it is $M^{-1}_{\cdot, \cdot, 0, 0}$.
Spaces on $\mathring{\mathcal F}_h$ and $\mathring{\mathcal E}_h$ can be thought of as Lagrange multiplier for the required continuity. For example, $V_{k-1}^{-1}(\mathring{\mathcal F}_h)$ is for $\tr_2(\bs \sigma)$ which is one degree lower than that of $\bs \sigma$ as $\tr_2(\bs \sigma)$ consists of first-order derivatives of $\bs \sigma$.
Space $\mathring{M}^{-1}_{r,k-1, k, k}$ can be treated as a subspace of $M^{-1}_{r,k-1, k, k}$ by zero extension to boundary faces and edges.
A function $v\in M^{-1}_{r,k-1, k, k}$ can be written as $v = (v_0, v_b, v_n, v_e)$, where $v_0$ represents function value in the interior, $v_b$ on faces, $v_e$ on edges, and $v_n$ for the normal derivative on faces. 

Introduce the inner products $(\cdot,\cdot)_{0,h}$ with weight:
\begin{align*}    
((u_0, u_b, u_n, u_e), (v_0, v_b, v_n, v_e))_{0,h} &=  \sum_{T\in \mathcal T_h}(u_0, v_0)_{T} + \sum_{F\in\mathcal F_h}h_F(u_b, v_b)_{F} \\
&\quad+ \sum_{F\in\mathcal F_h}h_F^3(u_n, v_n)_{F}+ \sum_{e\in\mathcal E_h}h_e^2(u_e, v_e)_{e}.
\end{align*}
The induced norm is denoted by $\|\cdot\|_{0,h}$. Different scalings are introduced such that all terms have the same scaling as the $L^2$-inner product $(u_0, v_0)$.

We will use either the pair $\Sigma_{k,k-2}^{-1} - \mathring{M}^{-1}_{k-2,k-1, k, k} $ or $\Sigma_{k,k-1}^{-1} - \mathring{M}^{-1}_{k-1,k-1, k, k}$, and unify the notation as $\Sigma_{k,r}^{-1} - \mathring{M}^{-1}_{r,k-1, k, k}$.


Define weak divdiv operator $(\div\div)_w: \Sigma_{k,r}^{-1} \to \mathring{M}^{-1}_{r,k-1, k, k}$ as
$$
(\div\div)_w\bs \sigma := ((\div\div)_T\bs \sigma, -h_F^{-1}[\tr_2(\bs \sigma)]|_F, h_F^{-3}[\bs n^{\intercal}\bs \sigma\bs n]|_F, h_e^{-2}[ \tr_e(\bs \sigma)]|_e),
$$
and extend to $\overline{(\div\div)}_w:  \Sigma_{k,r}^{-1} \to  M^{-1}_{r,k-1, k, k}$ by including boundary faces and edges. The negative power scaling is introduced to match the scaling of the second order derivative $(\div\div)_T\bs \sigma$.
When $\bs \sigma\in \Sigma_{k,r}^{-1}\cap H(\div\div,\Omega;\mathbb S)$, $(\div\div)_w\bs \sigma = (\div\div)\bs \sigma$ but $\overline{(\div\div)}_w\bs \sigma \neq \div\div \bs \sigma$ pointwisely as terms on the boundary faces and edges are included. However, $\overline{(\div\div)}_w\bs \sigma = (\div\div)_w\bs \sigma = \div\div \bs \sigma$ in the distribution sense as the test function vanishes on the boundary. 


For $v = (v_0, v_b, v_n, v_e)\in M^{-1}_{r,k-1, k, k}$, $k\geq 0$, define weak Hessian $\nabla_w^2 v\in \Sigma_{k,r}^{-1}$ s.t. for all $\bs \sigma \in \Sigma_{k,r}(T;\mathbb S)$ and $T\in \mathcal T_h$:
\begin{align}
\notag(\nabla_w^2 v, \bs \sigma )_T := & (v_0, \div\div_h \bs \sigma)_T \\
&- (v_b, \tr_2(\bs \sigma))_{\partial T} + (v_n\bs n_F\cdot\bs n,\bs n^{\intercal}\label{eq:weakhess}
\bs \sigma\bs n)_{\partial T}+ \sum_{e\in \Delta_{d-2}(T)}(v_e, \tr_e(\bs \sigma))_e. 
\end{align}
Using integration by parts, we also have an equivalent formula on $\nabla_w^2v$
\begin{equation}\label{eq:weakhess2}
\begin{aligned}
(\nabla_w^2 v, \bs \sigma )_T  = & (\nabla_h^2 v_0, \bs\sigma)_T +  (v_0 - v_b, \tr_2(\bs \sigma))_{\partial T} - ( \partial_n v_0 - v_n\bs n_F\cdot\bs n,\bs n^{\intercal}\bs \sigma\bs n)_{\partial T} \\
&+ \sum_{e\in \Delta_{d-2}(T)}(v_e - v_0, \tr_e(\bs \sigma))_e.
\end{aligned}
\end{equation}


For piecewise smooth $v\in H^2(\Omega)$, define $Q_M v\in M^{-1}_{r, k-1, k, k}$ by local $L^2$-projection
$$
Q_M v = (Q_{r,T} v, Q_{k-1,F}v, Q_{k,F} \partial_{n_F} v, Q_{k,e} v)_{T\in \mathcal T_h, F\in \mathcal F_h, e\in \mathcal E_h},
$$
then by definition
\begin{equation}\label{weakhessProjcd}    
\nabla_w^2 Q_M v = Q_{\Sigma} \nabla^2 v,
\end{equation}
where $Q_{\Sigma}$ is the $L^2$-projection to the space $\Sigma_{k,r}^{-1}$. 

By definition, we have the following formulae on the integration by parts. 
\begin{lemma}
 We have the integration by parts
\begin{align*}
((\div\div)_w\bs \sigma, v)_{0,h} = ( \bs \sigma, \nabla_w^2 v), \quad \bs \sigma \in \Sigma_{k,r}^{-1}, v\in \mathring{M}^{-1}_{r,k-1, k, k},\\
(\overline{(\div\div)}_w\bs \sigma, v)_{0,h} = ( \bs \sigma, \nabla_w^2 v), \quad \bs \sigma \in \Sigma_{k,r}^{-1}, v\in M^{-1}_{r,k-1, k, k}.
\end{align*}
\end{lemma}

As a consequence, for $\bs \sigma\in \Sigma_{k,r}^{-1}, v\in C_0^{\infty}(\Omega)$, we have $$((\div\div)_w\bs \sigma, Q_Mv)_{0,h} = (\bs \sigma, \nabla^2 v) = \langle \div\div \bs \sigma, v\rangle,$$ where the last $\langle\cdot, \cdot \rangle$ is the duality pair. Namely $(\div\div)_w$ can be viewed as a discretization of $\div\div$ operator in the distributional sense.


\subsection{Weak divdiv stability}
Introduce the norm square $\|\boldsymbol{\tau}\|_{\div\div_w}^2:=\|\boldsymbol{\tau}\|_{0}^2+\|(\div\div)_w\boldsymbol{\tau}\|_{0,h}^2$, and $\|\boldsymbol{\tau}\|_{\overline{\div\div}_w}^2:=\|\boldsymbol{\tau}\|_{0}^2+\|\overline{(\div\div)}_w\boldsymbol{\tau}\|_{0,h}^2$.

\begin{theorem}\label{thm:weakdivdiv}
We have the inf-sup condition: there exist constants $\alpha$ and $\bar{\alpha}$ independent of $h$ s.t.
\begin{equation}\label{eq:divdivwinfsup}
 \inf_{v\in \mathring{M}^{-1}_{r,k-1,k,k}} \sup_{\boldsymbol{\tau}\in \Sigma_{k,r}^{-1}} \frac{((\div\div)_w\boldsymbol{\tau}, v)_{0,h}}{\|\boldsymbol{\tau}\|_{\div\div_w}\|v\|_{0,h}} = \alpha >0, \quad \text{ for } k \geq 0,
\end{equation}
\begin{equation}\label{eq:divdivbarwinfsup}
 \inf_{v\in M^{-1}_{r,k-1,k,k}/\mathbb P_1} \sup_{\boldsymbol{\tau}\in \Sigma_{k,r}^{-1}} \frac{(\overline{(\div\div)}_w\boldsymbol{\tau}, v)_{0,h}}{\|\boldsymbol{\tau}\|_{\overline{\div\div}_w}\|v\|_{0,h}} = \bar{\alpha} >0, \quad \text{ for } k \geq 0.
\end{equation}
\end{theorem}
\begin{proof}
The proof of~\eqref{eq:divdivwinfsup} and~\eqref{eq:divdivbarwinfsup} is similar. We will prove~\eqref{eq:divdivwinfsup} for $r=k-2$ which also works for $r=k-1$ with appropriate change of DoFs to define $\Sigma_{k^+}^{\div\div}$ rather than $\Sigma_{k,{\rm new}}^{\div\div}$.

\step{1} We first consider the case $r\geq 1$ for which an $H(\div\div)$-conforming finite element either $\Sigma_{k}^{\div\div}, k\geq 3,$ or $\Sigma_{k^+}^{\div\div}, k\geq 2,$ have been constructed.

For $e\in \mathring{\mathcal{E}}_h$, let $|\omega_e|$ be the number of elements in $\omega_e$.
First consider a tensor $\boldsymbol{\tau}_{b}\in \Sigma_k^{-1}$ with DoFs
$$
\tr_2(\boldsymbol{\tau}_{b})|_F =  - \frac{1}{2}h_F v_b,\quad\bs n^{\intercal}\boldsymbol{\tau}_{b}\bs n |_F = \frac{1}{2}h_F^3 v_n (\bs n_F\cdot\bs n) \;\; \textrm{ on } F\in\partial T,
$$
$$
(\bs n_{F_1,e}^{\intercal}\boldsymbol{\tau}_{b} \bs n_{F_1}+\bs n_{F_2,e}^{\intercal}\boldsymbol{\tau}_{b} \bs n_{F_2})|_e = \frac{1}{|\omega_e|}h_e^2v_e \;\; \textrm{ on } e\in\Delta_{d-2}(T)
$$
for each $T\in\mathcal T_h$, and others in~\eqref{eq:newdivdivS} vanish. 
Consequently, 
$$
(\div\div)_w\boldsymbol{\tau}_{b} =((\div\div)_T\boldsymbol{\tau}_{b}, v_b, v_n, v_e)_{T\in\mathcal{T}_h}, \text{ and }
\|\boldsymbol{\tau}_{b}\|_{\div\div_w}\lesssim \|v\|_{0,h}.
$$

Then by the inf-sup condition~\eqref{eq:newdivdivinfsup}, we can find $\boldsymbol{\tau}_0\in\Sigma_{k,{\rm new}}^{\div\div}$ s.t. $\div\div\boldsymbol{\tau}_0 = v_0 - (\div\div)_h\boldsymbol{\tau}_{b}$, and $\|\boldsymbol{\tau}_0 \|_{\div\div}\lesssim \|v_0\|_0 + \| (\div\div)_h\boldsymbol{\tau}_{b}\|_{0} \lesssim \|v\|_{0,h}$. 

Set $\boldsymbol{\tau}=\boldsymbol{\tau}_0 + \boldsymbol{\tau}_{b}$. We have $(\div\div)_w\boldsymbol{\tau}= v$ and $\|\boldsymbol{\tau}\|_{\div\div_w} \lesssim \| v \|_{0,h}$, which verifies the inf-sup condition~\eqref{eq:divdivwinfsup}.

The pair $\Sigma_{1^{++}}^{-1}-\mathring{M}^{-1}_{1,1,1,1}$ can be proved similarly as an $H(\div\div)$-conforming finite element space $\Sigma_{1^{++}}^{\div\div}$ can be constructed. However, for $\Sigma_{k}^{-1}$, $k= 1,2$, and $\Sigma_{1^{+}}^{-1}$, no finite elements have been constructed and will be treated differently.

\smallskip
\step{2} Consider $k=2$. Given $v\in \mathring{M}^{-1}_{0,1,2,2}\subset  \mathring{M}^{-1}_{1,1,2,2}$, by the established inf-sup condition for $\Sigma_{2^+}^{-1}-\mathring{M}^{-1}_{1,1,2,2}$, we can find $\bs \tau \in \Sigma_{2^+}^{-1}$ s.t. $\div\div_w \bs \tau = v$. We claim $\bs \tau \in \Sigma_{2}^{-1}$ as $\div\div_T \bs \tau \in \mathbb P_0(T)$ and the range of the enrichment $\div\div(\boldsymbol x\boldsymbol x^{\intercal}\mathbb H_{1}(T)) = \mathbb H_{1}(T)$. This finishes the weak divdiv stability for $\div\div_w: \Sigma_{2}^{-1}\to \mathring{M}^{-1}_{0,1,2,2}$.

\smallskip
\step{3} Consider $k=1$. Given $v\in \mathring{M}^{-1}_{\cdot,0,1,1}\subset  \mathring{M}^{-1}_{1,1,1,1}$, by the established inf-sup condition for $\Sigma_{1^{++}}^{-1}-\mathring{M}^{-1}_{1,1,1,1}$, we can find $\bs \tau \in \Sigma_{1^{++}}^{-1}$ s.t. $\div\div_w \bs \tau = v$. As $v_0=0$, we conclude $\div\div_h \bs \tau=0$. Consequently $\bs \tau \in \mathbb P_1(T;\mathbb S)+ \sym(\boldsymbol{x}\otimes \mathbb P_1(T;\mathbb R^d))$. 
By the proof of Lemma \ref{lm:1++}, we can derive $\bs\tau\in\mathbb P_1(T;\mathbb S)$ from the fact $\tr_2(\bs \tau)\in \mathbb P_0(F)$. Namely we obtain the stability for the pair $\Sigma_{1}^{-1}-\mathring{M}^{-1}_{\cdot,0,1,1}$. 
Then by adding $\bs x^{\intercal}\bs x\mathbb P_0(T)$ element-wise, we obtain the stability for $\Sigma_{1^{+}}^{-1} -\mathring{M}^{-1}_{0,0,1,1}$. This finishes all $k=1$ cases.

\smallskip
\step{4} Consider $k=0$.
We shall use the non-conforming finite element space as the bridge. The space $\mathring{M}^{-1}_{\cdot, \cdot, 0, 0}$ can be identified as the Morley-Wang-Xu (MWX) element $\mathring{V}^{\rm MWX}_2$~\cite{WangXu2006} through the bijection $Q_M: \mathring{V}^{\rm MWX}_2\to \mathring{M}^{-1}_{\cdot,\cdot,0,0}$. Similar as~\eqref{weakhessProjcd}, it holds 
$\nabla_w^2 Q_M\chi = Q_{\Sigma} \nabla_h^2 \chi$ for $\chi\in \mathring{V}^{\rm MWX}_2$.
Given $v\in \mathring{M}^{-1}_{\cdot,\cdot,0,0}$, let $w_h\in\mathring{V}^{\rm MWX}_2$ satisfy
\begin{equation*}
(\nabla_h^2w_h, \nabla_h^2\chi)=(v, Q_M\chi)_{0,h}, \quad\chi\in\mathring{V}^{\rm MWX}_2.
\end{equation*}
Take $\boldsymbol{\tau}=\nabla_h^2w_h\in \Sigma_{0}^{-1}$, then $\div\div_w \bs \tau = v$, and
\begin{equation*}
\|\boldsymbol{\tau}\|_0^2=(\boldsymbol{\tau}, \nabla_w^2 Q_Mw_h)=(v, Q_Mw_h)_{0,h}\leq \|v\|_{0,h}\|Q_Mw_h\|_{0,h}.
\end{equation*}
By the norm equivalence $\|Q_Mw_h\|_{0,h}\eqsim \|w_h\|_0$ of MWX element and the Poincar\'e inequality $\|w_h\|_0\lesssim \|\nabla_h^2w_h\|_0$ \cite[Lemma~8]{WangXu2006}, we have $\|\boldsymbol{\tau}\|_0\lesssim \|v\|_{0,h}$, which means $\|\boldsymbol{\tau}\|_{\div\div_w} \lesssim \| v \|_{0,h}$. Thus the inf-sup condition~\eqref{eq:divdivwinfsup} holds for $k=0$.
\end{proof}

As the adjoint of the $\div\div_w$, $\nabla_w^2$ is injective. We obtain another version of the inf-sup condition.
\begin{corollary}
We have
\begin{equation}\label{eq:infsuphessian}
\inf_{v\in \mathring{M}^{-1}_{r,k-1,k,k}} \sup_{\boldsymbol{\tau}\in \Sigma_{k,r}^{-1}} \frac{(\div\div_w \boldsymbol{\tau}, v)_{0,h}}{\|\boldsymbol{\tau}\|_{0}\| \nabla_w^2v \|_0} = 1, \quad \text{ for } k\geq 0.
\end{equation}
\end{corollary}
\begin{proof}
We can take $\bs \tau = \nabla_w^2v$ to finish the proof as $\nabla^2_w: \mathring{M}^{-1}_{r,k-1, k, k}\to \Sigma_{k,r}^{-1}$ is injective and $\| \nabla^2_w \cdot\|_0$ is a norm on $\mathring{M}^{-1}_{r,k-1, k, k}$.
\end{proof}

\subsection{Hybridized discretization of the biharmonic equation}
A hybridization of the mixed finite element discretization~\eqref{eq:biharmonicMfem} of the biharmonic equation is: 
Find $\bs \sigma_h \in \Sigma_{k,r}^{-1}$ and $u_h\in \mathring{M}^{-1}_{r,k-1, k, k}$ s.t.
\begin{subequations}\label{eq:biharmonicMfemhy}
 \begin{align}
\label{eq:biharmonicMfemhy1}
(\bs \sigma_h, \bs \tau) + (\div\div_w \bs \tau,  u_h)_{0,h}&=0  \qquad\qquad\quad \forall~\boldsymbol{\tau} \in \Sigma_{k,r}^{-1},\\
\label{eq:biharmonicMfemhy2}
(\div\div_w \bs \sigma_h, v)_{0,h} & = -(f, v_0) \quad \quad  \forall~v\in \mathring{M}^{-1}_{r,k-1, k, k},
\end{align}
\end{subequations}
with appropriate modification of $(f, v_0)$ for the case $r\leq 0$ which will be discussed later.

More generally, for a given function $f_h = (f_0, f_b, f_n, f_e) \in \mathring{M}^{-1}_{r,k-1, k, k}$, we consider the mixed variational problem
\begin{subequations}\label{eq:hy}
 \begin{align}
\label{eq:hy1}
(\bs \sigma_h, \bs \tau) + (\div\div_w \bs \tau,  u_h)_{0,h}&=0  \qquad\qquad\quad \forall~\boldsymbol{\tau} \in \Sigma_{k,r}^{-1},\\
\label{eq:hy2}
(\div\div_w \bs \sigma_h, v)_{0,h} & = (f_h, v)_{0,h} \quad \;\;  \forall~v\in \mathring{M}^{-1}_{r,k-1, k, k}.
\end{align}
\end{subequations}
The biharmonic equation is a special case with $f_h=(-Q_r f, 0,0,0)$. 

\begin{lemma}\label{lm:equivalence}
The hybridized mixed finite element method~\eqref{eq:hy} has a unique solution $\bs \sigma_h \in \Sigma_{k,r}^{-1}$ and $u_h=((u_h)_0, (u_h)_b, (u_h)_n, (u_h)_e)\in \mathring{M}^{-1}_{r,k-1, k, k}$ for $k\geq 0$, and 
\begin{equation}\label{eq:hystability}
\| \bs \sigma_h\|_{\div\div_w} + \| u_h\|_{0,h} \lesssim \| f_h\|_{0,h}. 
\end{equation}
Moreover,  when $r\geq 1$, the solution $(\bs \sigma_h, (u_h)_0)\in\Sigma_{k,r}^{\div\div}\times V^{-1}_{r}$ to~\eqref{eq:biharmonicMfemhy} is the solution of the mixed finite element method~\eqref{eq:biharmonicMfem}.
\end{lemma}
\begin{proof}
The discrete method~\eqref{eq:hy} is well-posed thanks to the weak divdiv stability~\eqref{eq:divdivwinfsup}. 
The stability~\eqref{eq:hystability} is from the Babuska-Brezzi theory. 
%

For the biharnominc equation~\eqref{eq:biharmonicMfemhy},  $f_h=(-Q_r f, 0,0,0)$. Therefore $\bs \sigma_h\in \Sigma_{k,r}^{\div\div}$ and $\div\div \bs \sigma_h = Q_{r}f$. By restricting $\bs \tau\in \Sigma_{k,r}^{\div\div}$ in~\eqref{eq:biharmonicMfemhy1}, we conclude $(\bs \sigma_h, (u_h)_0)$ is the solution to~\eqref{eq:biharmonicMfem}.  
\end{proof}
Notice that the mixed formulation~\eqref{eq:biharmonicMfem} is only presented for $r\geq 1, k\geq 2$, where $H(\div\div)$-conforming finite elements are constructed. While the hybridized version is well-posed for all $k\geq 0$.

Using the stability result~\eqref{eq:hystability}, we can prove the following discrete Poincar\'e inequality. 
\begin{lemma}
 On the space $\mathring{M}^{-1}_{r,k-1, k, k}$, we have
\begin{equation}\label{eq:Poincare}    
\|u\|_{0,h}  \lesssim \|\nabla_w^2u\|_0, \quad u\in \mathring{M}^{-1}_{r,k-1, k, k}\quad \text{ for } k\geq 0.
\end{equation}
\end{lemma}
\begin{proof}
For $f_h = u$ in~\eqref{eq:hy}, we can find $\bs \sigma\in \Sigma^{-1}_{k,r}$ s.t. $\div\div_w\bs \sigma = u$ and $\|\bs \sigma\|_0\lesssim \|u\|_{0,h}$. Set $v = u$ in ~\eqref{eq:hy2}, we obtain
$$
\|u\|_{0,h}^2 = (\div\div_w\bs \sigma, u)_{0,h} = (\bs \sigma, \nabla_w^2 u)\leq \| \bs \sigma\|_0\| \nabla_w^2 u\|_0\lesssim \|u\|_{0,h}\| \nabla_w^2 u\|_0,
$$
which implies the desired inequality.
\end{proof}


We now present error analysis of scheme~\eqref{eq:biharmonicMfemhy} for $r \geq 1$ which is equivalent to the mixed finite element method~\eqref{eq:biharmonicMfem}. Thus we focus on the error estimate of $u_h$.
\begin{theorem}\label{th:mfemerror}
Let $u\in H_0^2(\Omega)$ be the solution of biharmonic equation~\eqref{eq:biharmonic} and $\bs \sigma=-\nabla^2u$. 
Let $\bs \sigma_h\in \Sigma_{k,r}^{-1}, u_h\in \mathring{M}^{-1}_{r,k-1, k, k}$ be the solution of the discrete method~\eqref{eq:biharmonicMfemhy} for $r \geq 1$ and $k\geq 2$. 
Assume $u\in H^{k+3}(\Omega)$. We have
\begin{equation*}
\|\nabla_w^2(Q_Mu-u_h)\|_0 + \|Q_Mu-u_h\|_{0,h} \lesssim h^{k+1}|u|_{k+3}.    
\end{equation*}
\end{theorem}
\begin{proof}
In~\eqref{eq:biharmonicMfemhy}, as $\bs \sigma_h$ is discontinuous, we can eliminate $\bs \sigma_h$ elementwisely and use the weak Hessian to obtain an equivalent formulation: find $u_h\in \mathring{M}^{-1}_{r,k-1, k, k}$, s.t. 
\begin{equation*}
(\nabla_w^2 u_h, \nabla_w^2 v) = (f, v_0) \quad \forall~v\in \mathring{M}^{-1}_{r,k-1, k, k}.
\end{equation*}
For $r\geq 1$, we have the canonical interpolation $\bs \sigma_I\in \Sigma_{k,r}$ satisfying 
$$
(\bs \sigma_I, \nabla_w^2 v) = (\div\div_w \bs \sigma_I , v)_{0,h} = (\div\div \bs \sigma_I , v_0)= (Q_r \div\div \bs \sigma, v_0) = -(f, v_0).
$$
On the other hand, we have the property $\nabla_w^2 Q_M u = Q_{\Sigma}\nabla^2 u = -Q_{\Sigma}\bs \sigma$. 

Let $v = Q_Mu-u_h$. We then have
\begin{align*}
\|\nabla_w^2(Q_Mu-u_h)\|_0^2 &
=-(Q_{\Sigma}\boldsymbol{\sigma}, \nabla_w^2v) - (f, v_0) = ( \bs \sigma_I-Q_{\Sigma}\boldsymbol{\sigma}, \nabla_w^2 v).
\end{align*}
Then the error estimate on $\|\nabla_w^2(Q_Mu-u_h)\|_0$ follows from Cauch-Schwarz inequality, triangle inequality, and the estimate of $\|\bs \sigma - Q_{\Sigma}\bs \sigma\|$ and $\|\bs \sigma - \bs \sigma_I\|$.

Estimate on $\|Q_Mu-u_h\|_{0,h}$ is a consequence of the Poincar\'e inequality~\eqref{eq:Poincare}.
\end{proof}

Note that Theorem \ref{th:mfemerror} covers only the case $r\geq 1, k\geq 2$. We now give corrections to low order cases: $k=0,1,2$ and $r\leq 0$. 

For $k=1,2$, we define $v^{\rm CR}\in\mathbb P_1(T)$ by $Q_{0,F}v^{\rm CR} = Q_{0,F}v_b$ for $F\in\partial T$. The load term 
$(f, v_0)$ is replaced by $(f, v^{\rm CR})$ for $k=1$ and by $(f, v^{\rm CR}+v_0-Q_0v^{\rm CR})$ for $k=2, r=0$. 

For $k=0$, $v = (v_n, v_e)\in M^{-1}_{\cdot, \cdot, 0, 0}$, we define $v_0 = Q_M^{-1}v\in \mathring{V}^{\rm MWX}_2$ and $v_b = Q_M^{-1}v$ on $\partial T$. With this $v_b$, we can define $v^{\rm CR}$. 
From this point of view,~\eqref{eq:biharmonicMfemhy} generalizes the well-known $\mathbb P_2$ Morley element to higher order and arbitrary dimensions.

We can write $(f, v_0 + (I - Q_r)v^{\rm CR})$ for all $k\geq 0$ cases. We will present the error analysis after we identify~\eqref{eq:biharmonicMfemhy} with the non-conforming virtual element methods (VEM).

\subsection{Equivalence to other methods}
In~\eqref{eq:biharmonicMfemhy}, as $\bs \sigma_h$ is discontinuous, we can eliminate $\bs \sigma_h$ elementwisely and use the weak Hessian to obtain a weak Galerkin formulation: find $u_h\in \mathring{M}^{-1}_{r,k-1, k, k}$, s.t. 
%
\begin{equation}\label{eq:biharmonicMfemWG}
(\nabla_w^2 u_h, \nabla_w^2 v) = (f, v_0 + (I - Q_r)v^{\rm CR}) \quad \forall~v\in \mathring{M}^{-1}_{r,k-1, k, k}, \quad k\geq 0.
\end{equation}
The discrete method~\eqref{eq:biharmonicMfemWG} is well-posed, since $\| \nabla^2_w (\cdot)\|$ constitutes a norm on the space $\mathring{M}^{-1}_{r,k-1, k, k}$ by~\eqref{eq:infsuphessian}. Indeed~\eqref{eq:biharmonicMfemWG} is equivalent to~\eqref{eq:biharmonicMfemhy}. Moreover, the weak divdiv stability, which is equivalent to the coercivity of the bilinear form $(\nabla_w^2 \cdot, \nabla_w^2 \cdot)$, obviates the need for any additional stabilization. This not only simplifies the implementation, but also facilitates the error analysis. Some weak Galerkin methods without extrinsic stabilization for the biharmonic equation are designed recently on polytopal meshes in~\cite{YeZhang2020,ZhuXieWang2023}. 

For a simplex $T$, recall the local space of the $H^2$-nonconforming virtual element introduced in~\cite{ChenHuang2020a} for $r=k-2$ or $ k-1$
\begin{align*}
V_{k+2}^{\rm VEM}(T):=\big\{  u \in H^2(T): & \tr_1(\nabla^2 u)|_F\in\mathbb P_{k}(F), \tr_2(\nabla^2 u)|_F\in\mathbb P_{k-1}(F), \\ &
 \tr_e(\nabla^2 u)\in\mathbb P_{k}(e)\; \forall~F\in\partial T, e\in \Delta_{d-2}(T), \Delta^2u \in \mathbb P_{r}(T) \big \}.
\end{align*}
Define the global virtual element space
\begin{align*}
\mathring{V}_{k+2}^{\rm VEM}:=\big\{ & u \in L^2(\Omega): u|_T \in V_{k+2}^{\rm VEM}(T) \textrm{ for } T\in\mathcal T_h, \; 
Q_{k-1,F}u, Q_{k,F}(\partial_{n_F}u),  \\
& Q_{k,e}u \textrm{ are single-valued for } F\in\mathring{\mathcal{F}}_h, e\in\mathring{\mathcal{E}}_h, \textrm{ and vanish on boundary } \partial\Omega\big\}.
\end{align*}
The well-posedness of VEM space using DoFs $(Q_{r,T}u,
Q_{k-1,F}u, Q_{k,F}(\partial_{n_F}u), Q_{k,e}u)$ can be found in~\cite{ChenHuang2020a}. In general a function $v\in \mathring{V}_{k+2}^{\rm VEM}$ is non-polynomial, with the only exception of $k=0$, and thus its point-wise value may not be known. Instead several projections to polynomial spaces using DoFs will be used.

Given a function $(v_0, v_b, v_n, v_e)\in M^{-1}_{r,k-1, k, k}$, we can  define an $H^2$ nonconforming virtual element function $v\in \mathring{V}_{k+2}^{\rm VEM}$ by $Q_Mv=(v_0, v_b, v_n, v_e)$. That is $Q_M:  \mathring{V}_{k+2}^{\rm VEM} \to \mathring{M}^{-1}_{r,k-1, k, k}$ is a bijection.
Similar as~\eqref{weakhessProjcd}, it holds 
\begin{equation}\label{weakhessVEMProjcd}
\nabla_w^2 Q_M v = Q_{\Sigma} \nabla_h^2 v\quad\quad\forall~v\in \mathring{V}_{k+2}^{\rm VEM}.
\end{equation}
We have a unified construction $v^{\rm CR} = I^{\rm CR} v$ where $I^{\rm CR}$ is the interpolation operator to the nonconforming linear element space. The face integral $\int_F v$ is a DoF when $k\geq 1$ and when $k = 0$, $\int_F v$ is computable as $v$ is a quadratic polynomial.

%

Then~\eqref{eq:biharmonicMfemWG} becomes: find $u_h\in \mathring{V}_{k+2}^{\rm VEM}$, for $k\geq 0$, s.t. 
\begin{equation}\label{eq:biharmonicMfemVEM}
(Q_{\Sigma} \nabla_h^2 u_h, Q_{\Sigma} \nabla_h^2 v) = (f, v^{\rm CR} + Q_r(v -v^{\rm CR})) \quad \forall~v\in \mathring{V}_{k+2}^{\rm VEM}.
\end{equation}
So we obtain a stabilization free non-conforming VEM for the biharmonic equation on triangular meshes due to the weak divdiv stability. 

We will use the following norm equivalence, whose proof can be found in Appendix~\ref{apdx:normequiv}, and the error analysis of VEM to provide another convergence analysis of~\eqref{eq:biharmonicMfemhy}
\begin{equation}
\label{eq:VEMnormequivH2}
\|Q_{\Sigma} \nabla_h^2 v\|_0 \eqsim \|\nabla_h^2 v\|_0, \quad v\in \mathring{V}_{k+2}^{\rm VEM}, k\geq 0.
\end{equation}



\begin{theorem}\label{th:mfemerrork012}
Let $u\in H_0^2(\Omega)$ be the solution of biharmonic equation~\eqref{eq:biharmonic} and $\bs \sigma=-\nabla^2u$. 
Let $\bs \sigma_h\in \Sigma_{k,r}^{-1}, u_h\in \mathring{M}^{-1}_{r,k-1, k, k}$ be the solution of the discrete method~\eqref{eq:biharmonicMfemhy} for $k \geq 0$. 
Assume $u\in H^{k+3}(\Omega)$. We have
\begin{equation*}
\|\nabla_w^2(Q_Mu-u_h)\|_0 + \|Q_Mu-u_h\|_{0,h} \lesssim h^{k+1}(|u|_{k+3}+\delta_{k0}h\|f\|_0).    
\end{equation*}
\end{theorem}
\begin{proof}
Due to the equivalence between~\eqref{eq:biharmonicMfemVEM} and~\eqref{eq:biharmonicMfemhy}, it is equivalent to prove
\begin{equation*}
\|\nabla_w^2Q_M(u-u_h)\|_0 + \|Q_M(u-u_h)\|_{0,h} \lesssim h^{k+1}(|u|_{k+3}+\delta_{k0}h\|f\|_0),    
\end{equation*}
where $u_h\in \mathring{V}_{k+2}^{\rm VEM}$ is the solution of the virtual element method~\eqref{eq:biharmonicMfemVEM}.

We outline the proof and refer to \cite{ChenHuang2020a} for details. Notice that there is an index shift in the notation. Results in \cite{ChenHuang2020a} are applied to $\mathring{V}_{k+2}^{\rm VEM}$ with degree $k+2$ for $k\geq 0$. 

Let $I_hu$ be the nodal interpolation of $u$ based on the DoFs of $V_{k+2}^{\rm VEM}(T)$ \cite[(2.6)-(2.9)]{ChenHuang2020a}. Then $Q_Mu=Q_M(I_hu)$ and thus $Q_{\Sigma} \nabla_h^2(I_hu) = \nabla_w^2 Q_M(I_hu) = \nabla_w^2 Q_M u = Q_{\Sigma} \nabla^2u$. Set $v=I_hu-u_h$. We have the error equation
\begin{align*}
\|Q_{\Sigma}\nabla_h^2(I_hu-u_h)\|_0^2 =&  ((Q_{\Sigma} - I) \nabla^2u,  \nabla_h^2v) + (\nabla^2u,  \nabla_h^2v) - (f, v) \\
&+ (f, (I-Q_r)(v -v^{\rm CR})).
\end{align*}
The first term is bounded by 
$$
((Q_{\Sigma} - I) \nabla^2u,  \nabla_h^2v)\leq \|(Q_{\Sigma} - I) \nabla^2u \|_0 \| \nabla_h^2v\|_0 \lesssim h^{k+1}|u|_{k+3} \| \nabla_h^2v\|_0.
$$
The second term is the consistence error \cite[Lemma 5.5 and 5.6]{ChenHuang2020a} 
\begin{equation*}
(\nabla^2u, \nabla_h^2 v_h) - (f, v_h) \lesssim h^{k+1}(|u|_{k+3}+\delta_{k0}h\|f\|_0)\| \nabla_h^2 v_h\|_0. 
\end{equation*}
The third term is a perturbation and can be bounded by
\begin{align*}
(f, (I-Q_r)(v -v^{\rm CR}))&=((I-Q_r)f, v -v^{\rm CR})\\
&\lesssim h^{k+1}\big((1-\delta_{k0})|f|_{k-1}+\delta_{k0}h\|f\|_0\big)\|\nabla_h^2v\|_0.
\end{align*}

Putting together, we have
\begin{align*}
\|Q_{\Sigma}\nabla_h^2(I_hu-u_h)\|_0^2  
&\lesssim h^{k+1}(|u|_{k+3}+\delta_{k0}h\|f\|_0)\|\nabla_h^2v\|_0 \\
&\lesssim h^{k+1}(|u|_{k+3}+\delta_{k0}h\|f\|_0)\| Q_{\Sigma}\nabla_h^2v\|_0,
\end{align*}
where in the last step, we have used the norm equivalence~\eqref{eq:VEMnormequivH2}. Canceling one $\| Q_{\Sigma}\nabla_h^2v\|_0$ to get the desired error estimate.
\end{proof}

In view of $\bs \sigma_h=-Q_{\Sigma} \nabla_h^2 u_h$, the post-processing $u_h^{\ast}$ defined by~\eqref{eq:postprocess} is indeed the local $H^2$ projection of $u_h$ to the polynomial space, i.e.,
\begin{equation*}
(\nabla^2u_h^{\ast}, \nabla^2v)_T=(\nabla^2u_h, \nabla^2v)_T,\quad v\in \mathbb P_{k+2}(T), T\in\mathcal T_h.
\end{equation*}

When some partial continuity is imposed on $\Sigma_k^{-1}$, we can simplify the pair space. For example, consider the normal-normal continuous element $\Sigma_{k,r}^{\rm nn}$ by asking DoFs on $\bs n^{\intercal}\bs \tau \bs n$ are unique, then there is no need of Lagrange multiplier for $u_n$. We have the surjectivity
$$\Sigma_k^{\rm nn} \xrightarrow{\div{\div}_w} \mathring{M}^{-1}_{k-2, k-1, \cdot , k}\xrightarrow{}0 \quad \text{ for } k\geq 0.
$$ 

Given a function $(u_0, u_b, u_e)\in \mathring{M}^{-1}_{k-2,k-1, \cdot, k}$, for $k\geq 1$, using $(u_0,u_b)$, we can define a weak gradient $\nabla_w (u_0,u_b)\in \mathbb P_{k-1}(T; \mathbb R^d)$ by 
$$
(\nabla_w (u_0,u_b), \bs q)_T = - (u_0, \div \bs q)_{T} + (u_b, \bs n^{\intercal}\bs q)_{\partial T}, \quad \bs q\in \mathbb P_{k-1}(T; \mathbb R^d),
$$
and a surface weak gradient $\nabla_{w,F}(u_b, u_e)\in \mathbb P_{k}(F; \mathbb R^{d-1})$ using $(u_b, u_e)$ by
$$
(\nabla_{w,F} (u_b, u_e), \bs q)_F = - (u_b, \div_F \bs q)_{F} + (u_e, \bs n_{F,e}^{\intercal}\bs q)_{\partial F}, \quad \bs q\in \mathbb P_{k}(F; \mathbb R^{d-1}),
$$
where $\mathbb P_{k-1}(F; \mathbb R^{d-1})$ is the polynomial vector tangential to $F$. For $k=0$, we only have $u_e$ on edges and can define $u_b$ as the nonconforming linear element on $F$ based on $u_e$ on $\partial F$. After that, using the average of $u_b$, to define the nonconforming linear element inside $T$. 

With this notation, we have a simpler formulation of $\div\div_w$
\begin{align}
\notag (\div\div_w \bs \tau, (v_0, v_b, v_e))_{0,h} =& -\sum_{T\in \mathcal T_h} (\div_h \bs \tau, \nabla_w (v_0,v_b))_T \\
\label{eq:HHJdivdiv} &+ \sum_{F\in \mathring{\mathcal F}_h} ([\Pi_F\bs \tau\bs n], \nabla_{w,F}(v_b,v_e))_F.
\end{align}
In computation,~\eqref{eq:HHJdivdiv} provides an alternative discretization without relatively complicated trace $\tr_2(\bs \tau)$ and $\tr_e(\bs \tau)$. 

In two dimensions, the space $\mathring{M}^{-1}_{k-2, k-1, \cdot , k}$ can be identified as the Lagrange element $\mathring{V}^L_{k+1}$. The weak gradient operators become the gradient operators and~\eqref{eq:HHJdivdiv} is the bilinear form used in the HHJ formulation. Therefore restricting to the pair $\Sigma^{\rm nn}_{k} - \mathring{M}^{-1}_{k-2,k-1, \cdot, k}$, we generalize HHJ to high dimensions whose hybridization is exactly~\eqref{eq:biharmonicMfemhy} with appropriate correction on $(f,v_0)$ for low order cases.

\subsection{A $C^0$ DG method for the biharmonic equation}\label{sec:cdgbiharmonic}
A $C^0$ discontinuous Galerkin (CDG) method for biharmonic equation can be developed by embedding the Lagrange element space $\mathring{V}_k(\mathcal T_h)$ into the broken space $\mathring{M}^{-1}_{r,k-1, k, k}$. This approach enables us to preserve the optimal order of convergence while reducing the size of the linear algebraic system.

We start with the embedding, for $k\geq 2$, 
\begin{align*}
E^{\rm CDG} &: \mathring{V}_k(\mathcal T_h) \to \mathring{M}^{-1}_{r,k-1, k, k},\\
E^{\rm CDG}u &:= (Q_{r,T}u, Q_{k-1,F}u, \{\partial_{n_F}u\}|_F, u|_e)_{T\in \mathcal T_h, F\in\mathring{\mathcal{F}}_h, e\in\mathring{\mathcal{E}}_h}. 
\end{align*}
For the boundary face $F\in \partial\mathcal{F}_h$ and $F\subset\partial T$, modify the jump and the average as
\begin{equation}\label{eq:scalejump}
[u]=2u|_{T},\quad \{u\}=\frac{1}{2}u|_{T}.
\end{equation}
By~\eqref{eq:weakhess}, for any $\boldsymbol{\tau}\in\Sigma_{k,r}^{-1}(T;\mathbb S)$, the weak Hessian $\nabla_w^2 E^{\rm CDG}u$ is
\begin{equation}\label{eq:weakhessCDG}
\begin{aligned}
(\nabla_w^2 E^{\rm CDG}u, \boldsymbol{\tau} )_T = & (u, (\div\div)_T \boldsymbol{\tau})_T+ (\{\partial_{n_F}u\}\bs n_F\cdot\bs n,\bs n^{\intercal}\boldsymbol{\tau}\bs n)_{\partial T\cap\mathring{\mathcal{F}}_h} \\
&- (u, \tr_2(\boldsymbol{\tau}))_{\partial T}+ \sum_{e\in \Delta_{d-2}(T)}(u, [\boldsymbol n_{F,e}^{\intercal}\boldsymbol{\tau}\boldsymbol n]|_T)_e\\
 = & (\nabla_h^2 u, \boldsymbol{\tau})_T  - \frac{1}{2}( [\partial_n u],\bs n^{\intercal}\boldsymbol{\tau}\bs n)_{\partial T},
\end{aligned}
\end{equation}
where we use the fact $\partial_nu - \{\partial_{n_F}u\}\bs n_F\cdot\bs n = \frac{1}{2}[\partial_nu]$.

Let $(0, 0, u_n, 0)\in \mathring{M}^{-1}_{r,k-1, k, k}$ be given. By the definition of the weak Hessian, we have
$$
(\nabla_w^2 u_n,\boldsymbol{\tau})_T:= (\nabla_w^2 (0, 0,u_n,0),\boldsymbol{\tau})_T = (u_n\bs n_F\cdot\bs n,\bs n^{\intercal}\boldsymbol{\tau}\bs n)_{\partial T},
$$
where $u_n$ is defined on faces only, while $\nabla_w^2 u_n$ is element-wise polynomial. This quantity is sometimes referred to as the ``lifting" of a boundary trace in the literature~\cite{WellsDung2007,BassiRebay1997,BrezziManziniMariniPietraEtAl2000}. 

To save notation, define $\nabla_{w}^2u :=\nabla_w^2 E^{\rm CDG}u$ for $u\in \mathring{V}_{k}$. We can write~\eqref{eq:weakhessCDG} as
\begin{equation}\label{eq:WGlocal2}
\nabla_{w}^2u= \nabla_h^2 u - \frac{1}{2}\nabla_w^2[\partial_nu]\boldsymbol{n}_{F}\cdot\boldsymbol{n}_{\partial T},
\end{equation}
where $[\partial_nu]\in V^{-1}_{k-1}(\mathcal F_h)$ and $\boldsymbol{n}_{F}\cdot\boldsymbol{n}_{\partial T}=\pm 1$ accounting for the consistency of orientation of face $F$.

Restricting the bilinear form $(\nabla_w^2 \cdot, \nabla_w^2 \cdot)$ to the subspace $E^{\rm CDG}\mathring{V}_k(\mathcal T_h)$, we obtain a $C^0$ DG formulation.
\begin{lemma}
For $u, v\in \mathring{V}_k(\mathcal T_h)$, for $k\geq 2$, we have
\begin{equation*}
(\nabla_w^2 u, \nabla_w^2 v) = a^{\rm CDG}(u,v),
\end{equation*}
where
\begin{align*}    
a^{\rm CDG}(u,v) &= \sum_{T\in \mathcal T_h} ( \nabla_h^2 u, \nabla_h^2 v)_T - \sum_{F\in{\mathcal F}_h}\left [( \{\partial_{nn}u \}, [\partial_nv])_F + ([\partial_nu], \{\partial_{nn}v\})_{F} \right ]\\
&\quad + \frac{1}{4}(\nabla_w^2[\partial_nu], \nabla_w^2 [\partial_nv]).
\end{align*}
\end{lemma}
\begin{proof}
It is a straightforward substitution of~\eqref{eq:WGlocal2} into $(\nabla_w^2 u, \nabla_w^2 v)$. The cross term 
$$
\frac{1}{2}\sum_{T\in \mathcal T_h} (\nabla_h^2u, \nabla_w^2[\partial_nv]\boldsymbol{n}_{F}\cdot\boldsymbol{n}_{\partial T}) = \frac{1}{2}\sum_{T\in \mathcal T_h} (\partial_{nn} u, [\partial_nv])_{\partial T} = \sum_{F\in{\mathcal F}_h} ( \{\partial_{nn}u \}, [\partial_nv])_F,
$$
where the scaling $2$ or $1/2$ in~\eqref{eq:scalejump} are introduced for the unity of notation for interior and boundary faces.
\end{proof}

We obtain a $C^0$ DG method for the biharmonic equation: Find $u_h\in \mathring{V}_k(\mathcal T_h)$ s.t. 
\begin{equation}\label{eq:CDG}
a^{\rm CDG}(u_h, v) = (f, Q_{r}v)\quad \forall~v\in \mathring{V}_k(\mathcal T_h).
\end{equation}
The boundary condition $u|_{\partial \Omega} = 0$ is build into the space $\mathring{V}_k(\mathcal T_h)$ while $\partial_n u |_{\partial \Omega}= 0$ is weakly imposed in DG sense. 

It is worth noting that the widely-used interior penalty $C^0$ DG (IPCDG) method for the biharmonic equation~\cite{EngelGarikipatiHughesLarsonEtAl2002,BrennerSung2005} requires a stabilization term in the form $\gamma(h_F^{-1}[\partial_n u],[\partial_n v])_{\mathcal{F}_h}$, where $\gamma$ is chosen to be sufficiently large. In contrast, the CDG method~\eqref{eq:CDG} employs the bilinear form of the weak Hessian of jumps, i.e., $(\nabla_w^2 [\partial_n u], \nabla_w^2 [\partial_n v])$, as a parameter-free stabilization technique. It coincides with the approach proposed in~\cite[(2.9)]{HuangHuang2014} for the two-dimensional case. 

The error analysis can be carried out following the approach in 2D \cite{HuangHuang2014}.
To save the space, we only present the result below. 
\begin{theorem}
Let $u\in H_0^2(\Omega)$ be the solution of biharmonic equation~\eqref{eq:biharmonic}. 
Let $u_h\in \mathring{V}_k(\mathcal T_h)$ be the solution of the discrete method~\eqref{eq:CDG} for $k\geq 2$. 
Assume $u\in H^{k+1}(\Omega)$. We have
\begin{equation*}
\|\nabla^2u- \nabla_w^2 u_h\|_0 \lesssim h^{k-1}(|u|_{k+1}+|f|_{\max\{k-3,0\}}).    
\end{equation*}
\end{theorem}

The resulting linear algebraic system from the $C^0$ DG discretization is significantly reduced compared to the hybridized version. Despite the use of this simpler element, the method retains the optimal order of convergence. Hence, the $C^0$ DG method provides an attractive alternative to the hybridized approach. On the other hand, the hybridized mixed finite element method~\eqref{intro:biharmonicMfemhy} can be post-processed to improve the convergence rate.

\section{Finite element divdiv complexes in three dimensions}
\label{sec:femdivdivcomplex}
In this section we will first present finite element divdiv complexes involving conforming finite element spaces. Then we construct the distributional finite element divdiv complexes using the weak divdiv operator.

\subsection{Conforming finite element divdiv complexes}
The three-dimensional divdiv complex is~\cite{Arnold;Hu:2020Complexes,PaulyZulehner2020}
\begin{align*}
\resizebox{1.0\hsize}{!}{$
{\bf RT}\xrightarrow{\subset} H^1(\Omega;\mathbb R^3)\xrightarrow{\dev\grad} H(\sym\curl,\Omega;\mathbb T)\xrightarrow{\sym\curl} H(\div\div, \Omega;\mathbb S) \xrightarrow{\div{\div}} L^2(\Omega)\xrightarrow{}0,
$}
\end{align*}
where ${\bf RT}= \{a\boldsymbol x + \boldsymbol b: a\in \mathbb R, \boldsymbol b \in \mathbb R^3\}$, $H(\sym\curl,\Omega;\mathbb T)$ is the space of traceless tensor $\boldsymbol \sigma\in L^2(\Omega;\mathbb T)$ such that $\sym \curl \boldsymbol \sigma \in L^2(\Omega; \mathbb S)$ with the row-wise $\curl$ operator. 

\subsubsection{Finite element complexes starting from Hermite element}
We start from the vectorial Hermite element space in three dimensions~\cite{Ciarlet1978}
\begin{align*}
V_{k+2}^H:=\{\boldsymbol v_h\in H^1(\Omega;\mathbb R^3): \; &\boldsymbol v_h|_T\in\mathbb P_{k+2}(T;\mathbb R^3)\textrm{ for each } T\in\mathcal T_h,\\
&\nabla \boldsymbol v_h(\delta) \textrm{ is single-valued at each vertex $\delta$ of $\mathcal T_h$} \}.
\end{align*}

Since no supersmooth DoFs in~\eqref{eq:newdivdivS}, we can use DoFs for $H(\sym\curl, \Omega;\mathbb T)$-conforming finite elements simpler than those in~\cite{ChenHuang2022,Chen;Huang:2020Finite,Hu;Liang;Ma:2021Finite}.
Take the space of shape functions as $\mathbb P_{k+1}(T;\mathbb T)$.
The degrees of freedom are given by
\begin{subequations}\label{newHsymcurlfemdof}
\begin{align}
\boldsymbol \tau(\delta), & \quad\delta\in \Delta_0(T), \bs \tau \in \mathbb T, \label{newHsymcurlfemdof1}\\
(\boldsymbol{n}_i^{\intercal}\boldsymbol \tau\boldsymbol{t}, q)_e, & \quad q\in\mathbb P_{k-1}(e), e\in\Delta_1(T), i=1,2,\label{newHsymcurlfemdof2}\\
(\boldsymbol n\times \sym(\boldsymbol\tau\times\boldsymbol n)\times \boldsymbol n, \boldsymbol q)_F, &\quad\boldsymbol q\in\mathbb B_{k+1}^{tt}(F;\mathbb S), F\in\partial T,\label{newHsymcurlfemdof3}\\
(\boldsymbol n\cdot \boldsymbol\tau\times\boldsymbol n, \boldsymbol q)_F, & \quad\boldsymbol q\in \mathbb B^{\div_F}_{k+1}(F),  F\in\partial T,\label{newHsymcurlfemdof4}\\
(\boldsymbol \tau, \boldsymbol q)_T, & \quad\boldsymbol q\in\mathbb B_{k+1}(\sym\curl, T;\mathbb T), \label{newHsymcurlfemdof5} 
\end{align}
\end{subequations}
where 
\begin{align*}
\mathbb B_{k+1}^{tt}(F;\mathbb S) &:= \{ \bs \tau\in \mathbb P_{k+1}(F;\mathbb S): \bs\tau(\texttt{v}) = 0 \textrm{  for }\texttt{v}\in\Delta_0(F), \bs t^{\intercal}\bs\tau\bs t|_{\partial F} = 0\}, \\
\mathbb B_{k+1}^{\div_F}(F) &:= \{ \bs v\in \mathbb P_{k+1}(F;\mathbb R^2): \bs v\cdot\bs n_{F,e}|_{\partial F} = 0\}, \\
\mathbb B_{k+1}(\sym\curl, T;\mathbb T)&:=\{  
\bs\tau\in\mathbb P_{k+1}(T;\mathbb T):(\boldsymbol n\cdot \boldsymbol \tau\times\boldsymbol n)|_{\partial T}=0, \\
& \qquad\qquad\qquad\qquad\quad\;\; (\boldsymbol n\times \sym(\boldsymbol\tau\times\boldsymbol n)\times\boldsymbol n)|_{\partial T}=0\}.
\end{align*}
Characterization of $\mathbb B_{k+1}^{\div_F}(F)$ can be found  in~\cite[Lemma 4.2]{ChenHuang2021Geometric} and $\mathbb B_{k+1}(\sym\curl, T;\mathbb T)$ in~\cite[Lemma 5.7]{Chen;Huang:2020Finite}. In particular, we know $\dim \mathbb B_{k+1}(\sym\curl, T;\mathbb T) =  \frac{1}{3}(4k^3+6k^2-10k)$ and $\dim \mathbb B_{k+1}^{\div_F}(F)  = 2 \dim\mathbb P_{k+1}(F) - 3\times (k+2) = k^2 + 2k$. The bubble space $\dim \mathbb B_{k+1}^{tt}(F;\mathbb S) = 3 \dim\mathbb P_{k+1}(F) - 3 \times 3 - 3\times k = \frac{3}{2} (k^2+3k)$.




\begin{lemma}
The DoFs~\eqref{newHsymcurlfemdof} are unisolvent for the space $\mathbb P_{k+1}(T;\mathbb T)$ for $k\geq 0$.
\end{lemma}
\begin{proof}
The number of DoFs~\eqref{newHsymcurlfemdof} is 
$$
4\times 8+6\times 2k+4\times \left(\frac{3}{2}(k^2+3k)+ (k^2+2k)\right)+\frac{1}{3}(4k^3+6k^2-10k)=8{k+4\choose 3},
$$
which equals $\dim\mathbb P_{k+1}(T;\mathbb T)$.

Assume $\bs\tau\in\mathbb P_{k+1}(T;\mathbb T)$ and all the DoFs~\eqref{newHsymcurlfemdof} vanish.
Clearly $(\boldsymbol{n}_i^{\intercal}\boldsymbol \tau\boldsymbol{t})|_e=0$ follows from the vanishing DoFs~\eqref{newHsymcurlfemdof1}-\eqref{newHsymcurlfemdof2} for $e\in\Delta_1(T)$ and $i=1,2$.
Notice that
for $e\in\Delta_1(T)$ being an edge of face $F\in\partial T$, we have
\begin{equation}\label{eq:20231129}
\bs n_{F,e}^{\intercal}\sym(\boldsymbol\tau\times\boldsymbol n_F)\bs n_{F,e}=\boldsymbol n_{F,e}^{\intercal}\boldsymbol\tau\bs t,\quad \bs n_{F}^{\intercal}(\boldsymbol\tau\times\boldsymbol n_F)\bs n_{F,e}=\boldsymbol n_{F}^{\intercal}\boldsymbol\tau\bs t.
\end{equation}
Hence $(\boldsymbol n\times \sym(\boldsymbol\tau\times\boldsymbol n)\times \boldsymbol n)|_F\in\mathbb B_{k+1}^{tt}(F;\mathbb S)$ and $(\boldsymbol n\cdot \boldsymbol\tau\times\boldsymbol n)|_F\in\mathbb B_{k+1}^{\div_F}(F)$ for $F\in\partial T$. Then we get from the vanishing DoFs~\eqref{newHsymcurlfemdof3}-\eqref{newHsymcurlfemdof4} that $\bs\tau\in\mathbb B_{k+1}(\sym\curl, T;\mathbb T)$, which together with the vanishing DoF~\eqref{newHsymcurlfemdof5} yields $\bs\tau=0$.
\end{proof}

The finite element space $\Sigma_{k+1}^{\sym\curl}$ is defined as follows
\begin{align*}
\Sigma_{k+1}^{\sym\curl}:=\{\boldsymbol \tau\in L^2(\Omega;\mathbb T):\,& \boldsymbol \tau|_T\in\mathbb P_{k+1}(T;\mathbb T) \textrm{ for each } T\in\mathcal T_h, \textrm{  } \\
&\textrm{ all the DoFs~\eqref{newHsymcurlfemdof} are single-valued} \}.    
\end{align*}
DoFs~\eqref{newHsymcurlfemdof1}-\eqref{newHsymcurlfemdof2} on $e\in\Delta_1(T)$ determine $(\boldsymbol{n}_i^{\intercal}\boldsymbol \tau\boldsymbol{t})|_e$. By~\eqref{eq:20231129}, $(\boldsymbol{n}_i^{\intercal}\boldsymbol \tau\boldsymbol{t})|_e$ and~\eqref{newHsymcurlfemdof3} determine $(\boldsymbol n\times \sym(\boldsymbol\tau\times\boldsymbol n)\times \boldsymbol n)|_F$, and $(\boldsymbol{n}_i^{\intercal}\boldsymbol \tau\boldsymbol{t})|_e$ and~\eqref{newHsymcurlfemdof4} determine $(\boldsymbol n\cdot \boldsymbol\tau\times\boldsymbol n)|_F$. Therefore, $\Sigma_{k+1}^{\sym\curl}\subset H(\sym\curl,\Omega;\mathbb T)$ by the characterization of traces of functions in $H(\sym\curl,\Omega;\mathbb T)$ given in \cite{Chen;Huang:2020Finite}.


\begin{theorem}\label{thm:divdivcomplexes1}
Assume $\Omega$ is a bounded and topologically trivial Lipschitz domain in $\mathbb R^3$.
The finite element  $\div\div$ complex
\begin{equation}\label{eq:divdivcomplex3dfem}
{\bf RT}\xrightarrow{\subset} V_{k+2}^H\xrightarrow{\dev\grad}\Sigma_{k+1}^{\sym\curl}\xrightarrow{\sym\curl} \Sigma_{k,{\rm new}}^{\div\div} \xrightarrow{\div{\div}} V^{-1}_{k-2}\xrightarrow{}0, \; \text{ for } k\geq 3,
\end{equation}
is exact. Similarly, the finite element  $\div\div$ complex
\begin{equation*}
{\bf RT}\xrightarrow{\subset} V_{k+2}^H\xrightarrow{\dev\grad} \Sigma_{k+1}^{\sym\curl}\xrightarrow{\sym\curl} \Sigma_{k^+}^{\div\div} \xrightarrow{\div{\div}} V^{-1}_{k-1}\xrightarrow{}0, \quad \text{ for } k\geq 2,
\end{equation*}
involving Raviart-Thomas type space $\Sigma_{k^+}^{\div\div}$ is exact. 
\end{theorem}
\begin{proof}
The proof of two complexes is similar. So we focus on~\eqref{eq:divdivcomplex3dfem}. 

Clearly~\eqref{eq:divdivcomplex3dfem} is complex. 
We have proved the $\div\div$ operator is surjective. For $\bs \tau \in\ker(\sym\curl)\cap \Sigma_{k+1}^{\sym\curl}$, there exists a $v\in H^1(\Omega)$ s.t. $\dev \grad v = \bs \tau$. As $\bs \tau$ is piecewise polynomial, so is $v$. And the continuity of $\bs \tau$ at vertices implies $v$ is $C^1$ at vertices. Therefore we verified $\bs \tau \in\ker(\sym\curl)\cap \Sigma_{k+1}^{\sym\curl} = \dev \grad V^{H}_{k+2}$.

It remains to verify $\Sigma_{k,{\rm new}}^{\div\div}\cap\ker(\div\div)=\sym\curl\Sigma_{k+1}^{\sym\curl}$ by dimension count. 
It is easy to show the constraints $ [\tr_e(\bs \tau)]|_e = 0 \text{ for all } e\in \mathring{\mathcal E}_h$ are linearly independent. Therefore 
\begin{align*}    
\dim\Sigma_{k,{\rm new}}^{\div\div}&=\dim\Sigma_{k}^{\div\div-}-(k+1)|\mathring{\mathcal E}_h| \\
&=(k+1)(k^2+k+2)|\mathcal T_h|+(k+1)^2|\mathcal F_h|-(k+1)|\mathring{\mathcal E}_h|.
\end{align*}
Hence we have
\begin{align*}
&\dim(\Sigma_{k,{\rm new}}^{\div\div}\cap\ker(\div\div))=\dim\Sigma_{k,{\rm new}}^{\div\div}-\dim V^{-1}_{k-2}(\mathcal T_h)\\
= &\frac{1}{6}(k+1)(5k^2+7k+12)|\mathcal T_h| +(k+1)^2|\mathcal F_h|-(k+1)|\mathring{\mathcal E}_h|.
\end{align*}
While
\begin{align*}
&\dim\sym\curl\Sigma_{k+1}^{\sym\curl}=
\dim\Sigma_{k+1}^{\sym\curl} -\dim V_{k+2}^H+\dim {\bf RT} \\
=&\frac{1}{6}(k-1)k(5k+17)|\mathcal T_h| + (k^2+5k)|\mathcal F_h| - (k-3)|\mathcal E_h|-4| \mathcal V_h |+ 4.
\end{align*}
Then
\begin{align*}
&\quad \dim(\Sigma_{k,{\rm new}}^{\div\div}\cap\ker(\div\div))-\dim\sym\curl\Sigma_{k+1}^{\sym\curl}\\
&=(6k+2)|\mathcal T_h|-(3k-1)|\mathcal F_h|-(k+1)|\mathring{\mathcal E}_h|+(k-3)|\mathcal E_h| + 4|\mathcal V_h|-4\\
&=k(6|\mathcal T_h|-3|\mathcal F_h|+|\mathcal E_h^{\partial}|) + 2|\mathcal T_h|+|\mathcal F_h|+|\mathcal E_h^{\partial}|-4|\mathcal E_h|+4|\mathcal V_h|-4.
\end{align*}
By the relation $4|\mathcal T_h|=2|\mathcal F_h|-|\mathcal F_h^{\partial}|$ and $3|\mathcal F_h^{\partial}|=2|\mathcal E_h^{\partial}|$,
\begin{align*}
6|\mathcal T_h|-3|\mathcal F_h|+|\mathcal E_h^{\partial}|=-\frac{3}{2}|\mathcal F_h^{\partial}|+|\mathcal E_h^{\partial}|=0.
\end{align*}
This together with the Euler's formula $|\mathcal V_h|-|\mathcal E_h|+|\mathcal F_h|-|\mathcal T_h|=1$ yields
\begin{align*}
2|\mathcal T_h|+|\mathcal F_h|+|\mathcal E_h^{\partial}|-4|\mathcal E_h|+4|\mathcal V_h|-4&=-4|\mathcal T_h|+4|\mathcal F_h|-4|\mathcal E_h|+4|\mathcal V_h|-4 =0.
\end{align*}
Combining the last three identities gives 
$$
\dim(\Sigma_{k,{\rm new}}^{\div\div}\cap\ker(\div\div))=\dim(\sym\curl\Sigma_{k+1}^{\sym\curl}).
$$
Therefore, $\Sigma_{k,{\rm new}}^{\div\div}\cap\ker(\div\div)=\sym\curl\Sigma_{k+1}^{\sym\curl}$.
\end{proof}

\subsubsection{Finite element complexes starting from Lagrange element}
We present finite element divdiv complexes with the lowest smoothness in three dimensions.

We start from the vectorial Lagrange element space $V_{k+2}^L$. 
Define the $H(\sym\curl, \Omega;\mathbb T)$-conforming space with the lowest smoothness
$$
\overline{\Sigma}_{k+1}^{\sym\curl}:=\{\boldsymbol \tau\in H(\sym\curl,\Omega;\mathbb T): \boldsymbol \tau|_T\in\mathbb P_{k+1}(T;\mathbb T)\textrm{ for each } T\in\mathcal T_h\}.
$$
Although $\overline{\Sigma}_{k+1}^{\sym\curl}$ exists, it is hard to give local DoFs. Notice that $\Sigma_{k+1}^{\sym\curl}\subseteq \overline{\Sigma}_{k+1}^{\sym\curl}$.

\begin{theorem}\label{thm:divdivcomplexes2}
Assume $\Omega$ is a bounded and topologically trivial Lipschitz domain in $\mathbb R^3$.
The finite element  $\div\div$ complexes
\begin{equation}\label{eq:divdivcomplex3dfem1}
{\bf RT}\xrightarrow{\subset} V_{k+2}^L\xrightarrow{\dev\grad}\overline{\Sigma}_{k+1}^{\sym\curl}\xrightarrow{\sym\curl} \Sigma_{k,{\rm new}}^{\div\div} \xrightarrow{\div{\div}} V^{-1}_{k-2}\xrightarrow{}0,  \text{ for } k\geq 3,
\end{equation}
and
\begin{equation}\label{eq:divdivcomplex3dfem2}
{\bf RT}\xrightarrow{\subset} V_{k+2}^L\xrightarrow{\dev\grad}\overline{\Sigma}_{k+1}^{\sym\curl}\xrightarrow{\sym\curl} \Sigma_{k^+}^{\div\div} \xrightarrow{\div{\div}} V^{-1}_{k-1} \xrightarrow{}0, \text{ for } k\geq 2,
\end{equation}
are exact. 
\end{theorem}
\begin{proof}
By the similarity of two complexes, we focus on the exactness of complex~\eqref{eq:divdivcomplex3dfem1}.

By the exactness of complex~\eqref{eq:divdivcomplex3dfem}, we have 
$$
\div\div\Sigma_{k,{\rm new}}^{\div\div} = V^{-1}_{k-2}, \quad \sym\curl\Sigma_{k+1}^{\sym\curl}=\Sigma_{k,{\rm new}}^{\div\div}\cap\ker(\div\div).
$$
Noting that $\Sigma_{k+1}^{\sym\curl}\subseteq\overline{\Sigma}_{k+1}^{\sym\curl}$, it follows
$$
\sym\curl\Sigma_{k+1}^{\sym\curl}\subseteq\sym\curl\overline{\Sigma}_{k+1}^{\sym\curl}\subseteq \Sigma_{k,{\rm new}}^{\div\div}\cap\ker(\div\div).
$$
Hence $\sym\curl\overline{\Sigma}_{k+1}^{\sym\curl}=\Sigma_{k,{\rm new}}^{\div\div}\cap\ker(\div\div)$.

Clearly $\dev\grad V_{k+2}^L\subseteq(\overline{\Sigma}_{k+1}^{\sym\curl}\cap\ker(\sym\curl))$. On the other side, for $\boldsymbol{\tau}\in\overline{\Sigma}_{k+1}^{\sym\curl}\cap\ker(\sym\curl)$, there exists $\boldsymbol{v}\in H^1(\Omega;\mathbb R^3)$ satisfying $\boldsymbol{\tau}=\sym\curl\boldsymbol{v}$. On each tetrahedron $T\in\mathcal{T}_h$, $\sym\curl(\boldsymbol{v}|_T)\in\mathbb P_{k+1}(T;\mathbb T)$, then $\boldsymbol{v}|_T\in\mathbb P_{k+2}(T;\mathbb R^3)$ and $\boldsymbol{v}\in V_{k+2}^L$. Therefore complex~\eqref{eq:divdivcomplex3dfem1} is exact.
\end{proof}

\subsubsection{Lower order finite element divdiv complexes} 
The previous divdiv complexes did not cover the case $k = 0, 1$. In this subsection, we consider $k=1$ and refer to Section \ref{sec:distributivedivdiv} for $k=0$ in the distributional sense.
For the $H(\sym\curl, \Omega;\mathbb T)$-conforming finite element,
we take the space of shape functions as 
$$
\Sigma_{2^+}(T;\mathbb T):=\mathbb P_{2}(T;\mathbb T)\oplus(\boldsymbol{x}\otimes(\boldsymbol{x}\times\mathbb H_1(T;\mathbb R^3))),
$$
whose dimension is $80+8=88$.
Since $\sym\curl(\boldsymbol{x}\otimes\boldsymbol{v})=\sym(\boldsymbol{x}\otimes\curl\boldsymbol{v})$ for $\boldsymbol{v}\in H^1(T;\mathbb R^3)$, we have
$
\sym\curl(\boldsymbol{x}\otimes (\boldsymbol{x}\times\mathbb H_1(T;\mathbb R^3)))=\sym(\boldsymbol{x}\otimes\curl \mathbb H_2(T;\mathbb R^3))
$,
which means $\dim\sym\curl(\boldsymbol{x}\otimes (\boldsymbol{x}\times\mathbb H_1(T;\mathbb R^3)))=\dim\curl \mathbb H_2(T;\mathbb R^3)=8$, hence $\sym\curl$ is injective on $\boldsymbol{x}\otimes (\boldsymbol{x}\times\mathbb H_1(T;\mathbb R^3))$.

The degrees of freedom are given by
\begin{subequations}\label{lowerHsymcurlfemdof}
\begin{align}
\boldsymbol \tau(\delta), & \quad\delta\in \Delta_0(T), \bs \tau\in \mathbb T, \label{lowerHsymcurlfemdof1}\\
(\boldsymbol{n}_i^{\intercal}\boldsymbol \tau\boldsymbol{t}, q)_e, & \quad q\in\mathbb P_{0}(e), e\in\Delta_1(T), i=1,2,\label{lowerHsymcurlfemdof2}\\
(\boldsymbol n\times \sym(\boldsymbol\tau\times\boldsymbol n)\times \boldsymbol n, \boldsymbol q)_F, &\quad\boldsymbol q\in\mathbb B_{2^+}^{tt}(F;\mathbb S), F\in\partial T,\label{lowerHsymcurlfemdof3}\\
(\boldsymbol n\cdot \boldsymbol\tau\times\boldsymbol n, \boldsymbol q)_F, & \quad\boldsymbol q\in \mathbb B^{\div_F}_{2}(F),  F\in\partial T,\label{lowerHsymcurlfemdof4}
\end{align}
\end{subequations}
where 
\begin{align*}
\mathbb B_{2^+}^{tt}(F;\mathbb S) := \{& \bs \tau\in \mathbb P_{2}(F;\mathbb S)+(\boldsymbol{x}\times \boldsymbol{n})\otimes (\boldsymbol{x}\times \boldsymbol{n})\mathbb P_1(F): \\
&\qquad \bs\tau(\texttt{v}) = 0 \textrm{  for }\texttt{v}\in\Delta_0(F), 
\bs t^{\intercal}\bs\tau\bs t|_{\partial F} = 0\}.
\end{align*}
By Section 3 in~\cite{ChenHuang2020}, $\dim\mathbb B_{2^+}^{tt}(F;\mathbb S)=8$. Recall that $\dim \mathbb B_{2}^{\div_F}(F)  = 3$.

\begin{lemma}
The DoFs~\eqref{lowerHsymcurlfemdof} are unisolvent for the space $\Sigma_{2^+}(T;\mathbb T)$.
\end{lemma}
\begin{proof}
The number of DoFs~\eqref{lowerHsymcurlfemdof} is 
$$
4\times 8+6\times 2+4\times 8+4\times 3=88=\dim\Sigma_{2^+}(T;\mathbb T).
$$

Assume $\bs\tau\in\Sigma_{2^+}(T;\mathbb T)$ and all the DoFs~\eqref{lowerHsymcurlfemdof} vanish.
Notice that $(\boldsymbol{n}_i^{\intercal}\boldsymbol \tau\boldsymbol{t})|_e\in \mathbb P_{2}(e)$ for $e\in\Delta_1(T)$, and $(\boldsymbol n\cdot \boldsymbol\tau\times\boldsymbol n)|_F\in\mathbb P_{2}(F;\mathbb R^2)$ and $(\boldsymbol n\times \sym(\boldsymbol\tau\times\boldsymbol n)\times \boldsymbol n)|_F\in \mathbb P_{2}(F;\mathbb S)+(\boldsymbol{x}\times \boldsymbol{n})\otimes (\boldsymbol{x}\times \boldsymbol{n})\mathbb P_1(F)$ for $F\in\partial T$. Hence the vanishing DoFs~\eqref{lowerHsymcurlfemdof} imply $\bs\tau\in\Sigma_{2^+}(T;\mathbb T)\cap\mathbb B_{3}(\sym\curl, T;\mathbb T)$. By Theorem 5.12 in~\cite{Chen;Huang:2020Finite} and $\sym\curl\bs\tau\in\Sigma_{1^{++}}(T;\mathbb S)$, we get $\sym\curl\bs\tau=0$. Thus, $\bs\tau=\dev\grad\bs q$ with $\bs q\in\mathbb P_3(T;\mathbb R^3)$ satisfying $\bs q|_{\partial T}=0$. Therefore, $\bs q=0$ and $\bs\tau=0$.
\end{proof}

Define $H(\sym\curl)$-conforming finite element spaces as follows
\begin{align*}
\Sigma_{2^+}^{\sym\curl}&:=\{\boldsymbol \tau\in L^2(\Omega;\mathbb T): \boldsymbol \tau|_T\in\Sigma_{2^+}(T;\mathbb T) \textrm{ for each } T\in\mathcal T_h, \textrm{  } \\
&\qquad\qquad\qquad\qquad\;\textrm{ all the DoFs~\eqref{lowerHsymcurlfemdof} are single-valued} \},    \\
\overline{\Sigma}_{2^+}^{\sym\curl}&:=\{\boldsymbol \tau\in H(\sym\curl,\Omega;\mathbb T): \boldsymbol \tau|_T\in\Sigma_{2^+}(T;\mathbb T)\textrm{ for each } T\in\mathcal T_h\}.
\end{align*}
Clearly, $\Sigma_{2^+}^{\sym\curl}\subset H(\sym\curl,\Omega;\mathbb T)$, and
$\dim\Sigma_{2^+}^{\sym\curl}=11|\mathcal F_h|+2|\mathcal E_h|+8|\mathcal V_h|$.

Applying the argument in Theorem~\ref{thm:divdivcomplexes1} and Theorem~\ref{thm:divdivcomplexes2}, we have the following lower order finite element divdiv complexes.
\begin{theorem}\label{thm:divdivcomplexesk1}
Assume $\Omega$ is a bounded and topologically trivial Lipschitz domain in $\mathbb R^3$.
The finite element $\div\div$ complexes
\begin{equation*}
{\bf RT}\xrightarrow{\subset} V_{3}^H\xrightarrow{\dev\grad} \Sigma_{2^+}^{\sym\curl}\xrightarrow{\sym\curl} \Sigma_{1^{++}}^{\div\div} \xrightarrow{\div{\div}} V^{-1}_{1}\xrightarrow{}0,
\end{equation*}
\begin{equation*}
{\bf RT}\xrightarrow{\subset} V_{3}^L\xrightarrow{\dev\grad}\overline{\Sigma}_{2^+}^{\sym\curl}\xrightarrow{\sym\curl} \Sigma_{1^{++}}^{\div\div} \xrightarrow{\div{\div}} V^{-1}_{1}\xrightarrow{}0
\end{equation*}
are exact. 
\end{theorem}

\subsection{Distributional finite element divdiv complexes}\label{sec:distributivedivdiv}
With the weak $\div\div_w$ operator, we can construct the distributional finite element divdiv complexes.
We first present finite element discretization of the distributional divdiv complex
\begin{align*}
\resizebox{1.0\hsize}{!}{$
{\bf RT}\xrightarrow{\subset} H^1(\Omega;\mathbb R^3)\xrightarrow{\dev\grad} H(\sym\curl,\Omega;\mathbb T)\xrightarrow{\sym\curl} L^2(\Omega;\mathbb S) \xrightarrow{\div{\div}} H^{-2}(\Omega)\xrightarrow{}0.
$}
\end{align*}

\begin{theorem}\label{th:distributivecomplex}
Assume $\Omega$ is a bounded and topologically trivial Lipschitz domain in $\mathbb R^3$. The following complex
\begin{equation}\label{eq:distribut2divdivcomplex3dfemkgeq3}
{\bf RT}\xrightarrow{\subset} V_{k+2}^H\xrightarrow{\dev\grad} \Sigma_{k+1}^{\sym\curl}\xrightarrow{\sym\curl} \Sigma_{k,r}^{-1} \xrightarrow{\div{\div}_w} \mathring{M}^{-1}_{r,k-1, k, k}\xrightarrow{}0, \text{ for } k\geq 1,
\end{equation}
is exact.
%
\end{theorem}

%
\begin{proof}
The proof is similar to that for Theorem~\ref{thm:divdivcomplexes1}. The only difference is 
%
to verify $\dim \sym \curl \Sigma_{k+1}^{\sym\curl}  = \dim \ker(\div\div_w)\cap  \Sigma_{k}^{-1}$ by dimension count:
$$
\dim \ker(\div\div_w)\cap  \Sigma_{k}^{-1} = \dim  \Sigma_{k}^{-1} - \dim \mathring{M}^{-1}_{k-2,k-1, k, k} = \dim  \Sigma_{k,\rm new}^{\div\div} - \dim V_{k-2}^{-1}.
$$ 
%
\end{proof}

By dimension count and the structure of the enrichment, we have two more complexes for $k = 0, 1$. 

\begin{proposition}
For $k = 1$, the following complex
\begin{equation*}
{\bf RT}\xrightarrow{\subset} V_{3}^H\xrightarrow{\dev\grad} \Sigma_{2^+}^{\sym\curl}\xrightarrow{\sym\curl} \Sigma_{1^{++}}^{-1} \xrightarrow{\div{\div}_w} \mathring{M}^{-1}_{1, 1, 1, 1}\xrightarrow{}0
\end{equation*}
is also exact.
For $k = 0$, the following complex
\begin{equation*}
{\bf RT}\xrightarrow{\subset} V_{2}^L\xrightarrow{\dev\grad} \overline{\Sigma}_{1}^{\sym\curl}\xrightarrow{\sym\curl} \Sigma_{0}^{-1} \xrightarrow{\div{\div}_w} \mathring{M}^{-1}_{\cdot, \cdot, 0, 0}\xrightarrow{}0
\end{equation*}
is exact.
\end{proposition}

We can define $\overline{\Sigma}_{k,r}^{\div\div}$, for $k=0,1,2, r \leq 0$,
$$
\overline{\Sigma}_{k,r}^{\div\div} = \{ \bs \tau \in H(\div\div, \Omega; \mathbb S): \bs \tau\mid_T\in \Sigma_{k,r}(T;\mathbb S)\}.
$$
Although local DoFs cannot be given for space $\overline{\Sigma}_{k}^{\div\div}, k=0,1,2$, a discretization of the biharmonic equation can be obtained by the hybridization.  
For example, $\overline{\Sigma}_{0}^{\div\div} = \ker(\div\div_w)\cap \Sigma_{0}^{-1}$ is defined by applying the following constraints to $\Sigma_0^{-1}$
\begin{equation*}
[\tr_e(\bs \tau)]|_e = 0 \;\; \textrm{ for } e\in \mathring{\mathcal E}_h,\quad [\bs n^{\intercal}\bs \tau\bs n]|_F = 0\;\; \textrm{ for } F\in \mathring{\mathcal F}_h.
\end{equation*}
By counting the dimension of $\overline{\Sigma}_0^{\div\div}$, these constraints are linearly independent. 

\begin{corollary}
Both conforming finite element $\div\div$ complexes~\eqref{eq:divdivcomplex3dfem1} and~\eqref{eq:divdivcomplex3dfem2} are exact for all $k = 0,1,2$ using space $\overline{\Sigma}_{k,r}^{\div\div}$ to replace $\Sigma_{k,\rm new}^{\div\div}$ or $\Sigma_{k^+}^{\div\div}$.
\end{corollary}

\begin{remark}\label{rm:moredivdivcomplex}\rm 
 The first half of complexes~\eqref{eq:distribut2divdivcomplex3dfemkgeq3} can be replaced by 
 $$
{\bf RT}\xrightarrow{\subset} V_{k+2}^L\xrightarrow{\dev\grad} \overline{\Sigma}_{k+1}^{\sym\curl}\xrightarrow{\sym\curl} \cdots \quad \text{ for } k\geq 0.
$$
\end{remark}

\begin{remark}\rm
Recall that we can identity non-conforming VEM space $Q_M:  \mathring{V}_{k+2}^{\rm VEM} \to \mathring{M}^{-1}_{r,k-1, k, k}$ through $Q_M$. Then we can rewrite the second half of complexes~\eqref{eq:distribut2divdivcomplex3dfemkgeq3} as 
 $$
\cdots \xrightarrow{\sym\curl} \Sigma_{k,r}^{-1} \xrightarrow{Q_M^{-1}\div{\div}_w} \mathring{V}_{k+2}^{\rm VEM}\xrightarrow{}0  \quad \text{ for } k\geq 0.
$$
\end{remark}

When some partial continuity is imposed on $\Sigma_k^{-1}$, we can simplify the last space. For example, consider the normal-normal continuous element $\Sigma_{k,r}^{\rm nn}$ by asking DoFs on $\bs n^{\intercal}\bs \tau \bs n$ are single valued, then there is no need of Lagrange multiplier $u_n$. 
The corresponding divdiv complexes are still exact as we only reduce the range space of $\div\div_w$; see the $\,\widetilde{}\,$ operation introduced in \cite{ChenHuang2022}.
As a result of Theorem~\ref{th:distributivecomplex}, we will get finite element discretizations of the distributional divdiv complex
\begin{align*}
\resizebox{1.0\hsize}{!}{$
{\bf RT}\xrightarrow{\subset} H^1(\Omega;\mathbb R^3)\xrightarrow{\dev\grad} H(\sym\curl,\Omega;\mathbb T)\xrightarrow{\sym\curl} H^{-1}(\div\div, \Omega;\mathbb S) \xrightarrow{\div{\div}} H^{-1}(\Omega)\xrightarrow{}0.
$}
\end{align*}

\begin{theorem}\label{th:distributive1complex}
Assume $\Omega$ is a bounded and topologically trivial Lipschitz domain in $\mathbb R^3$. The following complexes
\begin{equation*}
{\bf RT}\xrightarrow{\subset} V_{k+2}^H\xrightarrow{\dev\grad} \Sigma_{k+1}^{\sym\curl}\xrightarrow{\sym\curl} \Sigma_{k,r}^{\rm nn} \xrightarrow{\div{\div}_w} \mathring{M}^{-1}_{r,k-1, \cdot, k}\xrightarrow{}0, \text{ for } k\geq 1, 
\end{equation*}
\begin{equation*}
{\bf RT}\xrightarrow{\subset} V_{k+2}^L\xrightarrow{\dev\grad} \overline{\Sigma}_{k+1}^{\sym\curl}\xrightarrow{\sym\curl} \Sigma_{k,r}^{\rm nn} \xrightarrow{\div{\div}_w} \mathring{M}^{-1}_{r,k-1, \cdot, k}\xrightarrow{}0, \text{ for } k\geq 0, 
\end{equation*}
are exact.
%
%
%
\end{theorem}


In two dimensions, the space $\mathring{M}^{-1}_{k-2, k-1, \cdot , k}$ can be identified as the Lagrange element $\mathring{V}^L_{k+1}$. The first distributional divdiv complex constructed in \cite{ChenHuHuang2018} can be written as
\begin{equation*}
{\bf RT}\xrightarrow{\subset} (V_{k+1}^L)^2\xrightarrow{\sym\curl} \Sigma_{k}^{\rm nn} \xrightarrow{\div{\div}_w} \mathring{V}^L_{k+1}\xrightarrow{}0, \quad \text{for } k\geq 0.
\end{equation*}
Complexes in Theorem \ref{th:distributive1complex} are its generalization to 3-D. 

We can further reduce the space of $u$ to $\mathring{M}^{-1}_{r,k-1, \cdot, \cdot}$ when the normal-normal continuity and $[\tr_{e}(\cdot)] = 0$ are both imposed and denoted by $\Sigma_{k,r}^{\rm nn, e}$ for $k\geq 1$. The space $\mathring{M}^{-1}_{r,k-1, \cdot, \cdot}$ can be identified as the $H^1$ non-conforming virtual element space \cite[Section 2.2]{ChenHuang2020a}
\begin{align*}
\mathring{V}_{k}^{1,\rm VEM}:=\big\{ & u \in L^2(\Omega): u|_T \in V_{k}^{1,\rm VEM}(T) \textrm{ for } T\in\mathcal T_h,   \\
& Q_{k-1,F}u \textrm{ is single-valued for } F\in\mathring{\mathcal{F}}_h, \textrm{ and vanish on boundary } \partial\Omega\big\},
\end{align*} 
where
$
V_{k}^{1,\rm VEM}(T):=\big\{  u \in H^1(T): \Delta u \in \mathbb P_{r}(T), \partial_nu|_F\in\mathbb P_{k-1}(F) \textrm{ for } F\in\partial T \big \}.
$
The DoFs of  $\mathring{V}_{k}^{1,\rm VEM}$ is given by $Q_Mu := \{Q_{r,T}u, Q_{k-1,F}u\}_{T\in \mathcal T_h, F\in \mathring{\mathcal{F}}_h}$ through which can be identified $\mathring{M}^{-1}_{r,k-1, \cdot, \cdot}$. 

We then obtain a divdiv complex ending with $\mathring{V}_{k}^{1,\rm VEM}$
\begin{equation*}
{\bf RT}\xrightarrow{\subset} V_{k+2}^H\xrightarrow{\dev\grad} \Sigma_{k+1}^{\sym\curl}\xrightarrow{\sym\curl} \Sigma_{k,r}^{\rm nn, e} \xrightarrow{Q_M^{-1}\div{\div}_w} \mathring{V}_{k}^{1,\rm VEM} \xrightarrow{}0, \text{ for } k\geq 1. 
\end{equation*}

\vskip 16pt

{\bf Acknowledgement}. We greatly appreciate the anonymous reviewers’ revising suggestions. In response to their feedback, we have changed the article's title to more accurately represent its content and have restructured the material to ensure a more coherent and smoother presentation.

\appendix 

\section{Uni-sovlence}\label{apdx:uinsol}
In this appendix, we give the uni-solvence of DoFs ~\eqref{eq:newdivdivSk2RT} for the space $\Sigma_{2^+}(T;\mathbb S)$.
First we recall a decomposition of a polynomial space and some barycentric calculus developed in~\cite{Chen;Huang:2020Finite}. 

\begin{lemma}\label{lem:xhessx}
Let $
\mathbb P_3(T)\backslash \mathbb P_1(T):=\{q\in\mathbb P_3(T): q(0)=0, \nabla q(0)=0\}.
$
The mapping $\boldsymbol x^{\intercal}\nabla^2\cdot\boldsymbol x:  \mathbb P_3(T)\backslash \mathbb P_1(T)\to \mathbb P_3(T)\backslash \mathbb P_1(T)$ is one-to-one. The mapping $\boldsymbol x^{\intercal}\cdot\boldsymbol x: \mathbb P_1(T;\mathbb S)\to \mathbb P_3(T)\backslash \mathbb P_1(T)$ is surjective.
\end{lemma}
\begin{proof}
By direct computation $\boldsymbol x^{\intercal}(\nabla^2 q) \boldsymbol x = r(r-1) q$ for $q\in \mathbb H_r(T), r\geq 0$.  
\end{proof}

\begin{lemma}
 We have the decomposition
 \begin{equation}\label{eq:decP1}
\mathbb P_1(T;\mathbb S) = \nabla^2 \mathbb P_3(T)\oplus (\ker(\boldsymbol x^{\intercal}\cdot\boldsymbol x)\cap \mathbb P_{1}(T;\mathbb S)),
\end{equation}
and consequently,
\begin{align*}
\dim \ker(\boldsymbol x^{\intercal}\cdot\boldsymbol x)\cap \mathbb P_{1}(T;\mathbb S) &=\dim\mathbb P_{1}(T;\mathbb S)-\dim\mathbb P_{3}(T)+\dim\mathbb P_{1}(T)\\
& =(d+1){d+1\choose2}-{d+3\choose3}+d+1=2{d+1\choose3}.
\end{align*}
\end{lemma}
\begin{proof}
By Lemma~\ref{lem:xhessx}, $\nabla^2 \mathbb P_3(T)\cap (\ker(\boldsymbol x^{\intercal}\cdot\boldsymbol x)\cap \mathbb P_{1}(T;\mathbb S))=0$, and
$$
\dim \ker(\boldsymbol x^{\intercal}\cdot\boldsymbol x)\cap \mathbb P_{1}(T;\mathbb S) =\dim\mathbb P_{1}(T;\mathbb S)-\dim\mathbb P_{3}(T)+\dim\mathbb P_{1}(T),
$$
which ends the proof.
\end{proof}


Define 
$$
\mathbb B_k^{\div}(T) :=  \mathbb{P}_k\left(T ; \mathbb{R}^d\right) \cap H_0(\operatorname{div} , T) =  \{ \bs v\in \mathbb P_k(T;\mathbb R^d): \bs v\cdot \bs n|_{\partial T} = 0\}.
$$
Recall the characterization of the div bubble function. 
\begin{lemma}[Lemma 4.2 in~\cite{ChenHuang2021Geometric}]For an edge $e = [\texttt{v}_i, \texttt{v}_j]$, let $b_e = \lambda_{i}\lambda_j$ be the quadratic edge bubble function and $\bs t_e$ be tangential vector of $e$. Then we have
\begin{equation}\label{eq:B2}
\mathbb B_2^{\div}(T) = {\rm span}\{ b_e(x) \bs t_e: e\in \Delta_1(T)\}.
\end{equation}
\end{lemma}
As a consequence $\dim \mathbb B_2^{\div}(T) = |\Delta_1(T)| = {d + 1 \choose 2}$ for a $d$-dimensional simplex $T$. 

We can easily show that $b_e(x) \boldsymbol{t}_e$ is an element of $\mathbb{B}_2^{\text{div}}(T)$. In order to establish~\eqref{eq:B2}, it is necessary to demonstrate that all quadratic divergence-free bubbles can be expressed in this form. See \cite{ChenHuang2021Geometric} for details.

\begin{lemma}\label{lm:bubbledof}
Let $\boldsymbol{v} \in \mathbb B_2^{\div}(T)$ satisfy $\operatorname{div} \boldsymbol{v}=0$ and for one $F\in \partial T$
\begin{equation}\label{eq:Fbubble}
(\Pi_{F} \boldsymbol{v}, \bs q)_F=0, \quad \bs q\in \mathbb B_2^{\div}(F). 
\end{equation}
Then $\bs v=0$.
\end{lemma}
\begin{proof}
Without loss of generality, take $F = F_d$. Then~\eqref{eq:Fbubble} implies that $\bs v$ does not contain edge bubbles on $F_d$, i.e., $\bs v = \sum_{i=0}^{d-1}c_i b_{e_{d,i}}\bs t_{d,i}$ with $c_i\in\mathbb R$ and $\bs t_{d, i} = \texttt{v}_i - \texttt{v}_d$. 
By direct computation and the fact $\nabla \lambda_i\cdot \bs t_{d,j} = \delta_{ij}$, we have
$$
\frac{|T|}{(d+2)(d+1)}c_i = (\bs v, \nabla \lambda_i)_T = - (\div \bs v, \lambda_i)_T = 0
$$
for each $i=0,1\ldots, d-1$. 
So $\bs v= 0$. 
\end{proof}

To facilitate the proof of unisolvence, we can select an intrinsic coordinate system. Let $\boldsymbol t_i:=\texttt{v}_i-\texttt{v}_0$ for $i=1,\ldots, d$. The set of tangential vectors $\{\boldsymbol t_1, \ldots, \boldsymbol t_d\}$ forms a basis of $\mathbb{R}^d$, and its dual basis is given by $\{ \nabla \lambda_1, \ldots, \nabla \lambda_d \}$. We have the property that $\nabla \lambda_i \cdot \boldsymbol{t}_j = \delta_{ij}$ for $i,j=1,\ldots, d$, where $\delta_{ij}$ is the Kronecker delta.
%
We can then express the symmetric tensor $\boldsymbol{\tau}$ as $\bs\tau=\sum_{i,j=1}^d\tau_{ij}\boldsymbol t_i \otimes \boldsymbol t_j$ using coefficients $\tau_{ij}$, which are computed as $\tau_{ij}=(\nabla\lambda_i)^{\intercal}\boldsymbol{\tau}(\nabla\lambda_j)$.

Since $\boldsymbol{\tau}$ is symmetric, we have that $\tau_{ij} = \tau_{ji}$ for $1\leq i,j\leq d$. Therefore, we can represent $\boldsymbol{\tau}$ as a symmetric matrix function $(\tau_{ij}(x))$ in this coordinate system.

\begin{theorem}\label{th:k=2}
The DoFs~\eqref{eq:newdivdivSk2RT} are unisolvent for the space $\Sigma_{2^+}(T;\mathbb S)$.
\end{theorem}
\begin{proof}
\step{1} {\it Dimension count.}
The number of DoF~\eqref{eq:newdivdivdofk2RT4} is
$$
{2\choose2}+{3\choose2}+\ldots+{d-1\choose2}={d\choose3},
$$
and the number of DoF~\eqref{eq:newdivdivdofk2RT5} is $\dim \ker(\boldsymbol x^{\intercal}\cdot\boldsymbol x)\cap \mathbb P_{1}(T;\mathbb S) = 2{d+1\choose3}$. 
Hence the total number of DoFs~\eqref{eq:newdivdivSk2RT} is
\begin{align*}    
&\quad {d+1\choose2}{d\choose2}+(d+1){d+1\choose2}+(d+1)d+{d\choose3} +2{d+1\choose3} \\
&= {d+1\choose2}{d+2\choose2}+d,
\end{align*}
which is exactly the dimension of $\Sigma_{2^+}(T;\mathbb S)$.

\smallskip

\step{2} {\it Consequence of vanishing DoFs.} Assume $\boldsymbol{\tau}\in \Sigma_{2^+}(T;\mathbb S) = \mathbb P_2(T;\mathbb S) \oplus \boldsymbol x\boldsymbol x^{\intercal}\mathbb H_{1}(T)$, and all the DoFs~\eqref{eq:newdivdivSk2RT} vanish. The vanishing DoFs~\eqref{eq:newdivdivdofk2RT1}-\eqref{eq:newdivdivdofk2RT3} imply the traces of $\bs \tau$ vanish
\begin{equation}\label{eq:newdivdivdofk2RT123}
\tr_1(\bs \tau)=0,\quad \tr_2(\bs \tau)=0,
\end{equation}
and
\begin{equation*}
\quad Q_{\mathcal N_f}(\bs \tau)=0 \textrm{ for } f\in\Delta_r(T), r = 0,\ldots, d-2.
\end{equation*}

Then apply the integration by parts~\eqref{eq:greenidentitydivdiv} and the fact $\div\div\bs \tau \in \mathbb P_1(T)$ to conclude $\div\div\bs \tau=0$ and consequently
$$
\bs \tau \in \mathbb P_2(T; \mathbb S), \quad (\bs \tau, \nabla^2v)_T = 0 \quad \forall~v\in H^2(T).
$$ 
Then the vanishing DoF~\eqref{eq:newdivdivdofk2RT5} and the decomposition~\eqref{eq:decP1} imply
\begin{equation}\label{eq:newdivdivdofk2RT51}
(\boldsymbol \tau, \boldsymbol q)_T =0 \quad\forall~\boldsymbol q\in \mathbb P_{1}(T;\mathbb S). 
\end{equation}

Recall that $\bs \tau$ is represented as a symmetric matrix function $(\tau_{ij}(x))$ in the coordinate $\{\boldsymbol t_1, \ldots, \boldsymbol t_d\}$. We are going to show $\tau_{ij} = 0$ for all $1\leq i\leq j\leq d$. As $\bs \tau$ is quadratic, being orthogonal to $\mathbb P_1(T;\mathbb S)$ is not enough to conclude $\bs \tau = 0$. More conditions will be derived from vanishing DoFs. 

\smallskip

\step{3} {\it Diagonal is zero.} By $\tr_1(\bs \tau)=0$, it follows
\[
\tau_{ii}|_{F_i}=|\nabla\lambda_i|^2\boldsymbol n_i^{\intercal}\bs\tau\boldsymbol n_i|_{F_i}=0, \quad i = 1,\ldots,d.
\]
For each $i=1,\ldots,d$, there exists $p_i \in \mathbb P_{1}(K)$ satisfying $\tau_{ii}=\lambda_ip_i$.
Taking $\boldsymbol q = p_i\boldsymbol n_i\boldsymbol n_i^{\intercal}$ in~\eqref{eq:newdivdivdofk2RT51} will produce 
\begin{equation*}
\tau_{ii}=0, \quad i = 1,\ldots,d.
\end{equation*}
Namely the diagonal of $\boldsymbol \tau$ is zero. Notice that the index $i=1,\ldots d$ not including $i=0$. Will use vanishing $\boldsymbol{n}_{F_0}^{\intercal}\boldsymbol{\tau}\boldsymbol{n}_{F_0}|_{F_0}$ in the last step. 

\smallskip

\step{4} {\it Off-diagonal: the last row/column}. 
By $Q_{\mathcal N_e}(\bs \tau)=0$ in~\eqref{eq:newdivdivdofk2RT123}, we have 
$$
\Pi_{F}(\bs\tau\boldsymbol n_F)\in \mathbb B_2^{\div}(F) \quad \text{ for each } F\in\partial T.
$$
As $\boldsymbol n_{F_i}^{\intercal}\boldsymbol \tau \boldsymbol n_{F_i} = 0$ in $T$ for $i = 1,\ldots,d$, it follows $\partial_{n_{F_i}}(\boldsymbol n_{F_i}^{\intercal}\boldsymbol \tau \boldsymbol n_{F_i})=0$, and $\tr_2(\bs \tau)=0$ becomes
\begin{equation}\label{eq:newdivdivdofk2RT31}
\div_{F_r}(\Pi_{F_r}(\boldsymbol\tau \boldsymbol n_{F_r}))|_{F_r}=0,\quad r= d, \ldots, 1.
\end{equation}
Again $r=0$ is not included in~\eqref{eq:newdivdivdofk2RT31}.

Consider $r=d$ in~\eqref{eq:newdivdivdofk2RT4}. As $\Pi_{F_d}(\bs\tau\boldsymbol n_{F_d})\in \mathbb B_2^{\div}(F_d)$ and 
$$(\Pi_{f_{0:d-2}}\boldsymbol\tau \boldsymbol n_{F_d}, \bs q)_{f_{0:d-2}}=0, \quad \bs q\in \mathbb B_2^{\div}(f_{0:d-2}),$$ applying Lemma \ref{lm:bubbledof} to ($d-1$)-dimensional simplex $F_d$, we conclude $(\Pi_{F_d}\boldsymbol\tau \boldsymbol n_{F_d})|_{F_d}=0$. Together with the vanishing normal-normal component, we have $\boldsymbol\tau \boldsymbol n_{F_d}|_{F_d}=0$.

Then there exists $\boldsymbol p\in \mathbb P_{1}(T;\mathbb R^d)$ such that $\boldsymbol\tau \boldsymbol n_{F_d}=\lambda_d\, \boldsymbol p$.
Take $\bs q=\sym(\boldsymbol p\otimes \boldsymbol n_{F_d})$ in~\eqref{eq:newdivdivdofk2RT51} to conclude $\boldsymbol\tau \boldsymbol n_{F_d}=0$ in $T$. That is the last column of the symmetric matrix representation of $\boldsymbol \tau$ is zero. 



\smallskip
\step{5} {\it Off-diagonal: the $r$-th row/column}. Assume we have proved the $\ell$-th columns are zero for $\ell > r$. By symmetry and vanishing normal-normal component $\bs n^{\intercal}_{F_{\ell}}\bs \tau \bs n_{F_r} = 0$ for $\ell\geq r$. Expand in the edge coordinate $\bs \tau \bs n_{F_{r}} = \sum_{i=1}^{r-1} p_i(x) \bs t_{i}$ with $p_i(x)\in\mathbb P_2(T)$. So
$$
\Pi_{F_{r}}\bs \tau \bs n_{F_{r}}|_{F_r} = \sum_{e\in \Delta_1 ( f_{0:r-1})} c_e b_e(x) \bs t_e\in \mathbb B_2^{\div}(F_{r})  \textrm{ with } c_e\in\mathbb R,
$$
which contains only the edge bubble corresponding to edges of simplex $f_{0:r-1}$. 
Notice that $\Pi_{F_{r}}\bs \tau \bs n_{F_{r}}|_{f_{0:r-2}}\in \mathbb B_2^{\div}(f_{0:r-2})$.
The vanishing~\eqref{eq:newdivdivdofk2RT4} on $f_{0: r-2}$ will further rule out the edge bubbles on $f_{0:r-2}$ and simplify to 
$$
\Pi_{F_{r}}\bs \tau \bs n_{F_{r}}|_{F_r} =  \sum_{i=0}^{r-2} c_i b_{e_{r-1,i}}(x) \bs t_{i,r-1}.
$$ 
Use $- (\div_{F_{r}} \Pi_{F_{r}}\bs \tau \bs n_{F_{r}}, \lambda_i)_{F_{r}} = \frac{|F_{r}|}{(d+1)d}c_i = 0$ to conclude $\Pi_{F_{r}}\bs \tau \bs n_{F_{r}}  |_{F_{r}}= 0$. Together with the vanishing normal-normal component, we have $\boldsymbol\tau \boldsymbol n_{F_{r}}|_{F_{r}}=0$. The rest to prove $\boldsymbol\tau \boldsymbol n_{F_{r}} = 0$ in $T$ is like Step 4. 

%


\smallskip

\step{6} {\it Entry $\tau_{12}$.}
Only one entry $\tau_{12}$ is left, i.e., 
$\bs\tau=2\tau_{12}\sym(\boldsymbol t_1\boldsymbol t_2^{\intercal})$. Multiplying $\bs\tau$ by $\nabla\lambda_0$ from both sides and restricting to $F_0$, we have
\[
\tau_{12}|_{F_0}=\frac{1}{2}|\nabla\lambda_0|^2(\boldsymbol{n}_{F_0}^{\intercal}\boldsymbol{\tau}\boldsymbol{n}_{F_0})|_{F_0}=0.
\]
Again there exists $p\in \mathbb P_{1}(K)$ satisfying $\tau_{12}=\lambda_0p$.
Taking $\boldsymbol q=\sym(\boldsymbol t_1\boldsymbol t_2^{\intercal})p$ in~\eqref{eq:newdivdivdofk2RT51} gives $\tau_{12}=0$. We thus have $\boldsymbol \tau =0$ and consequently prove the uni-solvence.
\end{proof}

\begin{corollary}\label{cor:ncfmdivdivSk2}
The DoFs
\begin{subequations}\label{eq:ncfmdivdivSk2}
\begin{align}
(\tr_e(\bs \tau), q)_e, &\quad q\in \mathbb P_2(e), e\in \Delta_{d-2}(T),\label{eq:ncfmdivdivdofk21}\\
(\bs n^{\intercal}\bs \tau \bs n, q )_F, &\quad q\in \mathbb P_2(F), F\in \partial T,\label{eq:ncfmdivdivdofk22}\\
( \tr_2(\bs \tau), q)_F, &\quad q\in\mathbb P_{1}(F)/\mathbb R, F\in \partial T,\label{eq:ncfmdivdivdofk23}\\
(\Pi_f\boldsymbol \tau\boldsymbol n_{F_r}, \boldsymbol q)_{f}, & \quad \boldsymbol q\in \mathbb B_2^{\div}(f), f= f_{0:r-2}\in \Delta_{r-2}(F_r), r= d,\ldots, 3,
\label{eq:ncfmdivdivdofk24}\\
(\boldsymbol \tau, \boldsymbol q)_T, &\quad \boldsymbol q\in  \ker(\boldsymbol x^{\intercal}\cdot\boldsymbol x)\cap \mathbb P_{1}(T;\mathbb S),\label{eq:ncfmdivdivdofk25}
\\
(\div\div\boldsymbol \tau, q)_T, &\quad  q\in  \mathbb P_{0}(T),\label{eq:ncfmdivdivdofk26}
\end{align}
\end{subequations}
are unisolvent for $\mathbb P_2(T;\mathbb S)$.
\end{corollary}
\begin{proof}
Compared with DoFs~\eqref{eq:newdivdivSk2RT} for $\Sigma_{2^+}(T;\mathbb S)$, the number of DoFs~\eqref{eq:ncfmdivdivSk2} equals $\dim\mathbb P_2(T;\mathbb S)$. Assume $\boldsymbol{\tau}\in\mathbb P_2(T;\mathbb S)$ and all the DoFs~\eqref{eq:ncfmdivdivSk2} vanish. By the vanishing DoFs~\eqref{eq:ncfmdivdivdofk21}-\eqref{eq:ncfmdivdivdofk23} and~\eqref{eq:ncfmdivdivdofk26}, we have $\tr_e(\bs \tau)=0$ for $e\in \Delta_{d-2}(T)$, $(\bs n^{\intercal}\bs \tau \bs n)|_F=0$ and $\tr_2(\bs \tau)|_F\in\mathbb P_0(F)$ for $F\in\partial T$, and $\div\div\boldsymbol \tau=0$. Apply~\eqref{eq:greenidentityP1} to get
\begin{equation*}
\sum_{F\in\partial T}(\tr_2(\boldsymbol{\tau}), v)_F=0,\quad v\in\mathbb P_1(T),
\end{equation*}
which implies $\tr_2(\boldsymbol{\tau})=0$. Finally, $\boldsymbol{\tau}=0$ follows from Theorem~\ref{th:k=2}.
\end{proof}

The finite element space defined by (A.7) is not $H(\div\div)$-conforming as $\tr_2(\tau)$ is not continuous. It will be used in the proof of norm equivalence in Appendix~\ref{apdx:normequiv}. 

\section{Norm Equivalence}\label{apdx:normequiv}
For $u\in \mathring{M}^{-1}_{r,k-1, k, k}$ with $k\geq0$, define a discrete $H^2$-norm: 
\begin{align*}
|u|_{2,h}^2 = &\sum_{T\in \mathcal T_h} \left ( h_T^{-4}\|Q_{r,T}u^{\rm CR} - u_0\|_{0,T}^2+\sum_{F\in \partial T}h_T^{-3}\| Q_{k-1,F}u^{\rm CR} - u_b\|_{0,F}^2\right ) \\
&+\sum_{T\in \mathcal T_h} \left (\sum_{F\in \partial T}h_T^{-1}\| \partial_{n_F}u^{\rm CR} - u_n\|_{0,F}^2 + \sum_{e\in \Delta_{d-2}(T)}h_T^{-2}\|Q_{k,e}u^{\rm CR} - u_e\|_{0,e}^2 \right ),
\end{align*}
where $u^{\rm CR} = I^{\rm CR}(Q_M^{-1}u)$ with $I^{\rm CR}$ being the interpolation operator to the nonconforming linear element space and $Q_M^{-1}$ is the bijection from $\mathring{M}^{-1}_{r,k-1, k, k}$ to $\mathring{V}_{k+2}^{\rm VEM}$.
When $k=0,1, r<0$, it is simplified to
\begin{equation*}
|u|_{2,h}^2 = \sum_{T\in \mathcal T_h} \left (\sum_{F\in \partial T}h_T^{-1}\| \partial_{n_F}u^{\rm CR} - u_n\|_{0,F}^2 + \sum_{e\in \Delta_{d-2}(T)}h_T^{-2}\|Q_{k,e}u^{\rm CR} - u_e\|_{0,e}^2 \right ).
\end{equation*}

\begin{lemma}\label{lm:normequivalence}
On the space $\mathring{M}^{-1}_{r,k-1, k, k}$, we have the norm equivalence
\begin{equation}
\label{eq:normequivH2h1}    
\|\nabla_w^2 u \|_0 \eqsim |u|_{2,h}, \quad u\in \mathring{M}^{-1}_{r,k-1, k, k}\quad \text{ for } k\geq0. 
\end{equation}
\end{lemma}
\begin{proof}
By~\eqref{eq:weakhess} and the Green's identity~\eqref{eq:greenidentitydivdiv}, for $\bs \tau\in\Sigma_{k,r}(T;\mathbb S)$ we have
\begin{align}
\notag
&(\nabla_w^2 u, \bs \tau )_T=  (u_0-Q_{r,T}u^{\rm CR}, (\div\div)_T \bs \tau)_T- (u_b-Q_{k-1,F}u^{\rm CR}, \tr_2(\bs \tau))_{\partial T} \\
\label{weakhessVCR}
&\; + (u_n\bs n_F\cdot\bs n-\partial_nu^{\rm CR},\bs n^{\intercal}\bs \tau\bs n)_{\partial T} + \sum_{e\in \Delta_{d-2}(T)}(u_e-Q_{k,e}u^{\rm CR}, [\boldsymbol n_{F,e}^{\intercal}\boldsymbol \tau \boldsymbol n_{\partial T}]|_e)_e. 
\end{align}
Then $\|\nabla_w^2 u \|_0 \lesssim |u|_{2,h}$ follows from the Cauchy-Schwarz inequality, and the inverse trace inequality. 

Next we prove $|u|_{2,h}\lesssim \|\nabla_w^2 u \|_0$.

\step{1} First consider $k=0$. By~\eqref{weakhessVEMProjcd} and the fact $Q_M^{-1}u\in\mathring{V}^{\rm MWX}_2$, $\nabla_w^2u = \nabla_h^2Q_M^{-1}u$. It follows from the norm equivalence and the error estimate of $I^{\rm CR}$ that
\begin{equation*}
|u|_{2,h}^2\lesssim \sum_{T\in\mathcal T_h}h_T^{-4}\|Q_M^{-1}u-u^{\rm CR}\|_{0,T}^2\lesssim \|\nabla_h^2Q_M^{-1}u\|_0^2=\|\nabla_w^2u\|_0^2.
\end{equation*}

\step{2} Next consider $k=1,2$, and $r=k-2$.
By the DoFs~\eqref{eq:ncfmdivdivSk2}, 
we can construct $\boldsymbol{\tau}\in \Sigma_{k,r}(T;\mathbb S)$ such that 
\begin{equation*}
 \begin{aligned}
[\boldsymbol n_{F,e}^{\intercal}\boldsymbol \tau \boldsymbol n_{\partial T}]|_e &= h_T^{-2}(u_e - u^{\rm CR})|_e, \qquad\qquad\quad\; e\in\Delta_{d-2}(T),\\
(\bs n^{\intercal}\bs \tau\bs n)|_F &=  h_T^{-1}(u_n\bs n_F\cdot\bs n - \partial_n u^{\rm CR} )|_F, \quad F\in\Delta_{d-1}(T), \\
(I-Q_{0,F})\tr_2(\boldsymbol{\tau})|_F &=  h_T^{-3}(u^{\rm CR}-u_b)|_F, \qquad\qquad\;\;\;\, F\in\Delta_{d-1}(T)\textrm{ if } k =  2,\\
\div\div_T\boldsymbol{\tau} &= h_T^{-4}(u_0-Q_{r,T}u^{\rm CR}), \qquad\quad\;\;\textrm{ if } k = 2,
\end{aligned}
\end{equation*}
and all the other DoFs in~\eqref{eq:ncfmdivdivSk2} vanish. 
By the norm equivalence and the scaling argument, we have
\begin{align*}
\|\boldsymbol{\tau}\|_{0,T}^2\lesssim {}&h_T^{-4}\|Q_{r,T}u^{\rm CR} - u_0\|_{0,T}^2 + h_T^{-3}\| Q_{k-1,F}u^{\rm CR} - u_b\|_{0,\partial T}^2  \\
& + h_T^{-1}\|u_n - \partial_{n_F} u^{\rm CR}\|_{0,\partial T}^2 + \sum_{e\in \Delta_{d-2}(T)}h_T^{-2}\|u^{\rm CR} - u_e\|_{0,e}^2.
\end{align*}
Substitude into~\eqref{weakhessVCR}, we get
\begin{align*}    
&h_T^{-4}\|Q_{r,T}u^{\rm CR} - u_0\|_{0,T}^2+h_T^{-3}\| Q_{k-1,F}u^{\rm CR} - u_b\|_{0,\partial T}^2 + h_T^{-1}\| \partial_{n_F}u^{\rm CR} - u_n\|_{0,\partial T}^2  \\
&\qquad\qquad+ \sum_{e\in \Delta_{d-2}(T)}h_T^{-2}\|u^{\rm CR} - u_e\|_{0,e}^2   = (\nabla_w^2 u, \bs\tau)_T \leq \|\nabla_w^2 u \|_{0,T}\|\bs\tau\|_{0,T}.
\end{align*}
We conclude $|u|_{2,h}\lesssim \|\nabla_w^2 u \|_0$ by combining the last two inequalities.

\step{3} Consider $k\geq2$ and $r\geq1$.
By the DoFs~\eqref{eq:newdivdivS} or~\eqref{eq:newdivdivSk2RT},
we can construct $\boldsymbol{\tau}\in \mathbb P_k(T; \mathbb S)$ such that 
\begin{equation}\label{eq:normequivalencetau}
 \begin{aligned}
[\boldsymbol n_{F,e}^{\intercal}\boldsymbol \tau \boldsymbol n_{\partial T}]|_e &= h_T^{-2}(u_e - u_0)|_e, \qquad\qquad\quad\; e\in\Delta_{d-2}(T),\\
(\bs n^{\intercal}\bs \tau\bs n)|_F &=  h_T^{-1}(u_n\bs n_F\cdot\bs n - \partial_n u_0 )|_F, \quad F\in\Delta_{d-1}(T), \\
\tr_2(\boldsymbol{\tau})|_F &=  h_T^{-3}(u_0-u_b)|_F, \qquad\qquad\quad F\in\Delta_{d-1}(T), \\
(\boldsymbol{\tau}, \boldsymbol{q})_T &= (\nabla_h^2 u_0, \boldsymbol{q})_T, \qquad\qquad\qquad\;\;\; \boldsymbol{q}\in \nabla^2\mathbb P_{r}(T),
\end{aligned}
\end{equation}
and all the other DoFs in~\eqref{eq:newdivdivS} and~\eqref{eq:newdivdivSk2RT} vanish. 
By the norm equivalence and the scaling argument, we have
\begin{align*}
\|\boldsymbol{\tau}\|_{0,T}^2\lesssim {}&\|\nabla_h^2 u_0\|_{0,T}^2 + h_T^{-3}\| Q_{k-1,F}u_0 - u_b\|_{0,\partial T}^2 + h_T^{-1}\|u_n - \partial_{n_F} u_0\|_{0,\partial T}^2 \\
& + \sum_{e\in \Delta_{d-2}(T)}h_T^{-2}\|u_0 - u_e\|_{0,e}^2.
\end{align*}
By~\eqref{eq:weakhess2} we get
\begin{align*}    
\|\nabla_h^2 u_0\|_{0,T}^2 &+ h_T^{-3}\|u_0 - u_b\|_{0,\partial T}^2 + h_T^{-1}\| \partial_{n_F}u_0 - u_n\|_{0,\partial T}^2  \\
&+ \sum_{e\in \Delta_{d-2}(T)}h_T^{-2}\|u_0 - u_e\|_{0,e}^2   = (\nabla_w^2 u, \bs\tau)_T \leq \|\nabla_w^2 u \|_{0,T}\|\bs\tau\|_{0,T}.
\end{align*}
Finally, we obtain $|u|_{2,h}  \lesssim \|\nabla_w^2 u \|_0 $ by combining the last two inequalities. 
\end{proof}

\begin{lemma}\label{lm:normequivalencevem}
We have the norm equivalence
\begin{equation}
\label{app:VEMnormequivH2}
\|\nabla_w^2 Q_Mv \|_0= \|Q_{\Sigma} \nabla_h^2 v\|_0 \eqsim \|\nabla_h^2 v\|_0, \quad v\in \mathring{V}_{k+2}^{\rm VEM}, k\geq 0.
\end{equation}
\end{lemma}
\begin{proof}
First~\eqref{app:VEMnormequivH2} is obviously true for $k=0$, since $\nabla_w^2 Q_M v = Q_{\Sigma} \nabla_h^2 v=\nabla_h^2 v$.
Then we focus on $k\geq1$.

By the norm equivalence~\eqref{eq:normequivH2h1}, it suffices to prove
\begin{equation}\label{eq:VEMnormequivH3}
|Q_{M}v|_{2,h} \eqsim \|\nabla_h^2 v\|_0, \quad v\in \mathring{V}_{k+2}^{\rm VEM}.
\end{equation}
By the definition of $|Q_{M}v|_{2,h}$ and $v^{\rm CR}$, and the norm equivalence on $V_{k+2}^{\rm VEM}(T)$, 
\begin{align*}
|Q_{M}v|_{2,h}^2 &= \sum_{T\in \mathcal T_h} h_T^{-4}\|Q_{r,T}(v^{\rm CR} - v)\|_{0,T}^2+\sum_{T\in \mathcal T_h}\sum_{e\in \Delta_{d-2}(T)}h_T^{-2}\|Q_{k,e}(v^{\rm CR} - v)\|_{0,e}^2 \\
&\;+\sum_{T\in \mathcal T_h}\sum_{F\in \partial T}(h_T^{-1}\| Q_{k,F}\partial_{n_F}(v^{\rm CR} - v)\|_{0,F}^2 + h_T^{-3}\| Q_{k-1,F}(v^{\rm CR} - v)\|_{0,F}^2) \\
&\eqsim\sum_{T\in \mathcal T_h} h_T^{-4}\|v^{\rm CR} - v\|_{0,T}^2.
\end{align*}
Therefore,~\eqref{eq:VEMnormequivH3} follows from the inverse inequality and the interpolation estimate of the nonconforming linear element.
\end{proof}


\bibliographystyle{abbrv}
\bibliography{paper,FEMcomplex,FEEC,DGlocal,WGlocal,VEM}

\begin{thebibliography}{10}

\bibitem{ArnoldBrezzi1985}
D.~N. Arnold and F.~Brezzi.
\newblock Mixed and nonconforming finite element methods: implementation,
  postprocessing and error estimates.
\newblock {\em RAIRO Mod\'el. Math. Anal. Num\'er.}, 19(1):7--32, 1985.

\bibitem{ArnoldFalkWinther2006}
D.~N. Arnold, R.~S. Falk, and R.~Winther.
\newblock Finite element exterior calculus, homological techniques, and
  applications.
\newblock {\em Acta Numer.}, 15:1--155, 2006.

\bibitem{ArnoldFalkWinther2009}
D.~N. Arnold, R.~S. Falk, and R.~Winther.
\newblock Geometric decompositions and local bases for spaces of finite element
  differential forms.
\newblock {\em Computer Methods in Applied Mechanics and Engineering},
  198(21-26):1660--1672, May 2009.

\bibitem{Arnold;Hu:2020Complexes}
D.~N. Arnold and K.~Hu.
\newblock Complexes from complexes.
\newblock {\em Found. Comput. Math.}, 2021.

\bibitem{BassiRebay1997}
F.~Bassi and S.~Rebay.
\newblock A high-order accurate discontinuous finite element method for the
  numerical solution of the compressible {N}avier-{S}tokes equations.
\newblock {\em J. Comput. Phys.}, 131(2):267--279, 1997.

\bibitem{BlumRannacher1980}
H.~Blum and R.~Rannacher.
\newblock On the boundary value problem of the biharmonic operator on domains
  with angular corners.
\newblock {\em Math. Methods Appl. Sci.}, 2(4):556--581, 1980.

\bibitem{BrennerSung2005}
S.~C. Brenner and L.-Y. Sung.
\newblock {$C\sp 0$} interior penalty methods for fourth order elliptic
  boundary value problems on polygonal domains.
\newblock {\em J. Sci. Comput.}, 22/23:83--118, 2005.

\bibitem{BrezziManziniMariniPietraEtAl2000}
F.~Brezzi, G.~Manzini, D.~Marini, P.~Pietra, and A.~Russo.
\newblock Discontinuous {G}alerkin approximations for elliptic problems.
\newblock {\em Numer. Methods Partial Differential Equations}, 16(4):365--378,
  2000.

\bibitem{ChenHuHuang2018}
L.~Chen, J.~Hu, and X.~Huang.
\newblock Multigrid methods for {H}ellan-{H}errmann-{J}ohnson mixed method of
  {K}irchhoff plate bending problems.
\newblock {\em J. Sci. Comput.}, 76(2):673--696, 2018.

\bibitem{ChenHuang2020}
L.~Chen and X.~Huang.
\newblock Finite elements for divdiv-conforming symmetric tensors.
\newblock {\em arXiv preprint arXiv:2005.01271}, 2020.

\bibitem{ChenHuang2020a}
L.~Chen and X.~Huang.
\newblock Nonconforming virtual element method for {$2m$}th order partial
  differential equations in {$\Bbb{R}^n$}.
\newblock {\em Math. Comp.}, 89(324):1711--1744, 2020.

\bibitem{ChenHuang2021Geometric}
L.~Chen and X.~Huang.
\newblock Geometric decompositions of div-conforming finite element tensors.
\newblock {\em arXiv preprint arXiv:2112.14351}, 2021.

\bibitem{ChenHuang2022}
L.~Chen and X.~Huang.
\newblock Complexes from complexes: Finite element complexes in three
  dimensions.
\newblock {\em arXiv preprint arXiv:2211.08656}, 2022.

\bibitem{ChenHuangDivRn2022}
L.~Chen and X.~Huang.
\newblock Finite elements for div- and divdiv-conforming symmetric tensors in
  arbitrary dimension.
\newblock {\em SIAM J. Numer. Anal.}, 60(4):1932--1961, 2022.

\bibitem{Chen;Huang:2020Finite}
L.~Chen and X.~Huang.
\newblock Finite elements for {${\rm div\,div}$} conforming symmetric tensors
  in three dimensions.
\newblock {\em Math. Comp.}, 91(335):1107--1142, 2022.

\bibitem{christiansenFiniteElementSystems2023}
S.~H. Christiansen and K.~Hu.
\newblock Finite element systems for vector bundles: elasticity and curvature.
\newblock {\em Found. Comput. Math.}, 23(2):545--596, 2023.

\bibitem{Ciarlet1978}
P.~G. Ciarlet.
\newblock {\em The finite element method for elliptic problems}.
\newblock North-Holland Publishing Co., Amsterdam, 1978.

\bibitem{CockbuGopala2005Incompressible}
B.~Cockburn and J.~Gopalakrishnan.
\newblock Incompressible finite elements via hybridization. {II}. {T}he
  {S}tokes system in three space dimensions.
\newblock {\em SIAM J. Numer. Anal.}, 43(4):1651--1672, 2005.

\bibitem{Comodi1989}
M.~I. Comodi.
\newblock The {H}ellan-{H}errmann-{J}ohnson method: some new error estimates
  and postprocessing.
\newblock {\em Math. Comp.}, 52(185):17--29, 1989.

\bibitem{EngelGarikipatiHughesLarsonEtAl2002}
G.~Engel, K.~Garikipati, T.~J.~R. Hughes, M.~G. Larson, L.~Mazzei, and R.~L.
  Taylor.
\newblock Continuous/discontinuous finite element approximations of
  fourth-order elliptic problems in structural and continuum mechanics with
  applications to thin beams and plates, and strain gradient elasticity.
\newblock {\em Comput. Methods Appl. Mech. Engrg.}, 191(34):3669--3750, 2002.

\bibitem{fraeijs1965displacement}
B.~Fraeijs~de Veubeke.
\newblock Displacement and equilibrium models in the finite element method.
\newblock In O.~Zienkiewicz and G.~Holister, editors, {\em Stress analysis}.
  John Wiley \& Sons, 1965.

\bibitem{FuehrerHeuer2023}
T.~F\"{u}hrer and N.~Heuer.
\newblock Mixed finite elements for {K}irchhoff-{L}ove plate bending.
\newblock {\em arXiv preprint arXiv:2305.08693}, 2023.

\bibitem{Fuhrer;Heuer;Niemi:2019ultraweak}
T.~F\"{u}hrer, N.~Heuer, and A.~H. Niemi.
\newblock An ultraweak formulation of the {K}irchhoff-{L}ove plate bending
  model and {DPG} approximation.
\newblock {\em Math. Comp.}, 88(318):1587--1619, 2019.

\bibitem{Hellan1967}
K.~Hellan.
\newblock {\em Analysis of elastic plates in flexure by a simplified finite
  element method}, volume~46 of {\em Acta polytechnica Scandinavica. Civil
  engineering and building construction series}.
\newblock Norges tekniske vitenskapsakademi, Trondheim, 1967.

\bibitem{Herrmann1967}
L.~R. Herrmann.
\newblock Finite element bending analysis for plates.
\newblock {\em Journal of the Engineering Mechanics Division}, 93(EM5):49--83,
  1967.

\bibitem{Hu;Liang;Ma:2021Finite}
J.~Hu, Y.~Liang, and R.~Ma.
\newblock Conforming finite element divdiv complexes and the application for
  the linearized {E}instein--{B}ianchi system.
\newblock {\em SIAM J. Numer. Anal.}, 60(3):1307--1330, 2022.

\bibitem{Hu;Liang;Ma;Zhang:2022conforming}
J.~Hu, Y.~Liang, R.~Ma, and M.~Zhang.
\newblock New conforming finite element divdiv complexes in three dimensions.
\newblock {\em arXiv preprint arXiv:2204.07895}, 2022.

\bibitem{Hu;Ma;Zhang:2020family}
J.~Hu, R.~Ma, and M.~Zhang.
\newblock A family of mixed finite elements for the biharmonic equations on
  triangular and tetrahedral grids.
\newblock {\em Sci. China Math.}, 64(12):2793--2816, 2021.

\bibitem{HuangHuang2014}
X.~Huang and J.~Huang.
\newblock A reduced local {$C^0$} discontinuous {G}alerkin method for
  {K}irchhoff plates.
\newblock {\em Numer. Methods Partial Differential Equations},
  30(6):1902--1930, 2014.

\bibitem{Johnson1973}
C.~Johnson.
\newblock On the convergence of a mixed finite-element method for plate bending
  problems.
\newblock {\em Numer. Math.}, 21:43--62, 1973.

\bibitem{Nedelec:1986family}
J.-C. N{\'e}d{\'e}lec.
\newblock A new family of mixed finite elements in {${\bf R}^3$}.
\newblock {\em Numer. Math.}, 50(1):57--81, 1986.

\bibitem{PaulyZulehner2020}
D.~Pauly and W.~Zulehner.
\newblock The div{D}iv-complex and applications to biharmonic equations.
\newblock {\em Appl. Anal.}, 99(9):1579--1630, 2020.

\bibitem{PechsteinSchoeberl2011}
A.~S. Pechstein and J.~Sch{\"o}berl.
\newblock Tangential-displacement and normal-normal-stress continuous mixed
  finite elements for elasticity.
\newblock {\em Math. Models Methods Appl. Sci.}, 21(8):1761--1782, 2011.

\bibitem{Pechstein;Schoberl:2018analysis}
A.~S. Pechstein and J.~Sch{\"o}berl.
\newblock An analysis of the {TDNNS} method using natural norms.
\newblock {\em Numerische Mathematik}, 139(1):93--120, 2018.

\bibitem{Raviart.P;Thomas.J1977}
P.-A. Raviart and J.~M. Thomas.
\newblock A mixed finite element method for 2nd order elliptic problems.
\newblock In {\em Mathematical aspects of finite element methods ({P}roc.
  {C}onf., {C}onsiglio {N}az. delle {R}icerche ({C}.{N}.{R}.), {R}ome, 1975)},
  pages 292--315. Lecture Notes in Math., Vol. 606. Springer, Berlin, 1977.

\bibitem{Stenberg1991}
R.~Stenberg.
\newblock Postprocessing schemes for some mixed finite elements.
\newblock {\em RAIRO Mod\'el. Math. Anal. Num\'er.}, 25(1):151--167, 1991.

\bibitem{WangXu2006}
M.~Wang and J.~Xu.
\newblock The {M}orley element for fourth order elliptic equations in any
  dimensions.
\newblock {\em Numer. Math.}, 103(1):155--169, 2006.

\bibitem{WellsDung2007}
G.~N. Wells and N.~T. Dung.
\newblock A {$C\sp 0$} discontinuous {G}alerkin formulation for {K}irchhoff
  plates.
\newblock {\em Comput. Methods Appl. Mech. Engrg.}, 196(35-36):3370--3380,
  2007.

\bibitem{YeZhang2020}
X.~Ye and S.~Zhang.
\newblock A stabilizer free weak {G}alerkin method for the biharmonic equation
  on polytopal meshes.
\newblock {\em SIAM J. Numer. Anal.}, 58(5):2572--2588, 2020.

\bibitem{ZhuXieWang2023}
P.~Zhu, S.~Xie, and X.~Wang.
\newblock A stabilizer-free {$C^0$} weak {G}alerkin method for the biharmonic
  equations.
\newblock {\em Sci. China Math.}, 66(3):627--646, 2023.

\end{thebibliography}
\end{document}